\documentclass[12pt]{amsart}
\usepackage{amssymb, amstext, amscd, amsmath, amssymb}
\usepackage{mathtools, xypic, paralist, color, dsfont, rotating}
\usepackage{verbatim}
\usepackage{enumitem}
\usepackage{fullpage}
\usepackage{changepage}

\usepackage{showlabels}
\usepackage{color}
\usepackage{hyperref}

\usepackage{quoting}
\quotingsetup{vskip=.1in}
\quotingsetup{leftmargin=.17in}
\quotingsetup{rightmargin=.17in}

\makeatletter
\def\@cite#1#2{{\m@th\upshape\bfseries%
[{#1\if@tempswa{\m@th\upshape\mdseries, #2}\fi}]}}
\makeatother
\theoremstyle{plain}
\newtheorem{theorem}{Theorem}[section]
\newtheorem{corollary}[theorem]{Corollary}
\newtheorem{proposition}[theorem]{Proposition}
\newtheorem{lemma}[theorem]{Lemma}
\theoremstyle{definition}
\newtheorem{definition}[theorem]{Definition}
\newtheorem{example}[theorem]{Example}
\newtheorem{examples}[theorem]{Examples}

\newtheorem{remark}[theorem]{Remark}

\newtheorem{question}[theorem]{Question}
\newtheorem*{acknow}{Acknowledgements}
\theoremstyle{remark}

\mathtoolsset{centercolon}
  \newcommand{\A}{{\mathcal{A}}}
  \newcommand{\B}{{\mathcal{B}}}

  \newcommand{\E}{{\mathcal{E}}}
  \newcommand{\F}{{\mathcal{F}}}
  \newcommand{\G}{{\mathcal{G}}}

  \newcommand{\I}{{\mathcal{I}}}
  \newcommand{\J}{{\mathcal{J}}}
  \newcommand{\K}{{\mathcal{K}}}
\renewcommand{\L}{{\mathcal{L}}}
  \newcommand{\M}{{\mathcal{M}}}
  
\renewcommand{\O}{{\mathcal{O}}}
\renewcommand{\P}{{\mathcal{P}}}
  \newcommand{\Q}{{\mathcal{Q}}}
  
\renewcommand{\S}{{\mathcal{S}}}
  \newcommand{\T}{{\mathcal{T}}}
  \newcommand{\U}{{\mathcal{U}}}
  \newcommand{\V}{{\mathcal{V}}}
  \newcommand{\W}{{\mathcal{W}}}
  \newcommand{\X}{{\mathcal{X}}}

\newcommand{\eps}{\varepsilon}
\def\al{\alpha}
\def\be{\beta}
\def\ga{\gamma}
\def\de{\delta}
\def\ze{\zeta}

\def\la{\lambda}
\def\La{\Lambda}
\def\om{\omega}
\def\Om{\Omega}
\def\si{\sigma}
\def\Si{\Sigma}
\def\Ups{\Upsilon}
\newcommand\vth{\vartheta}
\newcommand\vpi{\varphi}

\newcommand{\bB}{\mathbb{B}}
\newcommand{\bC}{\mathbb{C}}
\newcommand{\bD}{\mathbb{D}}
\newcommand{\bF}{\mathbb{F}}
\newcommand{\bN}{\mathbb{N}}
\newcommand{\bT}{\mathbb{T}}
\newcommand{\bZ}{\mathbb{Z}}

\newcommand{\fA}{{\mathfrak{A}}}

\newcommand{\fF}{{\mathfrak{F}}}

\newcommand{\fM}{{\mathfrak{M}}}

\newcommand{\fT}{{\mathfrak{T}}}

\newcommand{\Bs}{{\mathbf{s}}}
\newcommand{\Bt}{{\mathbf{t}}}

\newcommand{\Bv}{{\mathbf{v}}}
\newcommand{\Bw}{{\mathbf{w}}}

\newcommand{\foral}{\text{ for all }}
\newcommand{\qand}{\quad\text{and}\quad}
\newcommand{\qif}{\quad\text{if}\quad}

\newcommand{\qfor}{\quad\text{for}\ }

\newcommand{\qotherwise}{\quad\text{otherwise}}

\newcommand{\ca}{\mathrm{C}^*}

\newcommand{\cenv}{\mathrm{C}^*_{\textup{env}}}

\newcommand{\ol}{\overline}
\newcommand{\wt}{\widetilde}
\newcommand{\wh}{\widehat}

\newcommand{\lh}{\, \lhd \,}
\newcommand{\ad}{\operatorname{ad}}

\newcommand{\alg}{\operatorname{alg}}

\newcommand{\diag}{\operatorname{diag}}

\newcommand{\End}{\operatorname{End}}
\newcommand{\ev}{\operatorname{ev}}

\newcommand{\id}{{\operatorname{id}}}

\newcommand{\mt}{\emptyset}
\newcommand{\Mult}{\operatorname{Mult}}

\newcommand{\ran}{\operatorname{Ran}}

\newcommand{\spn}{\operatorname{span}}

\newcommand{\supp}{\operatorname{supp}}

\newcommand{\sca}[1]{\left\langle#1\right\rangle} 
\newcommand{\nor}[1]{\left\Vert #1\right\Vert} 
\newcommand{\un}[1]{{\underline{#1}}} 

\addtocontents{toc}{\protect\setcounter{tocdepth}{1}}

\begin{document}

\title[Operator Algebras and Monomial Ideals]{Operator Algebras of Monomial Ideals in Noncommuting Variables}

\author[E.T.A. Kakariadis]{Evgenios T.A. Kakariadis}
\address{School of Mathematics and Statistics\\ Newcastle University\\ Newcastle upon Tyne\\ NE1 7RU\\ UK}
\email{evgenios.kakariadis@ncl.ac.uk}

\author[O.M. Shalit]{Orr Moshe Shalit}
\address{Department of Mathematics\\Technion - Israel Institute of Technology\\Haifa 3200003\\Israel}
\email{oshalit@technion.ac.il}



\maketitle


\begin{abstract}
We study operator algebras arising from monomial ideals in the ring of polynomials in noncommuting variables, through the apparatus of subproduct systems and C*-correspondences.
We provide a full comparison amongst the related operator algebras.
For our analysis we isolate a partially defined dynamical system, to which we refer as the {\em quantised dynamics} of the monomial ideal.

In addition we revisit several previously considered constructions.
These include Matsumoto's subshift C*-algebras, as well as the tensor and the Pimsner algebras associated with dynamical systems or graphs.
We sort out the various relations by giving concrete conditions and counterexamples that orientate the operator algebras of our context.

It appears that the boundary C*-algebras do not arise as the quotient with the compact operators unconditionally.
We establish a dichotomy to this effect by examining the resulting tensor algebras.
We identify their boundary representations, we analyse their C*-envelopes, and we give criteria for hyperrigidity.
Moreover we completely classify them in terms of the data provided by the monomial ideals.
For tensor algebras of C*-correspondences and bounded isomorphisms this is achieved up to the level of local conjugacy (in the sense of Davidson and Roydor) for the quantised dynamics.
For tensor algebras of subproduct systems and algebraic isomorphisms this is achieved up to the level of equality of monomial ideals modulo permutations of the variables.

In the process we accomplish more in different directions.
Most notably we show that tensor algebras form a complete invariant for isomorphic (resp. similar) subproduct systems of homogeneous ideals up to isometric (resp. bounded) isomorphisms.

The results on local conjugacy are obtained via an alternative proof of the breakthrough result of Davidson and Katsoulis on piecewise conjugate systems.
For our purposes we use appropriate compressions of the Fock representation.
We then apply this alternative proof locally for the partially defined quantised dynamics.
In this way we avoid the topological graphs machinery and pave the way for further applications.
These include operator algebras of dynamical systems over commuting contractions or over row commuting contractions.
\end{abstract}


\tableofcontents

\section{Introduction}

A contemporary trend in noncommutative geometry is to encode geometrical and topological objects in terms of operator algebras.
This program is developed around two main objectives:

\begin{quoting}
\noindent {\bf (a)} Use C*-algebras and nonselfadjoint operator algebras as an invariant (resp. a complete invariant) for classifying (resp. encoding) the objects.

\noindent {\bf (b)} Explore the passage from intrinsic properties of the object into properties of the associated operator algebras.
\end{quoting}

\noindent In turn the invariants of the related operator algebras may be used to classify the objects.
In this course of study one has to answer two interrelated problems that specify the context:

\begin{quoting}
\noindent {\bf Problem 1.} Which are the (desirable) features of the object that determine the associated operator algebra?

\noindent {\bf Problem 2.} What is the (desirable) level of equivalence for classifying objects in the ``same'' family?
\end{quoting}

\noindent The enumeration is in fact redundant; \emph{an} answer to any of the above problems initiates a new program in finding \emph{an} answer to the other.
There is an extensive literature on the subject and we cannot list all papers.
We will only elaborate on the direct connections with results concerning:  graphs \cite{KatKri04, Sol04}, dynamical systems \cite{DavKak14, DavKat08, DavKat11, KakKat12}, topological graphs \cite{DavRoy11}, homogeneous ideals \cite{DRS11}, complex analytic varieties \cite{DRS15}, stochastic matrices \cite{DorMar14, DorMar15}, and C*-correspondences \cite{AEE98, KakKat14, MPT08, MuhSol00}.

The structure under consideration in the current paper is monomial ideals in the ring of polynomials in noncommuting variables.
This links to an ongoing effort to develop operator-algebraic geometry in the unit ball \cite{DRS11,Pop06}.
The goal is to study a number of C*- and nonselfadjoint algebras related to one monomial ideal, and study those in conjunction with Questions {\bf (a)-(b)} and Problems {\bf (1)-(2)} described above.

\subsection*{Motivation}

Our approach relies on the interplay between C*-algebras and nonselfadjoint operator algebras related to three classes: C*-correspondences, subshifts, and subproduct systems.
All three form active research areas with a rich literature in their own merits.
We provide a short presentation on key elements that we use in the current paper.

The theory of operator algebras of C*-correspo\-ndences has been under considerable development since their initiation by Pimsner \cite{Pim97}.
They have become a central part in the theory of C*-algebras covering a broad number of constructions that generalise the Toeplitz algebra and the Cuntz algebra.
Notable examples of Pimsner algebras are C*-algebras associated with topological graphs (hence Cuntz-Krieger algebras, graph algebras, dynamical systems etc.), Hilbert bimodules, and bimodules of unital completely positive maps (see \cite{Kat03} and \cite[Examples 4.6.10--4.6.12]{BroOza08}).
In addition the class of tensor algebras arising from C*-correspondences includes Peters' semicrossed product \cite{Pet84} and Popescu's noncommutative disc algebras \cite{Pop96}.

A \emph{subshift} is the dynamical system obtained by restricting the left shift $\sigma$ on $\{1, \ldots, d\}^\bZ$ to a closed invariant subspace $\Lambda \subseteq \{1,\ldots, d\}^\bZ$.
The study of subshifts is a significant branch of topological dynamics, e.g. \cite{LinMar95}.
Motivated by the Cuntz-Krieger algebras \cite{CunKri80}, Matsumoto \cite{Mat97} introduced a way to construct a C*-algebra out of a subshift and initiated an in-depth study of these so-called \emph{subshift C*-algebras}.
In \cite{Mat97} Matsumoto constructs $\ca(T)$ and then concentrates on its quotient by the compacts.
In a series of papers (e.g. \cite{Mat98, Mat99}) he shows how the C*-algebraic structure of this quotient is determined by the topological and combinatorial properties of $\La$.
Later Carlsen and Matsumoto \cite{CarMat04} proposed another construction that differs in general from the one in \cite{Mat97}.
The reason for doing so was the discovery of a flaw in one of the earlier papers.
They found that many of the theorems in the earlier papers which relied on the flawed one are true if one replaces the original algebra by their new construction, see \cite[Introduction]{CarMat04}.
In turn the construction of \cite{CarMat04} did not satisfy a desirable universal property, and around the same time Carlsen \cite{Car08} proposed a third subshift C*-algebra by using an associated C*-correspondence.
In addition the object in \cite{Car08} is shown to be an invariant for conjugacy of subshifts.
Matsumoto \cite{Mat02} introduced later C*-algebras associated with symbolic matrix systems and $\la$-graph systems that generalise the previous constructions.

Subshift algebras form an example of operator algebras related to subproduct systems \cite[Section 12]{ShaSol09}.
A \emph{subproduct system} is a collection of Hilbert (C*- or W*-) correspondences $X = \{X(s)\}_{s\in\S}$ indexed by a semigroup $\S$, which carries an associative family of multiplication operators $U_{s,t} \colon X(s) \otimes X(t) \rightarrow X(s+t)$.
They were formally introduced in \cite{ShaSol09} and \cite{BhaMuk10} as a generalisation of \emph{product systems} \cite{Arv03}, and now they form a basic technical tool in the analysis of \emph{noncommutative dynamics}, e.g. semigroups of CP-maps and $*$-endomorphisms.
In fact they were present in the literature in various guises prior to their formal introduction, see \cite{Arv03, BhaSke00, Mar03, MuhSol02}.
Simple cases like $\S = \bN$ are still exploitable with numerous successful applications \cite{BhaMuk10, Muk11, ShaSke11, ShaSke15, ShaSol09, Ver16}.

The operator algebras of a subproduct system are defined concretely by an appropriate compression on the Fock representation of $X(1)$ and quotients by the compacts.
The motivation was to approach the unstudied subproduct systems from the well trodden path of C*-correspondences.
Soon it was realised that there are substantial differences, yet there is enough structure to be tractable to analysis.
Several subsequent works have thus considered operator algebras of subproduct systems \emph{per se} \cite{DRS11, DorMar14, Gur12, KenSha14, Vis11, Vis12}.
There are also works without regard to quantum dynamics or operator algebras, e.g. \cite{GerSke14, Tsi09a, Tsi09b}.
Examples include the continuous functions on the unit sphere $C(\partial \bB_d)$, operator algebras of polynomial relations, algebras of analytic functions on homogeneous varieties of the unit ball, as well as algebras of noncommutative analytic functions on homogeneous subvarieties of the noncommutative unit ball.

A main obstacle so far in the category of subproduct systems is the lack of a universal property for the concretely defined Toeplitz and Cuntz algebras, that could act in analogy to the theory of C*-correspondences.
For example Katsoulis and Kribs \cite{KatKri06} use such results of Katsura \cite{Kat04} to obtain that the C*-envelope (in the sense of Arveson \cite{Arv69}) of the tensor algebra of a C*-correspondence is its Cuntz-Pimsner algebra.
Such a result is known only in some cases for product systems, e.g. \cite{DFK14}, and does not hold for subproduct systems.
In fact, until recently all examples indicated that the C*-envelope of the tensor algebra of a subproduct system could be either its Toeplitz algebra or its Cuntz algebra.
In the course of writing this paper we were informed by Dor-On and Markiewicz that they discovered a subclass for which there may be several possibilities between the Toeplitz and the Cuntz algebras \cite{DorMar15}.
The on-going program asks for properties that control such behaviour and includes a number of further questions about: (a) nuclearity, exactness, ideal structure, KMS structure for the C*-algebras, and (b) boundary representations, hyperrigidity, (hyper)reflexivity for the tensor algebras.
A unified treatment for all the aforementioned objects cannot be proposed at this moment.
Our study on monomial ideals should be seen as one more step in this larger program.

\subsection*{Main results}

We fix an orthonormal basis $\{e_1, \ldots, e_d\}$ for $\bC^d$ and we write $e_{\mu} = e_{\mu_1} \otimes \cdots \otimes e_{\mu_n}$ for every word $\mu = \mu_1 \dots \mu_n  \in \bF_+^d$.
Given a monomial ideal $\I$ in $\bC\sca{x_1,\dots,x_d}$ let $X = (X(n))$ such that
\[
X(n) = (\bC^d)^{\otimes n} \ominus \{e_\mu \mid \ol{x}^\mu \in \I\},
\]
and let the multiplication $U_{n,m}$ be concatenation of words followed by the projection on $X(n+m)$.
We fix the set $\La^* = \{\mu \in \bF_+^d \mid \ol{x}^\mu \notin \I \}$ of {\em allowable words}.
Let us write $\F_X = \oplus_{n \geq 0} X(n)$ and let the shift operators $\{T_i\}_{i=1}^d$ defined by
\[
T_i e_\nu = \begin{cases} e_{i\nu} & \text{ if } i\nu \in \La^*,\\ 0 & \text{ otherwise}. \end{cases}
\]
The C*-algebras
\[
\ca(T) :=\ca(I, T_i \mid i=1, \dots, d) \qand \ca(T)/\K(\F_X)
\]
are the \emph{Toeplitz} and the \emph{Cuntz} algebra of $X$.
The nonselfadjoint subalgebra
\[
\A_X : = \ol{\alg}\{I, T_i \mid i=1, \dots, d\}
\]
of $\ca(T)$ is the \emph{tensor algebra} of $X$.
It may happen that $\ca(T)/\K(\F_X)$ remembers remarkably less than the original data.
For example, if $\I = \sca{xx, xy} \lhd \bC\sca{x,y}$ then $\ca(T) = \ca(T_x, T_y)$ is highly not commutative whereas $\ca(T)/\K(\F_X) \simeq C(\bT)$ (Example \ref{E: ss not sft}).

The family $\{T_i\}_{i=1}^d$ satisfies a number of properties; for example it is orthogonal and $T_\mu^* T_\mu T_i = T_i T_{\mu i}^* T_{\mu i}$.
Hence the linear space
\[
E = \ol{\spn} \{ T_i a \mid a \in A, i=1, \dots, d\},
\]
becomes a C*-correspondence over the commutative unital C*-algebra
\[
A = \ca(T_\mu^* T_\mu \mid \mu \in \La^*).
\]
Consequently we obtain the Toeplitz-Pimsner algebra $\T_E$, the Cuntz-Pimsner algebra $\O_E$, and the tensor algebra $\T_E^+$ in the sense of Muhly and Solel \cite{MuhSol98}.
We give a list of equivalent conditions for the left action to be injective (Proposition \ref{P: kernel}).
For the discussion let us mention that the left action is injective if and only if there exists an $i_0 \in\{1, \dots, d\}$ such that $\La^* i_0 \subseteq \La^*$.
Notice that the left action is not injective for the example of $\I = \sca{xx,xy} \lhd \bC\sca{x,y}$.

In the first part of the paper we settle the relation between the appearing operator algebras  (Theorem \ref{T: dichotomy}).
We show that $\ca(T)$ is the $J$-relative Cuntz-Pimsner algebra of $E$ for the ideal $J$ of $A$ generated by
\[
\{I - T_\mu^*T_\mu \mid \mu \in \La^*\},
\]
whereas $\ca(T)/\K(\F_X)$ is the $A$-relative Cuntz-Pimsner algebra of $E$.
In particular there are canonical $*$-epimorphisms
\[
\T_E \stackrel{\phi_1}{\longrightarrow} \ca(T) \stackrel{\phi_2}{\longrightarrow} \O_E \stackrel{\phi_3}{\longrightarrow} \ca(T)/\K(\F_X)
\]
where:
\begin{enumerate}
\item[{\bf (a)}] $\phi_1$ is injective if and only if $\I = (0)$;
\item[{\bf (b)}] exactly one of the $\phi_2$ and $\phi_3$ is injective; and
\item[{\bf (c)}] $\phi_3$ is injective if and only if the left action is injective.
\end{enumerate}
Thus we get that $\A_X \subseteq \T_E^+$ (Corollary \ref{C: tensor}).
Item (c) above explains why we may lose data (even as much as one generator) when passing to the quotient.

There are two immediate remarks following these first results.
First we get an example of a class of subproduct systems whose operator algebras coincide in a non-trivial way with those of a class of C*-correspondences.
This is quite surprising and contradicts to what was thought in the past\footnote{After this paper was complete, we were made aware of the paper \cite{AS08}, which considers {\em interacting Fock spaces}. 
Interacting Fock spaces provide a framework that encompasses subproduct systems. 
A major theme in \cite{AS08} was the problem of representing the creation operators on an interacting Fock space by operators on the full Fock module of a C*-correspondence, similar to the way we represent $T$ in $E$ and in $\T_E$. }, e.g. \cite{Vis12}.
Secondly we get that an appropriate Cuntz-type C*-algebra for a subshift is either $\ca(T)$ or $\ca(T)/\K(\F_X)$ accordingly to the form of the subshift.
This is again surprising and should be compared to Matsumoto's work where $\ca(T)/\K(\F_X)$ is considered \emph{a priori}, e.g. \cite{Mat97}.

Our analysis works in parallel with what is known for graphs with sources which were not tractable before Katsura's work \cite{Kat04}.
In analogy the new treatment of subshift C*-algebras and the identification of $\O_E$ that we offer here can work effectively to treat general subshifts, that may not satisfy the property (I) of Matsumoto and Carlsen \cite{Car08, CarMat04, Mat97, Mat98, Mat99}.
In Section \ref{S: compare} we include a review of a number of constructions, including the previous approaches on subshifts, with which we compare our findings. 
On the one hand, we find that our construction differs from ones that have been considered.
On the other hand, in Section \ref{Ss: OE as a graph algebra}, we show that when $E$ arises from a sofic subshift, $\O_E$ arises via two familiar constructions: it the {\em graph C*-algebra} of the {\em follower set graph of the subshift}.

In this setting,  any structure of $\I$ is tracked via the related $E$.
We define an intrinsic dynamical system that gives the encoding by
\[
\al_i \colon A \to A : a \mapsto T_i^*a T_i, \foral i=1, \dots ,d.
\]
We coin $(A,\al) \equiv (A,\al_1, \dots, \al_d)$ as \emph{the quantised dynamics of $\I$} (Section \ref{Ss: qd}).
Even though the quantised dynamics form a complete \emph{conjugacy} invariant for the ideal $\I$ (Theorem \ref{T: qd complete invariant}), some elements are not recognised by $E$.
Following Davidson and Katsoulis \cite{DavKat11} and Davidson and Roydor \cite{DavRoy11} we show that unitary equivalence of the C*-correspondences is equivalent to {\em local conjugacy} (which is a weaker condition than conjugacy) of the quantised dynamics.
By universality, local conjugate systems have $*$-isomorphic Toeplitz-Pimsner algebras and completely isometric tensor algebras.
By the co-universal property of the C*-envelope and by using \cite{KatKri06} this passes to the Cuntz-Pimsner algebras.
In addition we show that this holds for the C*-algebras $\ca(T)$ and $\ca(T)/\K(\F_X)$ as well (Corollary \ref{C: rel CP}).

Next we turn our focus to the structure of $\T_E^+$ and $\A_X$.
By using a universal property for $\O_E$ (Theorem \ref{T: CP}) and \cite{KatKri06} we prove that $\T_E^+$ is hyperrigid (Theorem \ref{T: hyper T_E^+}).
If $\I$ is finitely generated then the same holds for $q(\A_X)$ where $q \colon \ca(T) \to \ca(T)/\K(\F_X)$ is the quotient map (Proposition \ref{P: hyper A_X}).
Thus $\ca(T)/\K(\F_X)$ is the C*-envelope of $\A_X$ if $q|_{\A_X}$ is completely isometric, and we provide a criterion for this to happen (Theorem \ref{T: cenv A_X}).
This criterion becomes a necessary condition for a big class of monomial ideals, which contains subshifts of finite type.
Hyperrigidity is further used to obtain a universal property for $\ca(T)/\K(\F_X)$ when $\I$ is of finite type (Theorem \ref{T: universal}).
This is quite pleasing as it establishes a strong interaction between C*-algebras and nonselfadjoint operator algebras.

We already commented on that $\T_E^+$ is an isometric isomorphic invariant for local conjugate quantised dynamics.
Remarkably the converse also holds, and moreover it holds for continuous isomorphisms (Corollary \ref{C: classification}).
This can be derived by the results of Davidson and Roydor \cite{DavRoy11} (see Remark \ref{R: DavRoy11} as well).
Nevertheless, we do so by applying the tools we exhibit in the appendix instead of the topological graph language.
In particular we provide an alternative proof of \cite[Theorem 3.22]{DavKat11}.
Our remark here is that one can work directly on an appropriate compression of the Fock representation.
We then obtain the proof of Corollary \ref{C: classification} by applying locally these ideas.

On the other hand $\A_X$ appears to provide a sharper encoding of the monomial ideals.
This is not a coincidence, as $\A_X$ is in a sense the universal operator algebra generated by a row contraction satisfying the relations in $\I$.
Therefore its space of (completely contractive) representations is parameterized by the noncommutative variety
\[
V(\I) = \{\un{S} = [S_1, \dots, S_d] \in (\B(H)^d)_1 \mid p(\un{S}) = 0 \textrm{ for all } p\in \I \}.
\]
However, our results are stronger than expected.
We show that $\A_X$ is a complete (algebraic isomorphic) invariant for monomial ideals that coincide up to a permutation of the symbols (Theorem \ref{T: class sps}).
In the process we show that this is true for $\bC\sca{x_1,\dots,x_d}/\I$ up to graded algebraic isomorphisms.
In contrast to what one would expect this behaviour is not met in other settings and it is quite unique to monomial ideals (Remark \ref{R: rig}).

Key to the proof of Theorem \ref{T: class sps} is Lemma \ref{L: vac pre} that applies in a more general setting.
We prove that algebraic isomorphisms between tensor algebras of subproduct systems related to \emph{homogeneous} ideals can be substituted by graded ones.
Then by using results of the second author with Davidson and Ramsey \cite{DRS11}, and of Dor-On and Markiewicz \cite{DorMar14} we settle the following classification problem (Theorem \ref{T: sps cla}): tensor algebras form a complete invariant up to isometric isomorphisms (resp. bounded isomorphisms) for isomorphic (resp. similar) subproduct systems of homogeneous ideals.

\subsection*{Open Questions}

Our alternative proof of \cite[Theorem 3.22]{DavKat11} is flexible enough to treat other classes of operator algebras related to classical systems.
The key requirement for applying our arguments is for a specific representation on a finite Hilbert space to be completely contractive for the class of the operator algebras under study.
We briefly describe some more examples of this phenomenon in Section \ref{S: app pc}.
We reserve the full discussion on general classes for a forthcoming project.

Some of our intermediate results are not obtained in the full generality of monomial ideals.
In Questions \ref{Q: hyperrigidity}, \ref{Q: aut cont}, and \ref{Q: aut cont 2} we gather some of the possible considerations.
For the conclusion let us add some more that appear to be programs on their own.

There are several notions of equivalence for subshifts such as conjugacy and flow equivalence.
It will be interesting to research similarities and differences with local conjugacy of the quantised dynamics.
Since local conjugacy (of the quantised dynamics) requires the same number of symbols we expect it to be stronger than conjugacy (of the subshifts).
However it is reasonable to ask whether local conjugate systems imply conjugacy on the corresponding edge shifts.
On the other hand flow equivalence is translated into strong shift equivalence of the associated $0$-$1$ matrices \cite{ParSul75, Wil74} (see also \cite{LinMar95}).
There is also a fourth equivalent relation, that of shift equivalence \cite{Wil74}.
It is reasonable to ask for the connections with local conjugacy.

The related C*-correspondences that we provide herein suggest a second direction.
The first author and Katsoulis \cite{KakKat14} have formulated shift equivalence relations for C*-correspon\-dences, following the work of Abadie, Eilers and Exel \cite{AEE98}, of Muhly and Solel \cite{MuhSol00}, and of Muhly, Pask and Tomforde \cite{MPT08}.
A combination of these works shows that these relations provide strong Morita equivalence of the Cuntz-Pimsner algebras.
A further project is to study this question for the relative Cuntz-Pimsner algebras $\ca(T)$.

Moreover another direction is to analyse the case of subshift matrix systems and $\la$-graph systems under the new prism we offer by Theorem \ref{T: dichotomy}.
A similar phenomenon is expected to appear, as \cite[Theorem A, equation 1.3]{Mat02} suggests.
Again we anticipate that the appropriate C*-algebra depends on the injectivity of a particular C*-correspondence.
We remark that the construction in \cite[Section 6]{Mat02} does not suffice for this purpose: in the case of subshifts it coincides with $q(E)$ instead of $E$ (see also Remark \ref{R: quotient corre}).

Theorem \ref{T: class sps} implies that the tensor algebras of the subproduct systems completely encode the subshifts, hence offering a complete invariant for languages.
All results include (in fact they are stronger in the case of) two-sided subshifts, and in particular for deterministic automata.
It now looks reasonable to move to other directions, for example towards non-deterministic automata or multivariable subshifts.
The language of subproduct systems is flexible enough to encode such structures and accommodate such results.

\begin{acknow}
The authors acknowledge support from London Ma\-thematical Society (Scheme 4, Grant Ref: 41411). The second author was partially supported by Israel Science Foundation Grant no. 474/12, by EU FP7/2007-2013 Grant no. 321749, and by GIF Grant no. 2297-2282.6/20.1.

The authors would like to thank Guy Salomon for the helpful remarks and comments.

The first author would like to dedicate this paper to Doukissa Markopoulou.
Thank you for your courage; life became darker but the light stays on.
\end{acknow}

\section{Preliminaries}\label{S: defn}

\subsection{Dilation theory}

The reader should be well acquainted with the representation theory of nonselfadjoint operator algebras \cite{BleLeM04, Pau02}.
For this paper a nonselfadjoint operator algebra will be a norm-closed non-involutive subalgebra of $\B(H)$ for a Hilbert space $H$.
The morphisms in the category of operator algebras consist of completely contractive homomorphisms, which we will refer to as \emph{representations}.

Arveson \cite{Arv69} noticed that a nonselfadjoint operator algebra $\A$ may admit a number of completely isometric homomorphisms $\iota_k \colon \A \to \B(H_k)$ such that $\ca(\iota_k(\A))$ is not $*$-isomorphic to $\ca(\iota_{k'}(\A))$ for $k \neq k'$.
Every $\ca(\iota_k(\A))$ is called \emph{a C*-cover of $\A$}.
This striking event can be realised in the sense that different algebraic relations (even $*$-algebraic relations) on the generators may define the same object.
Thus (and in contrast to C*-algebras) $*$-algebraic relations may not define uniquely universal objects.
It is dilation theory that facilitates the comparison and the identification of such objects.
The interested reader is addressed to \cite{DFK14} for a discussion on nonselfadjoint operator algebras relative to families of homomorphisms.

A second question that Arveson \cite{Arv69} posed was whether there is a co-universal C*-cover, with a role similar to the one that the injective envelope plays in ring theory.
Hamana \cite{Ham79} showed that this is indeed true: given a nonselfadjoint operator algebra $\A$ there is a C*-cover, denoted by $\cenv(\A)$, such that for any other C*-cover $\ca(\iota(\A))$ there exists a necessarily unique $*$-epimorphism $\Phi \colon \ca(\iota(\A)) \to \cenv(\A)$ such that $\Phi \iota (a) = a$ for all $a\in \A$.
The kernel of $\Phi$ is called \emph{the \v{S}ilov ideal of $\A$} in analogy to the \v{S}ilov boundary of commutative Banach algebras.
This result holds for unital operator spaces as well.
The first author \cite{Kak11-2} has provided an elementary proof of Hamana's Theorem.

Arveson's motivation was to find an interplay between the C*-envelope and dilation theory.
A representation $\rho \colon \A \to \B(K)$ is a \emph{dilation} of a representation $\phi \colon \A \to \B(H)$ if $H \subseteq K$ and $P_H \rho(a)|_H = \phi(a)$ for all $a\in \A$.
Dritschel and McCullough \cite{DriMcC05} showed that every completely isometric representation $\phi$ of $\A$ admits a maximal dilation $\rho$ (in the sense that all dilations of $\rho$ are trivial).
Remarkably they then obtain that $\cenv(\A) \simeq \ca(\rho(\A))$.
The interested reader is addressed to \cite{KakNotes} for an overview on the subject.

It follows that the maximal representation $\rho \colon \A \to \B(K)$ (that is, $\rho$ admits only trivial dilations) extends to a unique $*$-representation $\wt{\rho} \colon \cenv(\A) \to \B(K)$ \cite{Arv08}.
It remains an open problem whether the converse of this scheme holds.
To capture this property, Arveson \cite{Arv11} introduced the notion of hyperrigidity.
An operator algebra $\A \subseteq \ca(\A)$ is \emph{hyperrigid} if for every faithful $*$-representation $\pi \colon \ca(\A) \to \B(K)$ the restriction $\pi|_\A$ is maximal.
Then the same is true for non-injective $*$-representations $\pi$ of $\ca(\A)$, and $\ca(\A)$ is the C*-envelope of $\A$.

We mention that there is also the notion of \emph{the Choquet boundary}.
In fact Arveson's initial vision was to derive the existence of the C*-envelope through it.
He was able to achieve this in the separable case \cite{Arv08}, and eventually Davidson and Kennedy \cite{DavKen13} solved this long standing open problem.

\subsection{C*-correspondences}\label{S: corre}

The reader should be well acquainted with the general theory of Hilbert C*-modules \cite{Lan95, ManTro01}.
Throughout the last 15 years there have been many influential authors working on C*-correspondences (e.g.
\cite{AEE98, DPZ98, DykShl01, FMR03, FowRae99, KatKri06, KPW98, Kwa11, MPT08, MuhSol98, MuhSol00, MuhTom04, Pas73} to mention but a few) starting with the seminal work of Pimsner \cite{Pim97}.
The notation and definitions have been under several changes and in what follows we will try to highlight the main features by fixing notation for this paper.
Mainly we follow the breakthrough work of Katsura \cite{Kat04}.
Our suggestions to the reader include \cite{BroOza08} for an introduction to the language of C*-correspondences, and \cite{Kak13-2} for a brief history on the gauge invariant uniqueness theorem.

A \emph{C*-correspondence ${}_A E_A$} is a right Hilbert C*-module $E$ over a C*-algebra $A$ that admits a left action by a $*$-homomorphism $\phi_E \colon A \to \L(E)$.
We will often write $E$ instead of ${}_A E_A$ for simplicity.
A C*-correspondence is called \emph{regular} if it is \emph{injective} (i.e. $\phi_E$ is injective) and $\phi_E(A)$ sits inside the compacts $\K(E)$.
It is called \emph{non-degenerate} if $\ol{\spn}\{\phi_E(a)\xi \mid a \in A, \xi \in E\} = E$.
It is called \emph{full} if $\ol{\spn}\{\sca{\xi,\eta} \mid \xi, \eta \in E\} = A$.

A pair $(\pi,t)$ is said to be a \emph{representation} of $E$ if $\pi \colon A \to \B(H)$ is a $*$-representation and $t \colon E \to \B(H)$ is a linear mapping such that
\[
t(\xi)^* t(\eta) = \pi(\sca{\xi,\eta}) \qand \pi(a) t(\xi) = t(\phi_E(a)\xi)
\]
for all $\xi,\eta \in E$ and $a\in A$; then we get $t(\xi)\pi(a) = t(\xi a)$ for free.
If $\pi$ is injective then $t$ is isometric.
Every such pair defines a $*$-representation $\psi_t \colon \K(E) \to \B(H)$ such that $\psi_t(\Theta^E_{\xi,\eta}) = t(\xi) t(\eta)^*$ for all $\xi, \eta \in E$ \cite{KPW98}.
A pair $(\pi,t)$ is said to admit a gauge action if there is a point-norm continuous family $\{\be_z\}_{z \in \bT}$ of $*$-automorphisms on the C*-algebra
\[
\ca(\pi,t) := \ca(\pi(a), t(\xi) \mid a\in A, \xi \in E)
\]
such that
\[
\be_z(\pi(a)) = \pi(a) \qand \be_z(t(\xi)) = z t(\xi)
\]
for all $a\in A$, $\xi \in E$, and $z \in \bT$.
Let $J \subseteq \phi_E^{-1}(\K(E))$ be an ideal in $A$.
A representation $(\pi,t)$ of $E$ is called \emph{$J$-covariant} if
\[
\pi(a) = \psi_t(\phi_E(a)) \foral a \in J.
\]
If $J$ is \emph{Katsura's ideal}
\[
J_E := \ker\phi_E^\perp \cap \phi_E^{-1}(\K(E))
\]
then $(\pi,t)$ is called simply \emph{covariant} \cite{Kat04}.
The covariance of a pair $(\pi,t)$ is quantified by the ideal
\[
I_{(\pi,t)}':=\{a \in A \mid \pi(a) \in \psi_t(\K(E))\}.
\]
This follows by the insightful \cite[Proposition 3.3]{Kat04} where it is shown that $I_{(\pi,t)}' \subseteq J_E$ when $\pi$ is faithful.
There are special cases where the ideal $I_{(\pi,t)}'$ can be described fairly easily.

\begin{lemma}\label{L: cov}
Suppose that $\K(E)$ admits a unit $e$ and let $(\pi,t)$ be a representation of ${}_A E_A$.
Then $\pi(a) \in \psi_t(\K(E))$ if and only if $\pi(a)(I - \psi_t(e))=0$ if and only if $(I - \psi_t(e)) \pi(a) = 0$.
\end{lemma}

\begin{proof}
It follows by the observation that $\pi(a) \psi_t(e) = \psi_t(\phi_E(a) e)$ and that $\psi_t(e) \pi(a) = \psi_t(e \phi_E(a))$.
\end{proof}

Specific representations are given on the full Fock space on $E$ denoted by $\F_E$.
Let $E^{\otimes n +1} = E^{\otimes n} \otimes_A E$ (the interior tensor product) for $n \geq 1$, and set $E^{\otimes 0} = A$.
On the direct sum Hilbert module $\F_E:= \sum_{n \geq 0} E^{\otimes n}$ let the operators (the sums are taken in the strong operator topology)
\[
t_\infty(\xi) = \sum_{n \geq 0} \tau_n^{n+1}(\xi) \qand \pi_\infty(a) = \sum_{n \geq 0} \phi_n(a)
\]
for all $\xi \in E$ and $a \in A$, where $\tau_n^{n+1}(\xi) \colon E^{\otimes n} \to E^{\otimes n+1}$ is such that $\tau_n^{n+1}(\xi) \eta = \xi \otimes \eta$, $\phi_0 =\id$, and $\phi_n = \phi_E \otimes \id$ for $n \geq 1$.
It is not immediate but if we let $q \colon \L(\F_E) \to \L(\F_E)/ \K(\F_E J)$ be the quotient map for $J \subseteq \phi_E^{-1}(\K(E))$ then $(q \pi_\infty, q t_\infty)$ is a $J$-covariant representation \cite{Kak13-2}.

The universal C*-algebra $\T_E$ generated by $A$ and $E$ with respect to the pairs $(\pi,t)$ is called \emph{the Toeplitz-Pimsner algebra of $E$}.
The norm closed (non-involutive) subalgebra of $\T_E$ generated by $A$ and $E$ is called \emph{the tensor algebra of $E$}, and will be denoted by $\T_E^+$.
The universal C*-algebra $\O(J,E)$ generated by $A$ and $E$ with respect to the $J$-covariant representations $(\pi,t)$ is called \emph{the $J$-relative Cuntz-Pimsner algebra of $E$} \cite{MuhSol98}.
In particular $\O_E:= \O(J_E,E)$ is called \emph{the Cuntz-Pimsner algebra} \cite{Kat04}.
Katsoulis and Kribs \cite{KatKri06} show that $\cenv(\T_E^+) \simeq \O_E$ for any C*-correspondence $E$.

By universality $\O(J,E)$ is the quotient of $\T_E$ by the ideal generated by the differences
\[
\pi(a) - \psi_{t}(\phi_E(a)) \foral a \in J,
\]
where $(\pi,t)$ defines a faithful representation of $\T_E$.
In general $A$ may not embed faithfully in $\O(J,E)$.
In fact we have that $A \subseteq \O(J,E)$ if and only if $J \subseteq J_E$.
In particular when $J \subseteq J_E$ we obtain a commutative diagram
\[
\xymatrix{
\T_E \ar[rr] \ar[dr] & & \O_E \\
& \O(J,E) \ar@{-->}[ur] &
}
\]
where all arrows are $*$-epimorphisms that map elements of some index to elements of the same index.
These are immediate consequences of \emph{the gauge invariant uniqueness theorem} as presented in \cite{Kak13-2} (see also \cite{KakPet13}): if $J \subseteq J_E$ then $(\pi,t)$ defines a faithful representation of $\O(J,E)$ if and only if $(\pi,t)$ admits a gauge action, $\pi$ is injective and $I_{(\pi,t)}'=J$.
In particular if $(\pi,t)$ admits a gauge action and $\pi$ is faithful then:
\begin{inparaenum}
\item $\ca(\pi,t)$ is $*$-isomorphic to $\O(J,E)$ for $J = I_{(\pi,t)}'$;
\item $J \subseteq J_E$ by \cite[Proposition 3.3]{Kat04}; and
\item $\O(J,E)$ is a C*-cover of $\T_E^+$.
\end{inparaenum}

Let $E_A$ and $F_B$ be right Hilbert C*-modules and let $\ga \colon A \to B$ be a $*$-homomorphism.
A right $A$-module map $U \equiv (\ga, U) \colon E_A \to F_B$ is called an \emph{inner-product map} if there exists a linear mapping $U^* \colon F \to E$ such that
\[
\sca{U \xi, \eta}_F = \ga (\sca{\xi, U^* \eta}_E) \foral \xi, \eta \in E.
\]
When $\ga \colon A \to B$ is a $*$-isomorphism then an inner-product map is in $\L(E,F)$ with $U ^* \equiv (\ga^{-1}, U^*)$.
An adjointable map is called \emph{unitary} if $U^*U = I_E$ and $UU^* = I_F$, in which case $E_A$ is said to be \emph{unitarily equivalent} to $F_B$.
An adjointable map is called \emph{invertible} if there exists an adjointable $(\ga^{-1}, V) \colon F \to E$ such that $VU = I_E$ and $UV = I_F$; in this case we write $V = U^{-1}$.
If $U$ is an invertible adjointable map then $U^*U \in \L(E)$ is invertible as well and the map $U |U|^{-1}$ defines a unitary adjointable map in $\L(E,F)$.

An adjointable mapping $(\ga,U) \colon {}_A E_A \to {}_B F_B$ is a \emph{C*-correspondences mapping} if it is a left module map as well.
If $(\ga, U)$ is a unitary mapping then ${}_A E_A$ and ${}_B F_B$ are \emph{unitarily equivalent}.
Inverses of left module maps are again left module maps.
Therefore similar C*-correspondences are automatically unitarily equivalent.

If ${}_A E_A$ and ${}_B F_B$ are unitarily equivalent by a $(\ga,U)$ then $(\pi,t)$ is a representation for $F$ if and only if $(\pi\ga, t U)$ is a representation for $E$.
Due to the universal property we then obtain that the Toeplitz-Pimsner algebras are $*$-isomorphic.
In particular the $*$-isomorphism is completely isometric and restricts to an isomorphism of the tensor algebras.
By \cite{KatKri06} and the C*-envelope machinery (or by algebraic manipulations and Proposition \ref{P: isom rel} below) we then get that the Cuntz-Pimsner algebras are also $*$-isomorphic.
For $J$-relative Cuntz-Pimsner algebras we have the following.

\begin{proposition}\label{P: isom rel}
Let $(\ga,U) \colon {}_A E_A \to {}_B F_B$ be a C*-correspondences unitary mapping.
Let $J_1 \subseteq \phi_E^{-1}(\K(E))$ and $J_2 \subseteq \phi_F^{-1}(\K(F))$.
If $\ga(J_1) = J_2$ then the relative Cuntz-Pimsner algebras $\O(J_1,E)$ and $\O(J_2,F)$ are $*$-isomorphic.
\end{proposition}

\begin{proof}
First notice that $(\ga,U)$ induces an isomorphism $\psi_U \colon \K(E) \to \K(F)$.
The proof is identical to \cite[Lemma 2.2]{KPW98} (in fact one may use concrete faithful representations of $\K(E)$ and $\K(F)$ to achieve this).
Let $(\pi,t)$ be a $J_2$-covariant pair for $F$.
It suffices to show that $\psi_t(\phi_F\ga(a)) = \psi_{t U}( \phi_E(a))$ for all $a \in J_1$.
Indeed, in this case $\ga(a) \in J_2$ and the computation
\[
\pi \ga(a) = \psi_t(\phi_F\ga(a)) = \psi_{t U}( \phi_E(a))
\]
gives that $(\pi\ga, t U)$ is a $J_1$-covariant pair for $E$.
Universality of the relative Cuntz-Pimsner algebras will then complete the proof.
For convenience let $\ga(a) = b$ and observe that
\[
\psi_t(\phi_F(b)) t(U(\xi)) = t(\phi_F(b) U(\xi)) = tU(\phi_E(a) \xi),
\]
for all $\xi \in E$.
Then by applying $t(U(\eta))^*$ for all $\eta \in E$ and by taking limits of finite sums we get that
\[
\psi_t(\phi_F(b) \psi_U(k)) = \psi_t(\phi_F(b)) \psi_t(\psi_U(k)) = \psi_{tU}(\phi_E(a) k),
\]
for all $k \in \K(E)$.
However $\phi_F(b) \in \K(F) = \psi_U(\K(E))$ and by considering an approximate identity for $\K(E)$ we then have the required equation.
\end{proof}

\subsection{Subproduct systems}\label{S: SPS}

We will require knowledge of subproduct systems as initiated by the second author and Solel \cite{ShaSol09} for W*-correspon\-dences and by Viselter \cite{Vis11} for C*-correspon\-dences.
Here we will restrict our attention to subproduct systems over $\bZ_+$.
We strongly recommend the work of Dor-On and Markiewicz \cite{DorMar14} for a series of elegant results.

Let $E$ be a non-degenerate C*-correspondence over $A$.
A \textit{standard subproduct system $X$} consists of a sequence $(X(n), p_n)$ such that:
\begin{enumerate}
\item $X(0) = A$, $X(1) = E$;
\item $X(n)$ is an orthocomplemented subcorrespondence of $E^{\otimes n}$, for every $n$;
\item $X(m+n) \subseteq X(m) \otimes X(n)$ for every $m,n$; and
\item if $p_{n}$ is the orthogonal projection of $E^{\otimes n}$ onto $X(n)$ then the following associativity condition holds
    \[p_{k+m+n} ( I_{E^{\otimes k}} \otimes p_{m+n}) = p_{k+m+n}( p_{k+m} \otimes I_{E^{\otimes n}}) = p_{k+m+n}. \]
\end{enumerate}
Non-degeneracy implies that $X(0) \otimes X(n) = X(n)$ for all $n \in \bZ_+$.
In the case of W*-correspondences it suffices to assume that the $X(n)$ are closed subspaces (since then they are automatically orthocomplemented \cite{Pas73}).

For $p = \sum_n p_n$ in the full Fock space $\F_E$ define $\F_X = p \F_E$ and the compression operators
\[
T(a) = p \pi_\infty(a)|_{\F_X} \qand T(\xi) = p t_\infty(\xi)|_{\F_X}
\]
for all $a \in A$ and $\xi \in X(n)$ with $n \geq 1$.
The associativity condition on the $p_n$ shows that $\F_X$ is co-invariant for $t_\infty$ and reducing for $\pi_\infty$.
The nonselfadjoint operator algebra $\A_X$ generated by the $T(a)$ and $T(\xi)$ will be called \emph{the tensor algebra of $X$}.
Notice here that
\[
\A_X = \ol{\alg}\{p\pi_\infty(a), p t_\infty(\xi) \mid a\in A, \xi \in \F_X\} = \ol{p \cdot \T_E^+}.
\]
We will write
\[
\ca(T):= \ca(\A_X) \subseteq \B(\F_X).
\]
The C*-algebras $\ca(T)$ and $\ca(T) / \K(\F_X)$ are the so-called \emph{Toeplitz algebra} and \emph{Cuntz algebra}, respectively, of the subproduct system $X$.
In parts of the literature they are denoted by $\T_X$ and $\O_X$ \cite{DRS11}, or $\T(X)$ and $\O(X)$ \cite{DorMar14,Vis12} (in fact, the Cuntz algebra $\O(X)$ is defined slightly differently in \cite{Vis12}, but for all subproduct systems encountered in this paper, Viselter's definition coincides with $\ca(T) / \K(\F_X)$).
We avoid using this abbreviated notation here so as not to create confusion with $\T_E$ or $\O_E$, with a couple of exceptions required so as to make a connection to the literature.
Likewise, the tensor algebra $\A_X$ is also denoted by $\T^+(X)$ or $\T^+_X$, notation that we will also avoid.

The subproduct systems $X = (X(n), p_n)$ and $Y=(Y(n), q_n)$ are called \emph{similar} if there exists a sequence $V = (V_n)$ of invertible C*-correspondences maps $V_n \colon X(n) \to Y(n)$ such that $\sup_n \nor{V_n} < \infty$, $\sup_n \nor{V_n^{-1}} < \infty$ and
\begin{align*}
(\dagger) \quad V_{m+n} (p_{m+n}(\xi \otimes \eta)) = q_{m+n}(V_m \xi \otimes V_n \eta)
\end{align*}
for all $m,n$ and all $\xi \in X(m), \eta \in X(n)$.
Then $V = (V_n)$ is said to be a \emph{similarity}.
It follows that $V$ satisfies $(\dagger)$ if and only if
\[
V_n p_n = q_n \cdot \underbrace{V_1 \otimes \cdots \otimes V_1}_{n-\text{times}}, \foral n \in \bZ_+.
\]
By assumption $E^{\otimes n} = X(n) \oplus X(n)^\perp$, hence $V = (V_n)$ satisfies $(\dagger)$ if and only if
\[
\underbrace{V_1 \otimes \cdots \otimes V_1}_{n-\text{times}} = \begin{bmatrix} V_n & 0 \\ \ast & \ast \end{bmatrix} \in \L(X(n) \oplus X(n)^\perp, Y(n) \oplus Y(n)^\perp).
\]
Reflexivity of similarity is induced by $(V_1 \otimes \cdots \otimes V_1)^{-1} = V_1^{-1} \otimes \cdots \otimes V_1^{-1}$.

The subproduct systems $X = (X(n), p_n)$ and $Y=(Y(n), q_n)$ are called \emph{isomorphic} if they are similar by a sequence $U = (U_n)$ of unitary C*-correspondences maps $U_n \colon X(n) \to Y(n)$.
However in this case every $U_n$ is a unitary and a compression of the unitary $U_1^{\otimes n}$, hence we obtain
\[
\underbrace{U_1 \otimes \cdots \otimes U_1}_{n-\text{times}} = \begin{bmatrix} U_n & 0 \\ 0 & \ast \end{bmatrix} \in \L(X(n) \oplus X(n)^\perp, Y(n) \oplus Y(n)^\perp).
\]
for every $n$.
Therefore there is a significant difference between similar and isomorphic subproduct systems.
The reader is addressed to \cite{DRS11} for examples that depict this phenomenon.

\section{Subproduct systems of homogeneous ideals}\label{S: sps hom}

When $E = \bC^d$ (over $\bC)$ then the resulting subproduct systems are characterised by homogeneous ideals of the polynomial ring $\bC\sca{x_1, \dots, x_d}$ in $d$ noncommuting variables \cite{ShaSol09}.
Let $\{e_1, \dots, e_d\}$ be an o.n. basis of $\bC^d$.
For any word $w \in \bF_+^d$ with $w = w_1 \dots w_n$ we write $|w| := n$ for the length of $w$.
For an $X = (X(n), p_n)$ with $X(0) = \bC$ and $X(1) = E = \bC^d$ let
\[
\I_X := \spn\{f \in \bC\sca{x_1, \dots, x_d} \mid \exists n>0 \text{ such that } f(\un{e}) \in E^{\otimes n} \ominus X(n)\},
\]
with the understanding that
\[
f(\un{e}) = \sum_{w \in \bF_+^d} \la_w e_{w_1} \otimes \dots \otimes e_{w_{|w|}} \text{ when } f(\un{x}) = \sum_{w \in \bF_+^d} \la_w \un{x}^w.
\]
Then $\I_X$ is a homogeneous ideal of $\bC\sca{x_1, \dots, x_d}$.
On the other hand if $\I$ is a homogeneous ideal in $\bC\sca{x_1, \dots, x_d}$ let
\[
X_{\I}(n) := E^{\otimes n} \ominus \{f(\un{e}) \mid f \in \I\},
\]
and $p_n$ be the projection of $E^{\otimes n}$ onto $X_\I(n)$.
Then $X_\I = (X_\I(n), p_n)$ is a subproduct system.
By \cite[Proposition 7.2]{ShaSol09} we get the connection
\[
X_{\I_X} = X \qand \I_{X_\I} = \I.
\]

When $\I=(0)$ then we write $\fA_d$ for the multivariable noncommutative disc algebra.
It is accustomed to denote by $\un{L} = [L_1, \ldots, L_d]$ the generators of $\fA_d$.
When $\I$ is the commutator ideal $\I_d$, i.e. the ideal generated by $L_i L_j - L_j L_i$, then we will write $\A_d$ for the multivariable commutative disc algebra, that is, $\A_d$ is the compression of $\fA_d$ by the closure commutator ideal $\I_d$.

Given a homogeneous ideal $\I$ in $\bC\sca{x_1, \dots, x_d}$ the operators $T_i = T(e_i)$ that generate $\A_{X_\I}$ satisfy $f(\un{T}) = 0$ for all $f \in \I$.
If $\un{S} = [S_1, \ldots, S_d] \in \B(H^d, H)$ is a row contraction such that $f(\un{S}) = 0$ for all $f \in \I$ then the mapping $T_i \mapsto S_i$ extends to a unital completely contractive representation of $\A_{X_\I}$ in $\B(H)$ (for example, combine \cite[Theorem 7.5]{ShaSol09} with \cite[Theorem 8.2]{ShaSol09}).
There is an ``algebraic'' identification which we record for further use.

\begin{proposition}\label{P: id hom sps}
Let $X_\I = (X(n), p_n)$ be the subproduct system associated with a homogeneous ideal $\I$ in $\bC\sca{x_1,\dots,x_d}$.
If $\ol{\I}$ denotes the closure of $\I \equiv \{f(\un{L}) \mid f \in \I\}$ in $\fA_d$, then $\A_{X_\I}$ is completely isometrically isomorphic to $\fA_d / \ol{\I}$.
In particular, the isomorphism is given by $\psi(x + \ol{\I})=p x$, where $p = \sum_n p_n$, and $\A_{X_\I} = p \cdot \fA_d$.
\end{proposition}

\begin{proof}
Let us write $X$ for $X_\I$.
Let $\un{T} = [T_1, \dots, T_d]$ be the generators of $\A_{X}$.
By construction there is a unital completely contractive homomorphism $\fA_d \to \A_{X} : L_i \mapsto T_i$ for all $i=1, \dots, d$.
Its kernel contains $\I$ and consequently it contains $\ol{\I}$.
Thus we obtain a completely contractive homomorphism
\[
\psi \colon \fA_d/ \ol{\I} \to \A_{X} : L_i + \ol{\I} \mapsto T_i.
\]
On the other hand the row contraction $\un{\wh{L}} = [L_1 + \ol{\I}, \dots, L_d + \ol{\I}]$ satisfies $f(\un{\wh{L}}) = 0$ for all $f \in \I$.
By the remarks preceding the statement there is a unital completely contractive homomorphism $\si \colon \A_{X} \to \fA_d/ \ol{\I}$ such that $\si(T_i) = L_i + \ol{\I}$ for all $i=1,\dots, d$.

Thus $\psi \si$ is a completely contractive homomorphism sending $T_i$ to itself.
Hence it is the identity, and likewise for $\si \psi$.
Therefore $\psi \colon \fA_d/ \ol{\I} \to \A_{X}$ is a completely isometric isomorphism that maps every $L_i + \ol{\I}$ to $T_i = p L_i$.
Since $p L_i = p L_i p$, this extends to all polynomials $f \in \fA_d$.
The proof is completed by recalling that the quotient map $\fA_d \to \fA_d/\ol{\I}$ has closed range.
\end{proof}

\subsection{The character space}

A \emph{character} of an operator algebra, i.e. an algebraic homomorphism in $\bC$, is automatically completely contractive.
When $X$ is associated with a homogeneous ideal $\I$ in $\bC\sca{x_1,\dots,x_d}$ then the character space $\M_{\A_X}$ of $\A_X$ is identified with the set
\[
Z(\I) = \{z \in \ol{\bB}_d \mid f(z)=0 \foral f \in \I\},
\]
by the bijection
\[
\M_{\A_X} \ni \rho \longleftrightarrow (\rho(T_1),\dots, \rho(T_d)) \in Z(\I).
\]
Indeed, by Proposition \ref{P: id hom sps} we have that a character $\rho$ of $\A_X$ must vanish on $\ol{\I}$ and hence on $\I$.
Since $\rho$ is an algebraic homomorphism, we obtain that $f(\rho(\un{T})) = \rho(f(\un{T})) = 0$ for all $f \in \I$, which shows that $\rho(\un{T}) = [\rho(T_1),\dots, \rho(T_d)] \in Z(\I)$.
On the other hand if $z\in Z(\I)$ let $\rho(T_i) = z_i$.
Hence we have that $\rho(f(\un{T})) := f(\rho(\un{T})) = 0$ and by the remarks preceding Proposition \ref{P: id hom sps} we have that $\rho$ extends to a contractive homomorphism from $\A_X$ into $\bC$.

Given $\la \in Z(\I)$ we write $\rho_\la \in \M_{\A_X}$ obtained by the above identification.
Then for any element $t \in \A_X$ such that $t = \lim_n f_n(\un{T})$ for some polynomials $f_n \in \bC\sca{x_1, \dots, x_d}$ we get that
\[
\rho_\la(t) = \lim_n \rho_\la(f_n(\un{T})) = \lim_n f_n(\rho_\la(\un{T})) = \lim_n f_n(\la).
\]
Since $\A_X$ is the norm closure of polynomials in $T$, the above formula defines $\rho_\la$.
Furthermore every $\la \in Z(\I)$ is in $\ol{\bB}_d$, and the latter is the character space of $\fA_d$.
Even more $\ol{\bB}_d$ is the character space of the commutative algebra $\A_d$ as well.
Since it will be clear by the domain, we will write $\rho_\la$ for the characters of $\fA_d$, $\A_d$, and $\A_X$ defined by the same $\la \in Z(\I) \subseteq \ol{\bB}_d$.

\begin{lemma}\label{L: characters}
Let $t \in \A_X$ and $x \in \fA_d$ such that $t = x + \ol{\I}$ given by Proposition \ref{P: id hom sps}. If $\la \in Z(\I)$ then we obtain that $\rho_\la(t) = \rho_\la(x + \ol{\I}) = \rho_\la(x)$.

In particular, for every $x \in \fA_d$ the map $\wh{x} \colon \ol{\bB}_d \rightarrow \bC$ given by $\la \mapsto \rho_\la(x)$ is continuous on $\ol{\bB}_d$ and holomorphic in $\bB_d$.
\end{lemma}

\begin{proof}
The first equality holds because the isomorphism between $\A_X$ and $\fA_d/ \ol{\I}$ is isometric and preserves polynomials.
For the second equality recall that $f(\la) = 0$ for all $f \in \I$.
The von Neumann inequality $|f(\la)| \leq \nor{f(\un{L})}$ on $\fA_d$ \cite{Pop91} implies that $\rho_\la$ vanishes on $\ol{\I}$, thus it factors through the quotient.

Now we may apply in particular for the commutator ideal $\I_d$ to obtain that the function
\[
\la \mapsto \rho_\la(x) = \rho_\la(x + \ol{\I}_d)
\]
is in $\A_d$.
The latter is identified with the norm closure of polynomials in $\Mult(H^2_d)$ (see \cite{Arv98}) and the proof is complete.
\end{proof}

\subsection{Isomorphisms}

For the rest of the section fix the homogeneous ideals $\I \lhd \bC\sca{x_1,\dots,x_d}$ and $\J \lhd \bC\sca{y_1,\dots,y_{d'}}$.
If $\phi \colon \A_X \to \A_Y$ is an isomorphism then it defines a continuous map $\phi^* \colon \M_{\A_Y} \to \M_{\A_X}$.
We say that $\phi$ is a \emph{vacuum preserving isomorphism} if $\phi^* \rho_0 = \rho_0$ where we write $0$ for both zeroes in $\bB_d$ and in $\bB_{d'}$.

\begin{lemma}\label{L: vac pre}
Let $X$ and $Y$ be subproduct systems associated with the homogeneous ideals $\I \lhd \bC\sca{x_1,\dots,x_d}$ and $\J \lhd \bC\sca{y_1,\dots,y_{d'}}$.
If $\phi \colon \A_X \to \A_Y$ is an algebraic (resp. bounded, isometric) isomorphism then there exists a vacuum preserving algebraic (resp. bounded, isometric) isomorphism $\phi' \colon \A_X \to \A_Y$.
\end{lemma}

\begin{proof}
The proof follows verbatim by \cite[Proposition 4.7]{DRS11} once we prove an alternative of \cite[Lemma 4.4]{DRS11} for algebraic and for bounded isomorphisms.
To this end we have to show that if $\phi \colon \A_X \to \A_Y$ is an isomorphism, then there exists a continuous map $\wt{F} \colon \ol{\bB}_{d'} \rightarrow \bC^{d}$ that is holomorphic on $\bB_{d'}$ and extends $\phi^* \colon \M_{\A_Y} \to \M_{\A_X}$.
Let $F = \phi^*$ be the (continuous) map
\[
F \colon Z(\J) \rightarrow Z(\I) \, : \, \la  \mapsto (\phi(T_1)(\la), \dots, \phi(T_d)(\la)),
\]
after the identification $\M_{\A_X} \simeq Z(\I)$ and $\M_{\A_Y} \simeq Z(\J)$.
For every $i=1,\dots, d$, let $w_i \in \fA_{d'}$ so that
\[
\phi(T_i) = \psi(w_i + \ol{\J}) = q w_i,
\]
where $\psi\colon \fA_{d'}/\ol{\J} \to \A_Y$ is the isomorphism obtained by Proposition \ref{P: id hom sps}, and $q$ is the defining projection of $Y$.
Then the map $\la \mapsto \phi(T_i)(\la)$ on $Z(\J)$ extends to the continuous map $\wh{w_i} \colon \ol{\bB}_{d'} \mapsto \bC$ which is holomorphic on $\bB_{d'}$ by Lemma \ref{L: characters}.
The required map $\wt{F} \colon \ol{\bB}_{d'} \rightarrow \bC^{d}$ is then defined by $\wt{F}(\la) = (\wh{w_1}(\la),\dots, \wh{w_d}(\la))$.
\end{proof}

\begin{theorem}\label{T: sps cla}
Let $X$ and $Y$ be subproduct systems associated with the homogeneous ideals $\I \lhd \bC\sca{x_1,\dots,x_d}$ and $\J \lhd \bC\sca{y_1,\dots,y_{d'}}$.
Then:
\begin{enumerate}
\item $\A_X$ and $\A_Y$ are completely isometrically isomorphic if and only if $\A_X$ and $\A_Y$ are isometrically isomorphic if and only if $X$ and $Y$ are isomorphic;
\item $\A_X$ and $\A_Y$ are isomorphic by a completely bounded map if and only if $\A_X$ and $\A_Y$ are isomorphic as topological algebras if and only if $X$ and $Y$ are similar.
\end{enumerate}
\end{theorem}

\begin{proof}
It suffices to show that isometric isomorphism and bounded isomorphism imply isomorphism and similarity respectively.
By Lemma \ref{L: vac pre} we may assume that the isomorphisms are vacuum preserving.
Then the isometric case follows by \cite[Theorem 9.7]{ShaSol09}.
On the other hand if $\phi \colon \A_X \to \A_Y$ is a vacuum preserving bounded isomorphism, then it is \emph{semi-graded} by \cite[Proposition 6.16]{DorMar14}.
As a consequence it implements a similarity by \cite[Proposition 6.17]{DorMar14} and \cite[Proposition 6.12]{DorMar14} finishes the proof.
The converses follow from \cite[Proposition 6.12, Corollary 6.13]{DorMar14}.
\end{proof}

\section{Monomial ideals}\label{S: monomial}

An ideal $\I$ of $\bC\sca{x_1, \dots, x_d}$ is called \emph{monomial} if it is generated by monomials.
Monomial ideals in noncommuting variables are homogeneous and may be generated by infinite sets.
For example take $\I = \sca{ xy^nx \mid n \geq 2}$.
A set $S$ of monomials is called \emph{generating for $\I$} if $\I = \sca{S}$.
\emph{The degree $\deg S$ of $S$} is the maximum of the degrees of the monomials in $S$.
We say that $\I$ is of \emph{type $k$} if it is generated by a set $S$ with $\deg S = k +1$ and $\deg S \leq \deg S'$ for any generating set $S'$ of $\I$.
Then $k$ is finite if and only if $\I$ is generated by a finite set $S$.

Every monomial in $\bC\sca{x_1, \dots, x_d}$ corresponds to a finite sequence of the $x_i$.
Therefore we identify monomials in $\bC\sca{x_1, \dots, x_d}$ with finite words in $\bF_+^d$.
We call $\fF := \{w \in \bF_+^d \mid \un{x}^w \in \I\}$ \emph{the set of forbidden words}, and its complement $\La^*$ in $\bF_+^d$ \emph{the set of allowable words}.
We write
\[
\fF_l := \{ \mu \in \fF \mid |\mu| \leq l\} \qand \La^*_l:=\{\mu \in \La^* \mid |\mu| \leq l\}.
\]
We see that a word is forbidden if and only if it contains a forbidden word.
Hence if $\nu \mu \in \La^*$ then $\nu \in \La^*$ and $\mu \in \La^*$ as well.

In general an ideal $\I$ of $\bC\sca{x_1, \dots, x_d}$ may not have a \emph{basis}, that is a smallest generating set.
However when $\I$ is monomial we can find a basis in the following way.
Let
\[
\I^{(n)} = \{ f \in \I \mid \deg f = n\} = \spn \{ \un{x}^w \in \I \mid w \in \bF_+^d, |w| = n\}.
\]
and begin with $\I^{(n)}$ such that $\I^{(k)} = (0)$ for all $k < n$ and set $S^{(n)} = \I^{(n)}$ and $S^{(k)} = \mt$ for $k < n$.
In $\I^{(n+1)}$ let $S^{(n+1)}$ be the set that consists of monomials that cannot be generated by $S^{(n)}$ and continue inductively.
Then $S = \cup_{n=0}^\infty S^{(n)}$ is a generating set we can get for $\I$.
In particular if $\I$ is of finite type $k$ then $S = \cup_{n=0}^k S^{(n)}$.

Without loss of generality, we will only consider monomial ideals which are generated by monomials of degree $2$ or more, i.e of type greater than $1$. Equivalently, we will consider all the letters of the alphabet as allowed.
Indeed, if $x_i$ is in the ideal then we may simply throw out of the alphabet the letter $i$, and the following constructions produce the same outcome.

\begin{definition}
We say that a monomial $\un{x}^\mu \in \bC\sca{x_1,\dots,x_d} \setminus \I$ is \emph{a sink on the left for $\I$} (resp. \emph{a sink on the right for $\I$}) if $x_i \un{x}^\mu \in \I$ (resp. $\un{x}^\mu x_i \in \I$) for all $i=1, \dots, d$.
\end{definition}

\subsection{Fock representation}

We will write $X$ instead of $X_\I$ for the subproduct system associated with a monomial ideal $\I$.
Let the operators $T_i = T(e_i) \in \B(\F_X)$.
We fix once and for all
\[
\ca(T):=\ca(I, T_i \mid i=1, \dots, d).
\]
By definition the projections $p_n$ are diagonal, meaning that
\[
p_n e_\nu = \begin{cases} e_\nu & \qif |\nu|=n \text{ and } \nu \in \La^*,
\\ 0 & \qotherwise. \end{cases}
\]
Therefore $X(n) = \spn\{ e_\nu \in (\bC^d)^{\otimes n} \mid \nu \in \La^*, |\mu| = n\}$ and for every $\mu \in \La^*$ we have that
\[
T_\mu e_\nu =
\begin{cases}
e_{\mu\nu} & \qif \mu\nu \in \La^*, \\
0 & \qotherwise,
\end{cases}
\]
with the understanding that $T_\mu e_\mt = e_\mu$ and $T_\mt = I$.
Moreover
\[
T_\mu T_\nu =
\begin{cases}
T_{\mu \nu} & \qif \mu\nu \in \La^*,\\
0 & \qotherwise.
\end{cases}
\]
However it is convenient to write $T_\mu T_\nu = T_{\mu\nu}$ even when $T_{\mu\nu} = 0$, i.e.  $T_{\mu\nu} = 0$ if and only if $\mu \nu \notin \La^*$.
We will also write $T_\mt = I$.
The operators $T_\mu$ satisfy a list of properties which we gather in the following lemma.

\begin{lemma}\label{L: list Ti}
Let $T_\mu, T_\nu$ for $\mu, \nu \in \La^*$ as above.
Then the following hold:
\begin{enumerate}
\item $T_\mu^*T_\mu$ is an orthogonal projection on $\ol{\spn}\{e_\nu \mid \mu\nu \in\La^*\}$;
\item $T_\nu T_\nu^*$ is an orthogonal projection on $\ol{\spn}\{e_{\nu\mu} \mid \mu \in\La^*\}$;
\item if $|\mu| = |\nu|$, then $T_\mu^* T_\nu = 0$ if and only if $\mu \neq \nu$;
\item $T_\mu^*T_\mu$ commutes with $T_\nu^* T_\nu$, and with $T_\nu T_\nu^*$;
\item $T_\mu^* T_\mu \cdot T_i = T_i \cdot T_{\mu i}^* T_{\mu i}$ for all $i = 1, \dots, d$;
\item $\sum_{i=1}^d T_i T_i^* + P_\mt = I$ where $P_\mt$ is the projection on $\bC e_\mt$;
\item the rank one operator $e_\nu \mapsto e_\mu$ equals $T_\mu P_\mt T_\nu^*$;
\item $\K(\F_X) \subseteq \ca(T)$.
\end{enumerate}
\end{lemma}

On the other hand such a structure identifies the monomial ideals.

\begin{proposition}\label{P: mon id}
Let $\I$ be a homogeneous ideal in $\bC\sca{x_1,\dots,x_d}$.
Then $\I$ is a monomial ideal if and only if the $T_i := T(e_i)$ are partial isometries on some $e_\nu$, and have pairwise orthogonal ranges.
\end{proposition}

\begin{proof}
Lemma \ref{L: list Ti} contains the forward implication.
For the converse we have to show that $p_n e_\nu$ is $0$ or $e_\nu$, for all $e_\nu \in (\bC^d)^{\otimes n}$.
It is evident that this holds for $n=1$ and suppose that it is true for $p_n$.
Fix a word $\nu \in \La^*$ such that $|\nu| = n$ and suppose that $p_{n+1} e_{i\nu} = \sum_{|\mu| = n+1} \la_\mu e_\mu$.
Then $T_j^* T_i e_\nu = \sum_{\mu = j \mu'} \la_\mu e_{\mu'}$ for all $j=1, \dots, d$.
For $j \neq i$ we get that $\la_{\mu} = 0$ for all $\mu = j \mu'$.
Now $T_i^* T_i e_\nu$ is either $0$ or $e_\nu$.
Hence we get that $\la_\mu =0$ when $\mu = i \mu'$ and $\mu' \neq \nu$, and that $\la_{i \nu} =0$ or $1$.
Therefore $p_{n+1} e_{i\nu}$ is either $0$ or $e_{i \nu}$ for all $i$, which shows that $p_{n+1}$ is diagonal.
\end{proof}

\subsection{$Q$-Projections}

We will be using the projections generated by the $T_i^*T_i$.
To this end we introduce the following enumeration.
Write all numbers from $0$ to $2^d-1$ by using $2$ as a base, but in reverse order.
Hence we write $[m]_2 \equiv [m] = [m_1 m_2 \dots m_d]$ so that $2=[0100\dots 0]$, and we define
\[
\supp [m] :=\{i =1, \dots d \mid m_i =1 \}.
\]
Let the (not necessarily one-to-one) assignment $[m] \mapsto Q_{[m]}$ given by
\[
Q_{[m]} \equiv Q_{[m_1 \dots m_d]} : = \prod_{m_i \in \supp [m]} T_{i}^* T_{i} \cdot \prod_{m_i=0} (I-T_{i}^*T_{i}),
\]
where $I \in \B(\F_X)$.
For example we write
\[
Q_{0} = Q_{[0\dots 0]} = \prod_{i=1}^d (I - T_i^*T_i) \qand Q_{2^d-1} = Q_{[1\dots 1]} = \prod_{i=1}^d T_i^*T_i.
\]
The $Q_{[m]}$ are the minimal projections in the C*-subalgebra $\ca(T_i^*T_i \mid i=1, \dots, d)$ of $\ca(T)$.
Consequently we obtain $\sum_{[m]=0}^{2^d -1} Q_{[m]} = I$.

\begin{lemma}\label{L: unital}
Let $\I$ be a monomial ideal in $\bC\sca{x_1,\dots,x_d}$.
If $Q_{[m]}$ are the projections defined above, then
\[
T_\mu^* T_\mu = T_\mu^* T_\mu \cdot \sum_{[m] = 1}^{2^d-1} Q_{[m]} = \sum_{[m] = 1}^{2^d-1} Q_{[m]} \cdot T_\mu^* T_\mu
\]
for all $\mt \neq \mu \in \La^*$.
In particular the unit of $\ca(T_i^*T_i \mid i=1, \dots, d)$ coincides with $I \in \B(\F_X)$ if and only if there are no sinks on the left for $\I$.
\end{lemma}

\begin{proof}
For $\mu = \nu j \in \La^*$ we have that $T_\mu^* T_\mu = T_j^* T_\nu^*T_\nu T_j  \leq T_j^*T_j$ and therefore $T_\mu^* T_\mu \cdot \prod_{i=1}^d (I - T_i^*T_i) = 0$.
Hence we obtain
\[
T_\mu^*T_\mu \cdot \sum_{[m] = 1}^{2^d-1} Q_{[m]} = T_\mu^*T_\mu (I - Q_{0}) = T_\mu^* T_\mu (I - \prod_{i=1}^d (I - T_i^*T_i)) = T_\mu^*T_\mu.
\]

For the second part of the proof we have that $I = \sum_{[m]=1}^{2^d-1} Q_{[m]}$ if and only if $Q_{0} = 0$.
Suppose that $Q_{0} = 0$ and that there is a word $\mu \in \La^*$ such that $i \mu \notin \La^*$ for all $i=1, \dots, d$.
For this $\mu$ we obtain
\[
0 = Q_{0} e_\mu = \prod_{i=1}^d (I - T_i^*T_i) e_\mu = e_\mu
\]
which is a contradiction.
Conversely, if there are no sinks on the left for $\I$ then for every $\mu \in \La^*$ there is an $i_0$ such that $i_0 \mu \in \La^*$, hence $(I - T_{i_0}^*T_{i_0}) e_\mu = 0$.
Commutativity in $\ca(T_i^*T_i \mid i=1, \dots, d)$ shows that $Q_{0} e_\mu = 0$ for all $\mu \in \La^*$, and the proof is complete.
\end{proof}

\subsection{Subshifts}\label{Ss: subshifts}

Special examples of monomial ideals relate to symbolic dynamics.
Let us give a brief description.

For the fixed symbol set $\Si = \{ 1, \dots, d \}$ let $\Si^\bZ$ be endowed with the product topology and let $\si \colon \Si^{\bZ} \to \Si^{\bZ}$ be the backward shift with $\si((x_i))_k = x_{k+1}$.
The pair $(\La_-, \si^{-1})$ is called a \emph{left subshift} if $\La_-$ is a closed subset of $\Si^{\bZ_-}$ with $\si^{-1}(\La_-) \subseteq \La_-$.
Respectively $(\La_+, \si)$ is called a \emph{right subshift} if $\La_+$ is a closed subset of $\Si^{\bZ_+}$ with $\si(\La_+) \subseteq \La_+$, and $(\La,\si)$ is called a \emph{two-sided subshift} if $\La$ is a closed subset of $\Si^{\bZ}$ with $\si(\La) = \La$.

A word $\mu = i_1 \dots i_n$ is said to occur in some (one-sided or two-sided) sequence $(x_i)$ if there is an $m$ such that $x_{m} = i_1, \dots, x_{m+n-1} = i_n$.
Let $\fF$ be a set of words and let
\[
\La_\fF : = \{(x_i) \in \Si^{\bZ_-} \mid \text{ no $\mu \in \fF$ occurs in $(x_i)$ } \}.
\]
All left subshifts arise this way, and a similar construction gives rise to all right or two-sided subshifts.
Indeed, if $\La_-$ is a left subshift, then by setting
\[
\fF_k = \{ \mu \in \fF \mid |\mu| \leq k, \text{ $\mu$ does not occur in $(x_i) \in \La_-$}\}
\]
we see that $\fF_k \subseteq \fF_{k+1}$ and $\fF = \bigcup_k \fF_k$; then $\La_- = \cap_k \La_{\fF_k} = \La_\fF$.
As in the case of monomial ideals the set of forbidden words admits a basis.
We say that $\La_-$ is a left subshift of \emph{finite type $k+1$} if the longest word in the basis of $\fF$ has length $k$.
Similar comments hold for the right or two-sided subshifts.

The forbidden words form a monomial ideal in $\bC\sca{x_1, \dots, x_d}$ and hence define a subproduct system.
The second author and Solel \cite[Section 12]{ShaSol09} give a characterisation of two-sided subshifts by monomial ideals.
This generalises to the one-sided shifts as well.

\begin{proposition}\label{P: subshift}
Let $\I = \sca{S}$ be a monomial ideal of $\bC\sca{x_1, \dots, x_d}$ with basis $S$.
Then $\I$ gives rise to a left subshift $(\La_-,\si^{-1})$  (resp. right subshift $(\La_+,\si)$, two-sided subshift $(\La,\si)$) on the forbidden words $S$ if and only if there are no sinks on the left (resp. on the right, or on either side) for $\I$.
In addition $\I$ is of finite type $k$ if and only if the subshift is of finite type $k$.
\end{proposition}

\begin{proof}
There are no sinks on the left for $\I$ if and only if for all $n,k \in \bZ_+$ we have that for every $\mu \in \La^*$ of length $n$ there exists a $\nu \in \La^*$ of length $k$ such that $\nu\mu \in \La^*$.
The proof then follows in the same way as in \cite[Proposition 12.3]{ShaSol09} with the modification that the required sequences are taken in $\Si^{\bZ_-}$.
An analogous proof follows for the right subshifts.
\end{proof}

The main point is that the allowable words in a left (resp. right) subshift come from a left-infinite (resp. right-infinite) sequence of symbols.
Hence for every length there must be at least one word that we can concatenate on the left (resp. on the right).
It is evident that not all monomial ideals come from subshifts.
In particular, the allowable words in a subshift related to a set $S$ of forbidden words may be different from the allowable words given directly by the monomial ideal generated by $S$.
We underline this by the following examples.

\begin{examples}
Let the symbol space $\{1, 2\}$.
Then the subproduct system of the ideal $\sca{x_1^2, x_1x_2} \lhd \bC\sca{x_1, x_2}$ differs from the subproduct system related to the right subshift on the forbidden words with basis $\{11, 12\}$.
Indeed in the latter case we have only one allowable right-infinite word $(22\dots)$.

On the other hand the subproduct system associated with the monomial ideal $\sca{x_1^2, x_1x_2} \lhd \bC\sca{x_1, x_2}$ coincides with the subproduct system of the left subshift on the forbidden words with basis $\{11, 12\}$.
In this case the subshift contains the left-infinite words $(\dots22)$ and $(\dots21)$.

As a second example, we have that the subproduct system associated with the ideal $\sca{x_1x_2, x_2x_1}$ coincides with the subproduct system of the left subshift, the right subshift and the two-sided subshift on the forbidden words with basis $\{12, 21\}$.
All subshift spaces consist of two points $(\dots 11.11 \dots)$ and $(\dots 22.22 \dots)$.
\end{examples}

\subsection{The quantised dynamics on the allowable words}\label{Ss: qd}

Given a monomial ideal we isolate a dynamical system originating from the C*-algebra generated by the $T_\mu^*T_\mu$.
As we will see later there is a strong connection between these dynamics and a class of nonselfadjoint operator algebras (a further description and study of the quantised dynamics can be found in the preprint of Christopher Barrett with the first author \cite{BarKak}, that followed the current work).

We fix once and for all the unital C*-subalgebra
\[
A:=\ca(T_\mu^* T_\mu \mid \mu \in \La^*)
\]
of $\ca(T)$.
We define the positive maps $\al_i \colon A \to A$ such that $\al_i(a) = T_i^* a T_i$.
Since $T_\mu^*T_\mu$ commutes with $T_iT_i^*$ and $\al_i(T_\mu^* T_\mu) = T_{\mu i}^* T_{\mu i} \in A$ for all $\mu \in \La^*$ then every $\al_i$ is a $*$-endomorphism of $A$.
Our objective is to describe the dynamical system $(A,\al) \equiv (A,\al_1, \dots, \al_d)$ in terms of the original data.

\begin{proposition}\label{P: A is AF}
The C*-algebra $A = \ca(T_\mu^* T_\mu \mid \mu \in \La^*)$ is a unital commutative AF algebra.
\end{proposition}

\begin{proof}
We see that $A= \ol{\cup_l A_l}$ for the commutative C*-subalgebras
\[
A_l = \ca(T_\mu^* T_\mu \mid |\mu| \leq l, \mu \in \La^*) \text{ with } l \geq 1,
\]
of $\ca(T)$.
Every $A_l$ is finite dimensional since it is generated by a finite family of projections and by definition $A_l \subseteq A_{l+1}$.
\end{proof}

The C*-algebra $A$ can be characterised by using the allowable words in $\La^*$.
For $l \geq 0$ we define the equivalence relation $\sim_l$ on $\La^*$ by the rule
\[
\mu \sim_l \nu \Leftrightarrow \{w \in \La^*_l \mid w\mu \in \La^*\} = \{w \in \La^*_l \mid w \nu \in \La^* \}.
\]
We stress that this equivalence relation is different than the ones considered by Matsumoto \cite{Mat97} and by Carlsen-Matsumoto \cite{CarMat04}.
We define the topological space $\Om_l = \La^*/ \sim_l$ and write $[\mu]_l$ for the points in $\Om_l$.
Every $\mu \in \La^*$ partitions $\La^*_l$ into the set of the $w_i \in \La^*_l$ for which $w_i \mu \in \La^*$ and its complement.
There is a finite number of such partitions since $\La^*_l$ is finite.
These partitions completely identify single points in $\Om_l$, hence $\Om_l$ is a discrete finite space.
Furthermore the mapping
\[
\vth \colon \Om_{l+1} \to \Om_l: [\mu]_{l+1} \mapsto [\mu]_l
\]
is well defined (continuous) and onto.
Indeed if $\mu \not\sim_l \nu$ then we may suppose that there exists a word $w$ with $|w| \leq l$ such that $w \mu \in \La^*$ and $w \nu \notin \La^*$; hence $[\mu]_{l+1} \neq [\nu]_{l+1}$, so $\vth$ is well defined.
Continuity and surjectivity are evident.
We can then form the projective limit $\Om$ by the directed sequence
\[
\xymatrix{
\Om_0 & \Om_1 \ar[l]_\vth & \Om_2 \ar[l]_\vth & \dots \ar[l]_\vth & \ar[l]  \Om
}
\]
for which we obtain the following identification.
We write $\Om_\I$ for $\Om$ when we want to highlight the ideal $\I$ which $\Om$ is related to.

\begin{proposition}
Let $\I \lh \bC\sca{x_1, \dots, x_d}$ be a monomial ideal.
With the aforementioned notation, we have that $A_l \simeq C(\Om_l)$ and consequently $A \simeq C(\Om)$, where $A = \ol{\cup_l A_l}$ for $A_l = \ca(T_\mu^*T_\mu \mid \mu \in \La^*_l)$.
\end{proposition}

\begin{proof}
Recall that $A_l = \ca(T_\mu^* T_\mu \mid \mu \in \La^*_l)$ is generated by its minimal projections.
If $\La^*_l = \{\mu_1, \dots, \mu_n\}$ then the minimal projections arise as products of some $T_{\mu_i}^* T_{\mu_i}$ and the rest of the $I-T_{\mu_j}^*T_{\mu_j}$.
For the sake of clarity suppose that we have the minimal projection
\[
a = \prod_{j= 1}^k T_{\mu_j}^* T_{\mu_j} \prod_{j= k + 1}^n (I - T_{\mu_j}^* T_{\mu_j})
\]
and let a point $[\nu]_l \in \Om_l$ such that $\mu_1 \nu ,\dots, \mu_k \nu \in \La^*$ and $\mu_{k+1} \nu, \dots, \mu_n \nu \notin \La^*$.
At least one such $\nu$ exists if and only if $a \neq 0$.
Furthermore $[\nu]_l = [\nu']_l$ for any $\nu' \in\La^*$ with $\mu_1 \nu' ,\dots, \mu_k \nu' \in \La^*$ and $\mu_{k+1} \nu', \dots, \mu_n \nu' \notin \La^*$.
Therefore the mapping that associates $a$ to $\chi_{[\nu]_l}$ for the $\nu$ as above is well defined and one-to-one.
Thus it extends to an injective $*$-homomorphism $A_l \to C(\Om_l)$.
Now every $[\nu]_l \in \Om_l$ is completely characterised by the $\mu_j \in \La^*_l$ for which $\mu_j \nu \in \La^*$.
Hence the $*$-homomorphism is surjective.

Finally in order to show that $A \simeq C(\Om)$ it suffices to show that the diagram
\[
\xymatrix{
C(\Om_l) \ar[rr]^{\al_\vth} \ar[d] & & C(\Om_{l+1}) \ar[d] \\
A_l \ar[rr]^{\id} & & A_{l+1}
}
\]
commutes, where $\al_\vth(f) = f  \vth$.
This follows by the definitions since $\al_\vth(\chi_{[\nu]_l})$ is the characteristic function on the set $\{ [\mu]_{l+1} \in \Om_{l+1} \mid [\mu]_l = [\nu]_l\}$.
\end{proof}

\begin{proposition}\label{P: finite A}
Let $\I$ be a monomial ideal of $\bC\sca{x_1, \dots, x_d}$.
If $\I$ is of finite type $k$ then $A_l = A_{k}$ for all $l \geq k+1$.
Consequently $A$ is generated by a finite number of pairwise orthogonal projections, hence $\Om$ is a finite (discrete) set.
\end{proposition}

\begin{proof}
Fix $l \geq k + 1$.
We have to show that if $\mu \sim_{k} \nu$ then $\mu \sim_l \nu$.
If $w \mu \notin \La^*$ and $w\nu \notin \La^*$ for all $ w\in \La^*_l$ then $\mu \sim_l \nu$.
On the other hand if $w \mu \notin \La^*$ for all $w \in \La^*_l$ but there is one such that $w \nu \in \La^*$ then we exchange the roles of $\mu$ and $\nu$.
Hence without loss of generality let $w \in \La^*_l$ such that $w\mu \in \La^*$ and suppose that $w \nu \notin \La^*$ in order to reach contradiction.

Necessarily we have that $\nu \neq \mt$, hence $w \nu$ is a forbidden word of length $l + |\nu| \geq l + 1 > k+1$.
Therefore there are words $y,z \in \La^*$ and $q \notin \La^*$ with $|q| \leq k+1$ such that $w \nu = y q z$.
The word $z$ cannot contain $\nu$, for if $z = \nu' \nu$ then $w = yq \nu' \notin \La^*$ which is a contradiction.
Also the word $y$ cannot contain $w$, for if $y = w w'$ then $\nu = w' q z \notin \La^*$ which is a contradiction.
In the above cases we have included the options $z=\nu$ or $y = w$; consequently $\nu = \nu' z$ and $w = y w'$ with $\nu', w' \neq \mt$.
Thus we obtain $w' \nu = q z$, hence $|w'| + |\nu| = |q| + |z|$.
Since $|z| < |\nu|$ we get $|w'| < |q| \leq k+1$ so that $|w'| \leq k$.
The equation $w'\nu = q z$ shows that $w' \nu \notin \La^*$ and since $\nu \sim_k \mu$ we have that $w' \mu \notin \La^*$.
As a consequence $w \mu = y w' \mu \notin \La^*$ which is a contradiction.
\end{proof}

The converse of the above proposition does not hold.

\begin{example}\label{E: converse finite}
Consider the ideal $\I = \sca{x_1x_2^nx_1 \mid n \geq 1}$, which is not of finite type.
The sets $\Om_l$ stabilise at $l = 2$ at the equivalence classes $[\mt]_2$, $[x_1]_2$ and $[x_2x_1]_2$, and therefore $\Om = \{[\mt]_2, [x_1]_2, [x_2x_1]_2\}$.
Indeed, the only allowable words are of the form $\mt$, $x_1^m$, $x_2^n$, $x_1^m x_2^n$, $x_2^n x_1^m$, and $x_2^n x_1^m x_2^k$ for $k,m,n > 0$, and it is straightforward to check that all allowable words fall into one of the three specified $\sim_l$ equivalence classes for $l \geq 2$.
\end{example}

\begin{remark}\label{R: sofic}
In light of the above example, we wish to describe the class of monomial ideals for which $A$ is finite dimensional.
In the case where $\I$ comes from a two-sided subshift $(\La,\si)$ then $A$ is finite dimensional if and only if the subshift is {\em sofic}.
Indeed, recall that $(\La,\si)$ is called \emph{sofic} if it is a factor of a subshift of finite type.
This is equivalent to having a finite number of equivalence classes with respect to the equivalence relation
\[
\mu \sim \nu \Leftrightarrow \{w \in \La^* \mid \mu w \in \La^*\} = \{w \in \La^* \mid \nu w \in \La^*\},
\]
e.g. \cite[Theorem 3.2.10]{LinMar95}.
By using the forward shift $(\La,\si^{-1})$ and similar ideas one may show that the same holds for the equivalence relation
\[
\mu \sim \nu \Leftrightarrow \{w \in \La^* \mid w \mu \in \La^*\} = \{w \in \La^* \mid w \nu \in \La^*\},
\]
and eventually obtain that $A$ is finite dimensional if and only if the subshift $(\La,\si)$ is sofic.
In short if $A$ is not finite dimensional then the sets $\Om_l$ never stabilize.
But the number of equivalence classes with respect to $\sim$ is at least $|\Om_l|$, thus the subshift is not sofic.
Conversely if $A$ is finite dimensional then the $A_l$ stabilise eventually, say at $k$, thus $\mu \sim_k \nu$ implies that $\mu \sim_l \nu$ for all $l \geq k$, and hence that $\mu \sim \nu$.

More generally, a monomial ideal $\I$ may not come from a subshift.
In this case we pass to an augmented system.
That is if $\I \lhd \bC\sca{x_1,\dots,x_d}$ then let $(\hat\La,\si)$ be the two-sided subshift generated by $\I'=\sca{\I}$ as an ideal in $\bC\sca{x_0, x_1,\dots, x_d}$.
Indeed $\I'$ has no sinks on either side and the allowable words of the augmented shift $(\hat{\La}, \si)$ are then of the form
\[
0^{n_1} \mu_1 0^{n_2} \mu_2 \ldots \mu_k 0^{n_k}
\]
for $\mu_i \in \La^*$ and some $n_i \in \bZ_+$.
Then the quantised space $\hat{A}$ related to $\hat\La^*$ coincides with the quantised space $A$ of $\La^*$.
For example the equivalence class of $\mu_1 0^{n_2} \dots \mu_k 0^{n_k}$ coincides with the equivalence class of the last word $\mu_1$ appearing on the left, since $0$ does not interfere within the forbidden words from $\I$.
Therefore $A$ is finite dimensional if and only if the augmented shift $(\hat\La,\si)$ is sofic.
Note that it does not make any difference if we introduce more than one new variables for defining the augmented two-sided subshift.
\end{remark}

Next we use the identification of $A$ with $C(\Om)$ to get the following translation of each $\al_i \colon A \to A$ in a continuous partially defined map $\vpi_i$ on $\Om$.
If $A^i$ is the ideal $T_i^*T_i A$ in $A$ with unit $T_i^*T_i$, then $\al_i$ defines a unital $*$-homomorphism $\al_i \colon A \to A^i$.
The ideal $A^i$ is the direct limit of $A^i_l = A_l \cap A^i$, and the corresponding projective limit $\Om^i$ is determined by $\vth \colon \Om_{l+1} \to \Om_l$ and the spaces
\[
\Om_l^i = \{[\mu]_l \in \Om_l \mid i \mu \in \La^*\}.
\]
Hence $\al_i \colon A \to A^i$ is a unit preserving map from $A = C(\Om)$ into $A^i := T_i^* T_i A = C(\Om^i)$, and therefore induces a continuous map $\vpi_i : \Om^i \rightarrow \Om$.
If $\La^*_l = \{\mu_1, \dots, \mu_n\}$ then $\vpi_i$ is defined on $\Om^i_{l+1}$ in the following way:
\begin{quoting}
if $[\nu]_{l+1} \in \Om^i_{l+1}$ is such that $\mu_1 i \nu, \dots, \mu_k i \nu \in \La^*$ and $\mu_{k+1} i \nu, \dots, \mu_n i \nu \notin \La^*$, then $\vpi_i([\nu]_{l+1}) = [\nu']_{l}$ is such that $\mu_1 \nu', \dots, \mu_k \nu' \in \La^*$ and $\mu_{k+1} \nu', \dots, \mu_n \nu' \notin \La^*$.
\end{quoting}
Recall here that the universal property of the projective limit implies that this is enough to describe $\vpi_i$.
Indeed let the identification $\pi \colon A \to C(\Om)$ and fix $\mu \in \La^*_l = \{\mu_1, \dots, \mu_n\}$.
Then we get that
\[
\pi(T_\mu^* T_\mu) = \chi_{\{ [w]_l \, \mid \, \mu w \in \La^*\} } \qand \pi\al_i(T_\mu^* T_\mu) = \chi_{\{ [w]_l \, \mid \, \mu i w \in \La^*\} }.
\]
Now for every $[\nu]_{l+1} \in \Om^i_{l+1}$ we get a split into some $\mu_1 i \nu, \dots, \mu_k i \nu \in \La^*$ and the rest $\mu_{k+1} i \nu, \dots, \mu_n i \nu \notin \La^*$.
Then we obtain
\[
\pi\al_i(T_\mu^* T_\mu) ([\nu]_{l+1}) = \begin{cases} 1 & \qif \mu i \nu \in \La^*, \\ 0 & \qotherwise. \end{cases}
\]
Hence $\pi\al_i(T_\mu^* T_\mu) ([\nu]_{l+1}) = 1$ if and only if $\mu \in \{\mu_{1}, \dots, \mu_k\}$.
On the other hand $\vpi_i([\nu]_{l+1}) = [\nu']_l$ with $\mu_1 \nu', \dots, \mu_k \nu' \in \La^*$ and $\mu_{k+1}\nu', \dots, \mu_n\nu' \notin \La^*$, thus
\[
\pi(T_\mu^*T_\mu)\vpi_i([\nu]_{l+1}) = \begin{cases} 1 & \qif \mu \nu' \in \La^*, \\ 0 & \qotherwise. \end{cases}
\]
Thus $\pi(T_\mu^*T_\mu)\vpi_i([\nu]_{l+1}) = 1$ if and only if $\mu \in \{\mu_{1}, \dots, \mu_k\}$.

\begin{definition}\label{D: qd}
Let $\I$ be a monomial ideal of $\bC\sca{x_1, \dots, x_d}$.
We call the $(A,\al) \equiv (A,\al_1, \dots, \al_d)$, or alternatively the $(\Om, \vpi) \equiv (\Om,\vpi_1, \dots, \vpi_d)$, the \emph{quantised dynamics on the allowable words $\La^*$ of $\I$}.
\end{definition}

We chose the above terminology to prevent confusion between the above dynamical system and the dynamical system determined by the shift, when $\I$ gives rise to a subshift.

\begin{remark}
When the ideal $\I$ is of finite type $1$ then there is a simple description of the quantised dynamical system.
In this case $\Om$ is a finite space by Proposition \ref{P: finite A}.
In particular we have that $\Om = \{[\mt]_1, [i_1]_1, \ldots, [i_k]_1\}$, where it may happen or not that $[\mt]_1$ contains some of the variables.
Every set $\Om^i$ is the set of equivalence classes of variables (and the empty word) to which one may append $i$ on the left to obtain a legal word.

On $\Om^i$ the map $\vpi_i$ is defined to be the constant function with value $[i]_1$.
To see this, recall that $\vpi_i \colon \Om^i \rightarrow \Om$ is the dual of the map $\al_i = \ad_{T_i^*} \colon C(\Om) \rightarrow C(\Om^i)$.
But $A$ is generated by $T_j^*T_j$ for $j=1, \ldots, d$.
Hence by identifying $A$ with $C(\Om)$ we see that it is generated by the functions $\chi_{\Om^j}$ for $j=1, \ldots, d$.
Thus we need only check that if we define $\vpi_i$ to be constant $[i]_1$ on $\Om^i$ then $\al_i(\chi_{\Om^j}) = \chi_{\Om^j} \vpi_i$.
Since we are in a subshift of type $1$, we have the identifications
\[
\al_i(\chi_{\Om^j}) = T_i^* T_j^* T_j T_i  =
\begin{cases}
T_i^* T_i  = 1_{\Om^i} & \textrm{ if } ji \in \La^*, \\
0 & \textrm{ if } ji \notin \La^*.
\end{cases}
\]
On the other hand we have $ji \in \La^*$ if and only if $[i]_1 \in \Om^j$, so
\[
\chi_{\Om^j} \vpi_i =
\begin{cases}
1_{\Om^i} & \textrm{ if } ji \in \La^*, \\
0 & \textrm{ if } ji \notin \La^*.
\end{cases}
\]
This can be pictured as a graph on the points $[\mt]_1, [{i_1}]_1, \ldots, [{i_k}]_1$, with an edge labeled $i$ from every $p \in \Om^i$ to $[i]_1$ (including, maybe, loops).
The language can be completely read from this graph: there is an edge labeled $i$ from $[j]_1$ to $[i]_1$ if and only if the word $ij$ is a legal word.
\end{remark}

\begin{remark}\label{R: graph}
Similarly to the previous remark, whenever the augmented system of an ideal is sofic, then one may picture the quantised dynamics as a finite labeled graph.
\end{remark}

Two dynamical systems $(A,\al_{1}, \dots, \al_{d})$ and $(B,\be_{1}, \dots, \be_{d})$ are said to be {\em conjugate} if (after perhaps reordering the maps) there is a $*$-isomorphism $\ga \colon A \rightarrow B$ such that $\ga \al_{i} = \be_{i} \ga$ for all $i$.
Equivalently, (after perhaps reordering the maps) there is a homeomorphism $\ga_s \colon \Om_\J \rightarrow \Om_\I$ mapping $\Om_\J^i$ into $\Om_\I^i$ such that $\ga_s \vpi_{\J,i}|_{\Om_\J^i} = \vpi_{\I,i} \ga_s|_{\Om_\J^i}$ for all $i$.

\begin{example}\label{E: qd conjugacy}
Let $\J$ be the ideal in $\bC \sca{x_1, x_2}$ generated by $\{x_1^2, x_2^2\}$.
This ideal corresponds to the two-sided subshift $\La_\J$ on two symbols $\{1,2\}$ with illegal words $\fF = \{11,22\}$.
This subshift consists of two points
\[
\La_\J = \{(\dots 121 . 2121 \dots), (\dots 212 . 1212 \dots)\} ,
\]
and the shift just permutes these two points.
In this case, the quantised dynamics 
attain the following description.
The space $\Om_\J$ is a three point set $\{0,1,2\}$ where $1$ corresponds to the equivalence class of all words beginning with $1$, $2$ likewise, and $0$ corresponds to the equivalence class of the empty word.
Then $\Om_\J^1 = \{0,2\}$ and $\Om_\J^2 = \{0,1\}$.
The map $\vpi_{\J,1} \colon \Om_\J^1 \rightarrow \Om_\J$ is the constant function $1$, and $\vpi_{\J,2} \colon \Om_\J^2 \rightarrow \Om_\J$ is the constant function $2$.

Now let $\I$ be the ideal generated by $\{x_1x_2, x_2x_1\}$.
We find that $\Om_\I = \{0,1,2\}$ as above.
We also have that $\Om_\I^1 = \{0,1\}$ and $\Om_\I^2 = \{0,2\}$,
that $\vpi_{\I,1} \colon \Om_\I^1 \rightarrow \Om_\I$ is the constant function $1$, and the map $\vpi_{\I,2} \colon \Om_\I^2 \rightarrow \Om_\I$ is the constant function $2$.

The form of the partial dynamical systems is depicted in the following graphs
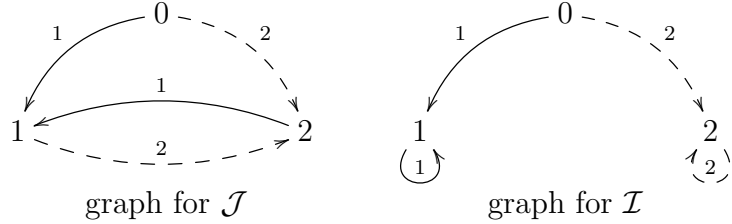
\[
\xymatrix@C=1em@R=4mm{
& 0 \ar@{-->}^{2}@/^1pc/[ddr] \ar@{->}_{1}@/_1pc/[ddl] & & & & 0 \ar@{-->}^{2}@/^1pc/[ddr] \ar@{->}_{1}@/_1pc/[ddl] &  \\
& & & & & & \\
1 \ar@{-->}^{2}@/_1pc/[rr] & & 2 \ar@{->}_{1}@/_1pc/[ll] & & 1 \ar@{->}^{1}@(dl,dr) & & 2 \ar@{-->}_{2}@(dr,dl) \\
& \text{ graph for $\J$} & & & & \text{ graph for $\I$ } & \\
}
\]
where $0 = [\mt]_1$, $1=[1]_1$, $2=[2]_1$, and the solid arrows represent $\vpi_1$ and the broken arrows represent $\vpi_2$.
We have here two very simple dynamical systems: each one is a three point set with a pair of distinct constant valued maps acting on it.
However, they cannot be conjugate, since these maps are {\em partially defined} constant maps.
\end{example}

\begin{theorem}\label{T: qd complete invariant}
The quantised dynamical system is a complete invariant of the monomial ideal: if the quantised dynamical systems of two monomial ideals $\I$ and $\J$ are conjugate, then $\I$ and $\J$ are the same modulo a permutation of the variables. In particular, if the quantised dynamics of two subshifts on the same set of symbols are conjugate, then the subshifts are conjugate.
\end{theorem}
\begin{proof}
Note that conjugacy determines the number of variables.
Letting $\al_\mu = \ad_{T^*_\mu}$, we have that $\mu = \mu_1 \dots \mu_k$ is a forbidden word if and only if $T_\mu^*T_\mu = \al_\mu (I) = \al_{\mu_k} \cdots \al_{\mu_1}(I) = 0$.
Since this determines the forbidden words, this determines the ideal, as well.
\end{proof}

Note that the converse of the second assertion in the proposition fails, since the number of symbols is invariant under conjugacy of the quantised dynamics but not under conjugacy of the shifts.

\begin{example}\label{E: quantised dynamics}
Let $\J$ be as in Example \ref{E: qd conjugacy}.
Let $\K$ be the ideal in the polynomial algebra $\bC \sca{x_1, \ldots, x_4}$ generated by
\[
\begin{pmatrix}
x_1^2 & & x_1 x_3 & \\
& x_2^2 & & x_2 x_4 \\
x_3 x_1 & & x_3^2 & \\
& x_4 x_2 & & x_4^2
\end{pmatrix}.
\]
This corresponds to the two-sided subshift $\La_\K$ on four symbols $\{1,2,3,4\}$ where a word is legal if and only if it contains no two consecutive even symbols and no two consecutive odd symbols.
The subshift $\La_\K$ differs from $\La_\J$, since it has uncountably many points.
On the other hand, the space $\Om_\K$ is also a three point set $\{0, 1, 2\}$, where $0$ is the equivalence class of $\mt$, $1$ is the equivalence class of all words beginning with an odd symbol, and similarly $2$.
Here too we have $\Om_\K^i = \{0,2\}$ for $i$ odd and $\Om_\K^i = \{0,1\}$ for $i$ even; the map $\vpi_{\K,i}$ is identically $1$ (resp. $2$), if $i$ is odd (resp. even).

We see that the quantized dynamics $(\Om_\J,\vpi_{\J,1},\vpi_{\J,2})$ and $(\Om_\K,\vpi_{\K,1}, \ldots, \vpi_{\K,4})$ are given by the same maps on the same space.
This may seem strange in light of Theorem \ref{T: qd complete invariant} but the difference lies on that in the second dynamical system each map is repeated twice.
\end{example}

\begin{remark}
In the previous discussion we enforced $I$ to be included in the C*-algebra $A$.
We could as well restrict our attention to the C*-algebra
\[
A_0 := \ca(T_\mu^* T_\mu \mid \mt \neq \mu \in \La^*) \subseteq \ca(T).
\]
Notice here that $\al_i(A_0) \subseteq A_0$ and that $A_0$ is also unital by Lemma \ref{L: unital}.
Therefore similar conclusions can be derived for $A_0$ which can be identified with a compact Hausdorff space $\Om_0$.
The reason of considering this setting is explained in Remark \ref{R: unital} that will follow.
\end{remark}

\section{The C*-correspondence of a monomial ideal}\label{S: cor}

The C*-algebra $\ca(T)$ is trivially a C*-correspondence over itself.
We want to isolate a C*-correspondence ${}_A E_A$ generated by the $T_i$ inside $\ca(T)$ with the inherited inner product.
There are (at least) three equivalent identifications.
The use of the symbol $A$ is not coincidental.

\subsection{Concrete construction}

First we remark that all $\sca{T_i,T_i} = T_i^* T_i$ must be in the $A$ we are looking for.
Now equation $T_j^* T_j \cdot T_i = T_i \cdot T_{j i}^* T_{j i}$ implies that all $T_i \cdot T_{ji}^* T_{ji}$ must be in the module $E$ we are looking for (with the understanding that $T_{ji}=0$ if $ji \notin \La^*$).
Consequently the inner product
\begin{align*}
(T_i \cdot T_{ji}^* T_{ji})^* T_i \cdot T_{ji}^* T_{ji}
& =
T_{ji}^* T_j \cdot T_i T_i^* T_i \cdot T_{ji}^* T_{ji}
=
T_{ji}^* T_j T_i T_{ji}^* T_{ji}
=
T_{ji}^* T_{ji}
\end{align*}
must be in $A$.
Inductively $E$ contains all ``vectors'' $T_i \cdot T_{\mu}^* T_{\mu}$ and $A$ contains all projections $T_{\mu}^* T_\mu$,  for $\mu \in \La^*$ and $i=1, \dots, d$.

\begin{definition}\label{D: E-A}
Let $\I$ be a monomial ideal of $\bC\sca{x_1, \dots, x_d}$ and let $\La^*$ be the set of allowable words of $\I$.
Then the linear space
\[
E := \ol{\spn}\{ T_i a \mid a \in A, i =1, \dots, d \}
\]
becomes a C*-correspondence over the C*-algebra
\[
A = \ca(T_\mu^* T_\mu \mid \mu \in \La^*)
\]
by defining the operations
\[
a \cdot \xi \cdot b : = a\xi b \,\, \text{and} \,\, \sca{\xi,\eta} = \xi^*\eta, \foral a,b \in A \text{ and } \xi, \eta \in E,
\]
inside $\ca(T)$.
We refer to ${}_A E_A$ as \emph{the C*-correspondence associated with the monomial ideal $\I$}.
We write $1 \equiv 1_A = I_{\F_X} \equiv I$ for the unit of $A$.
\end{definition}

\begin{remark}\label{R: unital}
Alternatively we consider the \emph{minimal} C*-correspondence generated by the $T_i$.
In this case we get the linear space
\[
E_0 := \ol{\spn}\{ T_i, T_i a \mid a \in A_0, i =1, \dots, d\}
\]
which becomes a C*-correspondence over the C*-algebra
\[
A_0: = \ca(T_\mu^* T_\mu \mid \mt \neq \mu \in \La^*)
\]
in the same way.
We include the subscript $0$ in the notation of $E_0$ to make a distinction with $E$ as Hilbert modules, even though $E_0$ and $E$ coincide as linear spaces.
The C*-algebra $A_0$ admits the unit $\sum_{[m] = 1}^{2^d-1} Q_{[m]}$ by Lemma \ref{L: unital}, but it may not contain $I \in \B(\F_X)$.
Consequently the left action on $E_0$ need not be unital, even though the right action is, since we have that
\[
T_i \cdot \sum_{[m] = 1}^{2^d-1} Q_{[m]} = T_iT_i^*T_i \sum_{[m] = 1}^{2^d-1} Q_{[m]} = T_i T_i^* T_i = T_i.
\]
A number of the arguments we will present henceforth hold by substituting $1$ with $\sum_{[m] = 1}^{2^d-1} Q_{[m]}$.
We will keep track of when this happens.
\end{remark}

\subsection{Direct sum construction}\label{Ss: direct sum}

For every $i=1, \dots, d$, let $E_i=A^i :=T_i^* T_i A$ as a vector space and define
\[
\sca{\xi_i, \eta_i} = \xi^*_i \eta_i \,\, \text{and} \,\, \xi_i \cdot a = \xi_i a, \foral a \in A \text{ and } \xi_i, \eta_i \in E_i,
\]
with the operations taking place inside $\ca(T)$.
Then $E_i = \sca{\de_i}$ for $\de_i = T_i^* T_i \in A$ as a Hilbert $A$-module and
\[
\sca{\de_i \cdot a, \de_i \cdot b} = a^* T_i^* T_i b = T_i^*T_i a^* b.
\]
Recall the $*$-homomorphism $\al_i \colon A \to A$ such that $\al_i(a) = T_i^* a T_i$.
Then $E_i$ becomes a C*-correspondence over the commutative $A$ by defining
\[
\phi_i(a) (\de_i \cdot b) = \de_i \cdot \al_i(a)b.
\]
Since $\al_i(a) T_i^* T_i = \al_i(a)$ the computation
\[
\Theta^{E_i}_{\de_i \cdot \al_i(a), \de_i} (\de_i \cdot b) = \de_i \cdot \al_i(a) T_i^*T_i b
\]
gives that $\phi_i(a) = \Theta^{E_i}_{\de_i \cdot \al_i(a), \de_i}$.
Furthermore we get that
\[
\sca{\de_i \cdot T_i^* T_i, \de_i \cdot b} = T_i^* T_i T_i^* T_i b = \sca{\de_i, \de_i \cdot b}
\]
for all $b \in A$.
Therefore we have that $\de_i = \de_i \cdot T_i^* T_i$, thus every $\phi_i$ is unital.

\begin{proposition}\label{P: E dir sum}
Let ${}_A E_A$ be the C*-correspondence of a monomial ideal $\I \lh \bC\sca{x_1, \dots, x_d}$ and let ${}_A (E_i) {}_A$ be as above.
Then $E$ is unitarily equivalent to the sum C*-correspondence of the $E_i$ over $A$.
\end{proposition}

\begin{proof}
The map $U(1 \cdot a_1, \dots, 1 \cdot a_d):= T_1a_1 + \cdots + T_d a_d$ defines the required equivalence.
\end{proof}

\subsection{Topological graph construction}\label{Ss: topological}

Katsura introduced a construction that generalises both graph algebras and dynamical systems \cite{Kat04-top}.
A \emph{topological graph} is a tuple $(\Ups^0, \Ups^1, r, s)$ such that $\Ups^0$ and $\Ups^1$ are locally compact Hausdorff spaces, $r \colon \Ups^1 \to \Ups^0$ is a continuous proper map and $s \colon \Ups^1 \to \Ups^0$ is a local homeomorphism.
For $\xi \in C_0(\Ups^1)$ let the map $\sca{\xi,\xi} \colon \Ups^0 \to [0,\infty]$ given by $\sca{\xi,\xi}(v) = \sum_{y \in s^{-1}(v)} |\xi(y)|^2$.
The linear space
\[
C_d(\Ups^1):=\{ \xi \in C_0(\Ups^1) \mid \sca{\xi,\xi} \in C_0(\Ups^0)\}
\]
becomes a C*-correspondence over $C_0(\Ups^0)$ by defining the inner product
\[
\sca{\xi, \eta}(v) = \sum_{y \in s^{-1}(v)} \ol{\xi(y)} \eta(y)
\]
and the module actions
\[
(\xi \cdot a)(y) = \xi(y) a(s(y)) \qand (a \cdot \xi) (y) = a(r(y)) \xi(y)
\]
for all $a \in C_0(\Ups^0)$ and $\xi, \eta \in C_d(\Ups^1)$.

Suppose that ${}_A E_A$ is the C*-correspondence associated with a monomial ideal.
Then $A$ is identified with $C(\Om)$ for some compact Hausdorff space $\Om$ and every $\al_i$ gives rise to a continuous mapping $\vpi_i \colon \Om^i \to \Om$, where $\Om^i$ is the clopen set induced by the projection $T_i^*T_i$ under the identification $A \simeq C(\Om)$.
Let $\chi_i$ be the characteristic function on $\Om^i$.
Let $\Ups_i^1 = \Om^i$ and define $\Ups^1$ be the disjoint union of the $\Ups_i^1$ for all $i=1, \dots, d$.
Thus a $y \in \Ups^1$ is determined by a tuple $(i,y_i)$ with $y_i \in \Ups_i^1$.
Let $\Ups^0 = \Om$ and define the continuous maps $s,r \colon \Ups^1 \to \Ups^0$ by
\[
s(i,y_i) = y_i \qand r(i,y_i) = \vpi_i(y_i).
\]
Then $s$ is a local homeomorphism.

\begin{proposition}\label{P: E top cor}
Let ${}_A E_A$ be the C*-correspondence of a monomial ideal $\I \lh \bC\sca{x_1, \dots, x_d}$.
Then $E$ is unitarily equivalent to the C*-correspondence of the topological graph $(\Ups^0, \Ups^1, r, s)$ as constructed above.
\end{proposition}

\begin{proof}
Every $\xi \in C_d(\Ups^1)$ can be decomposed into the sum of the $\xi_i = \chi_i \xi$ for $i=1, \dots, d$.
Each $\xi_i$ is a continuous function in $A^i = C(\Om^i) \subseteq C(\Ups^0)$.
Hence the map $U(\xi):= T_1 \xi_1 + \cdots + T_d \xi_d$ defines the required equivalence.
For surjectivity recall that $T_i a = T_i (T_i^*T_i)a$ with $T_i^*T_i a \in C(\Ups^1_i)$.
\end{proof}

\subsection{Analysis of the C*-correspondence}

We proceed to the properties of the C*-cor\-respondence ${}_A E_A$.

\begin{proposition}\label{P: E-A}
Let ${}_A E_A$ be the C*-correspondence of a monomial ideal $\I \lh \bC\sca{x_1, \dots, x_d}$.
Then the left action is non-degenerate and by the compacts.
If, in addition, there are no sinks on the left for $\I$ then ${}_A E_A$ is full.
\end{proposition}

\begin{proof}
By definition the left action $\phi_E$ is unital, hence $E$ is non-degenerate.
The computation
\begin{align*}
\sum_{j=1}^d \Theta^E_{T_j T_{\mu j}^* T_{\mu j}, T_j} (T_i a)
 =
\sum_{j=1}^d T_j T_{\mu j}^* T_{\mu j} T_j^* T_i a
 =
T_i T_{\mu i}^* T_{\mu i} a
 =
T_{\mu}^* T_\mu T_i a,
\end{align*}
shows that $\phi_E(T^*_\mu T_\mu) = \sum_{j=1}^d \Theta^E_{T_j T_{\mu j}^* T_{\mu j}, T_j}$, hence $\phi_E(A) \subseteq \K(E)$.

If $\I$ has no sinks on the left, then the joint projection of the $T_i^*T_i$ equals $I \in \B(\F_X)$ by Lemma \ref{L: unital}.
By using the minimal projections $Q_{[m]}$ in $\ca(T_i^*T_i \mid i=1, \dots,d)$ we can write $a = \sum_{i=1}^d a_i$ with $a_i = T_i^*T_i a_i$ for suitable $a_i \in A$.
Since $a_i = \sca{\de_i, \de_i a_i}$ we get that $a_i \in \sca{E,E}$ for all $i=1, \dots, d$ which completes the proof.
\end{proof}

\begin{remark}\label{R: E0-A0}
In particular we have that $E_0$ of Remark \ref{R: unital} is full and the action is by the compacts.
\end{remark}

It will be essential to identify the kernel of the left action on $E$.
By definition we have that $\ker\phi_E = \{a \in A \mid a T_i = 0 \foral i=1, \dots, d\}$.

\begin{lemma}\label{L: kernel}
If for each $i=1, \dots, d$ there is a $\mu_i \in \La^*$ such that $\mu_i i \notin \La^*$, then $P_\mt = T_{\mu_1}^*T_{\mu_1} \cdots T_{\mu_d}^*T_{\mu_d} $.
In particular we get that $P_\mt \in A$.
\end{lemma}

\begin{proof}
First note that $T_{\mu_1}^*T_{\mu_1} \cdots T_{\mu_n}^*T_{\mu_n} \geq P_\mt$ for any $\mu_1, \dots, \mu_n \in \La^*$.
Suppose that the $\mu_i$ are as in the statement and let $e_\nu \in \La^*$ with $\nu \neq \mt$.
We can write $\nu = i \nu'$ for some $i \in \{1, \dots, d\}$ and $\nu' \in \La^*$.
Then we get that $T_{\mu_i}^* T_{\mu_i} e_{\nu} = 0$ since $\mu_i i \notin \La^*$ by assumption.
Therefore we obtain that $T_{\mu_1}^*T_{\mu_1} \cdots T_{\mu_d}^*T_{\mu_d} e_\nu = 0$.
Since $\nu \neq \mt$ was arbitrary, the proof is complete.
\end{proof}

\begin{proposition}\label{P: kernel}
Let ${}_A E_A$ be the C*-correspondence of a monomial ideal $\I \lh \bC\sca{x_1, \dots, x_d}$.
The following are equivalent:
\begin{enumerate}
\item $P_\mt \in A$;
\item $\ker\phi_E = \sca{P_\mt}$;
\item $\ker\phi_E = \bC \cdot P_\mt$;
\item $\ker\phi_E \neq (0)$;
\item for every $i = 1, \dots, d$ there is a $\mu_i \in \La^*$ such that $\mu_i i \notin \La^*$;
\item for every $i=1, \dots, d$ there is a $\mu_i \in \bF_+^d$ such that $\un{x}^{\mu_i} \notin \I$ and $\un{x}^{\mu_i} x_i \in \I$;
\item $J_E := \ker\phi_E^\perp \cap \phi^{-1}(\K(E)) = \sca{1 - P_\mt} = A (1 - P_\mt)$;
\item $1 \notin J_E$.
\end{enumerate}
If these conditions hold then $\ker\phi_E = \sca{T_{\mu_1}^*T_{\mu_1} \cdots T_{\mu_d}^*T_{\mu_d}}$ for any tuple of words $(\mu_1, \dots, \mu_d)$ such that $\mu_i i \notin \La^*$ for all $i=1, \dots, d$.
\end{proposition}

\begin{proof}
We will be using that $I - P_\mt = \sum_{i=1}^d T_i T_i^*$.
Consequently if $a \in \ker\phi_E$ then we get that $a (I - P_\mt) = 0$.
Also observe that $P_\mt (I - T_\mu^* T_\mu) = 0$ for all $\mu \in \La^*$; hence $P_\mt T_\mu^* T_\mu = P_\mt$ for all $\mu \in \La^*$.

\noindent [(i) $\Rightarrow$ (ii)]: It is immediate that $P_\mt T_i =0$ for all $i=1, \dots, d$.
Thus if $P_\mt \in A$ then $P_\mt \in \ker\phi_E$. Now if $a \in \ker \phi_E$, then $a(I - P_\mt) = 0$, which shows that $a = a P_\mt \in \sca{P_\mt}$.

\noindent [(ii) $\Rightarrow$ (iii)]: This follows from the facts that $P_\mt T_\mu^* T_\mu = P_\mt$ for all $\mu \in \La^*$, and that the projections $T_\mu^* T_\mu$ generate the unital C*-algebra $A$.

\noindent [(iii) $\Rightarrow$ (iv)]: Immediate, since $P_\mt \neq 0$.

\noindent [(iv) $\Rightarrow$ (v)]: Recall that $A$ is an AF algebra by Proposition \ref{P: A is AF}.
Hence $\ker\phi_E = \ol{\cup_l \ker\phi_E \cap A_l}$.
Suppose that $|\La_l^*| = n$ and let the orthogonal minimal projections $\{Q_j\}$ that generate $A_l$ with the understanding that $Q_0 = \prod_{s=1}^n T_{\mu_s}^* T_{\mu_s}$ for some enumeration $\{\mu_1, \dots, \mu_n\}$ of $\La_l^*$.
Fix a non-zero element $a \in A_l \cap \ker\phi_E$; then $a = a P_\mt$.
Since $P_\mt (I - T_\mu^*T_\mu) = 0$ we obtain that $P_\mt Q_j = \de_{0,j} P_\mt$ and therefore
\[
a = a P_\mt = \sum_j \la_j Q_j P_\mt =  \la_0 P_\mt.
\]
As $a \neq 0$, it follows that $P_\mt \in A_k$.
Thus we can write $P_\mt = \sum_j \la_j' Q_j$.
However the projections $Q_j$ are orthogonal and thus for $i \neq 0$ we have
\[
0 = P_\mt Q_i = \la_i' Q_i.
\]
Consequently we have that $P_\mt = \la_0' Q_0$ which implies that $\la_0' = 1$ and so $P_\mt = \prod_{s=1}^n T_{\mu_s}^* T_{\mu_s}$.
As $P_\mt e_i = 0$ for $i \in \{1, \dots, d\}$ there is a $\mu_s \in \La^*_k$ such that $\mu_s i \notin \La^*$.

\noindent [(v) $\Rightarrow$ (i)]: This follows by Lemma \ref{L: kernel}.

\noindent [(v) $\Leftrightarrow$ (vi), (ii) $\Leftrightarrow$ (vii), (iv) $\Leftrightarrow$ (viii)]: These are immediate (recall that the left action is by the compacts, hence $J_E = \ker\phi_E^\perp$).
\end{proof}

\begin{remark}\label{R: kernel}
Proposition \ref{P: kernel} reads the same for the $E_0$ of Remark \ref{R: unital}.
Indeed if $1$ is the identity of $A_0$ then $a(1 - T_\mu^*T_\mu) = a(I - T_\mu^*T_\mu)$ and $a P_\mt = P_\mt a$ for all $a \in A_0$ and $\mu \in \La^*$.
\end{remark}

\begin{remark}\label{R: quotient corre}
There is another C*-correspondence we could relate to a monomial ideal.
Let $q \colon \ca(T) \to \ca(T)/\K(\F_X)$ be the quotient map, and form the C*-correspondence $q(E)$ over the C*-algebra $q(A)$ by using the same $*$-algebraic relations with $E$.
The left action on $q(E)$ is again by the compacts since this is verified by $*$-algebraic relations.
However there are some substantial differences between $q(E)$ and $E$.

First of all $\sum_{i=1}^d q(T_i) q(T_i)^* = I$ and therefore $q(E)$ is always injective.
Furthermore $\ca(T)/\K(\F_X)$ is the Cuntz-Pimsner algebra of $q(E)$ given by the covariant representation $\pi \colon q(a) \mapsto q(a)$ and $t \colon q(T_i) \mapsto q(T_i)$.
This follows from almost tautological algebraic equations, and by that $(\pi, t)$ is injective and admits a gauge action.

Similar remarks hold also for the C*-cor\-respon\-dence $q(E_0)$ over $q(A_0)$ for the C*-correspon\-dence $E_0$ of Remark \ref{R: unital}.
\end{remark}

\section{C*-algebras associated with a monomial ideal}\label{S: C*-alg}

Given a monomial ideal $\I \lh \bC\sca{x_1, \dots, x_d}$ we can form the C*-algebras $\ca(T)$ and $\ca(T) / \K(\F_X)$ related to the subproduct system $X_\I$.
On the other hand the C*-corresponden\-ce ${}_A E_A$ associated with $\I$ initiates automatically two more C*-algebras, namely the Pimsner algebras $\T_E$ and $\O_E$.
In this section we show the connection between all four.

\begin{theorem}\label{T: dichotomy}
Let ${}_A E_A$ be the C*-correspondence of a monomial ideal $\I \lh \bC\sca{x_1, \dots, x_d}$.
Then the following diagram holds:
\vspace{-.5em}
{\small
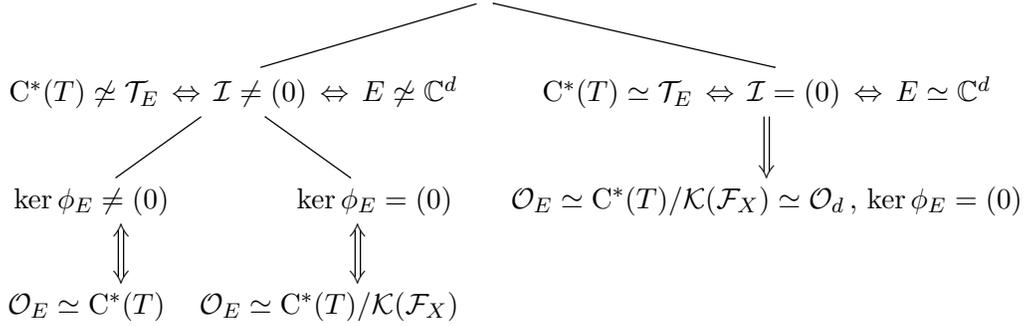
\begin{align*}
\xymatrix@R=2em@C=.3em{ & \ar@{-}[dl]!<-5ex,0ex> \ar@{-}[dr]!<10ex,0ex> & \\
\ca(T) \not\simeq \T_E \, \Leftrightarrow \, \I \neq (0) \, \Leftrightarrow \, E \not\simeq \bC^d \ar@{-}[d]!<-12ex,0ex> \ar@{-}[d]!<12ex,0ex> & &
\ca(T) \simeq \T_E \, \Leftrightarrow \, \I = (0) \, \Leftrightarrow \, E \simeq \bC^d \ar@{=>}[d] \\
\ker\phi_E \neq (0) \ar@{<=>}@<-9ex>[d] \qquad \qquad \, \ker\phi_E = (0) \ar@{<=>}@<10ex>[d] & &
\O_E \simeq \ca(T)/\K(\F_X) \simeq \O_d \, , \, \ker\phi_E = (0)\\
\O_E \simeq \ca(T) \quad \, \O_E \simeq \ca(T)/ \K(\F_X) & &
}
\end{align*}
} with the understanding that all $*$-isomorphisms are canonical.
\end{theorem}

The proof is induced by Proposition \ref{P: C*(T)}, Proposition ֿ\ref{P: dichotomy}, and Corollary \ref{C: C*(T)} that will follow.

\begin{remark}\label{R: dichotomy}
Theorem \ref{T: dichotomy} reads the same for the C*-correspondence $E_0$ of Remark \ref{R: unital} by substituting $\ca(T)$ with $\ca(T_\mu \mid \mt \neq \mu \in \La^*)$.
We will deliberately use both $I$ and the unit $1$ of the C*-algebras throughout the proofs even when these operators coincide.
By doing so, the proofs will read the same for both $E$ and $E_0$.
\end{remark}

\begin{proposition}\label{P: C*(T)}
Let ${}_A E_A$ be the C*-correspondence of a monomial ideal $\I$ in $\bC\sca{x_1, \dots, x_d}$.
Then $\ca(T)$ is the relative Cuntz-Pimsner algebra $\O(J,E)$ for the ideal $J$ generated by $\{1 - T_\mu^* T_\mu \mid \mu \in \La^*\}$.
Moreover $\ca(T)/\K(\F_X)$ is the relative Cuntz-Pimsner algebra $\O(A,E)$.

In particular $J \subseteq J_E \subseteq A$, and there are canonical $*$-epimorphisms
\[
\T_E \to \ca(T) \to \O_E \to \ca(T)/\K(\F_X).
\]
\end{proposition}

\begin{proof}
The mappings $t \colon T_i \mapsto T_i$ and $\pi \colon a \mapsto a$ define a representation $(\pi,t)$ for $E$ such that $\ca(T) = \ca(\pi,t)$.
If we let $u_z e_\nu = z^{|\nu|} e_\nu$ for $z \in \bT$ and $\nu \in \La^*$ then the family $\{\ad_{u_z}\}_{z \in \bT}$ defines a gauge action of $\ca(T)$.
Since $\pi$ is faithful on $A$ we obtain that $\ca(\pi,t) \simeq \O(J,E)$ for the ideal $J = \{ a \in A \mid \pi(a) \in \psi_t(\K(X))\}$ by \cite{Kak13-2}, and that $J \subseteq J_E$ by \cite{Kat04}.

We now show that $J$ is generated by $1-T_\mu^*T_\mu$ for all $\mu \in \La^*$.
Since $e = \sum_{i=1}^d\Theta_{\de_i, \de_i}^E$ is a unit of $\K(E)$ we have that $a \in J$ if and only if $\pi(a)(I - \psi_t(e)) = 0$ by Lemma \ref{L: cov}.
That is $a \in J$ if and only if
\[
a P_\mt = a(I - \sum_{i=1}^d T_i T_i^*) = 0.
\]
Since $P_\mt \leq T_\mu^* T_\mu$ we get that $1 - T_\mu^* T_\mu \in J$ for all $\mu \in \La^*$.
On the other hand recall that $A$ is an AF algebra.
If $a \in J \cap A_l \neq (0)$ then without loss of generality we may assume that $a = \sum \la_j Q_j$ where each minimal projection $Q_j$ of $A_l$ satisfies $\la_j Q_j = Q_j a \in J$ and has the form
\[
\prod_{j=1}^k T_{\mu_j}^* T_{\mu_j} \prod_{j=k+1}^n (I - T_{\mu_j}^* T_{\mu_j}),
\]
for some enumeration $\{\mu_1, \dots, \mu_n\} = \La^*_l$.
Since $a P_\mt =0$ then the products giving the $Q_j$ that appear in the sum for $a$ above must satisfy $k < n$.
Hence $a$ is generated by some $1-T_\mu^*T_\mu$ and the proof of the first part is complete.

For the second part notice that
\begin{align*}
\pi(a) - \psi_t(\phi_E(a))
& =
\pi(a) - \psi_t\left(\sum_{i=1}^d\Theta^E_{\xi_i\al_i(a), \xi_i}\right)
 =
\pi(a) - \sum_{i =1}^ d T_i \pi\al_i(a) T_i^* .
\end{align*}
However $\pi(a) T_i = T_i \pi\al_i(a)$ and therefore
\begin{align*}
\pi(a) - \psi_t(\phi_E(a))
& =
\pi(a) (I - \sum_{i=1}^d T_iT_i^*) =  \pi(a) P_\mt.
\end{align*}
Since $\ca(T)/\K(\F_X)$ is the quotient of $\ca(T)$ by the ideal generated by $P_\mt$ we have that $\ca(T)/\K(\F_X)$ is the relative Cuntz-Pimsner algebra for the ideal generated by $A$; hence $\ca(T)/ \K(\F_X)$ is $\O(A,E)$.
\end{proof}

The canonical $*$-epimorphism $\T_E \to \O_E$ is not faithful because $J_E \neq (0)$.
Indeed if $\phi_E$ is injective then $J_E = A$, whereas if $\phi_E$ is not injective then $J_E \neq (0)$ by Proposition \ref{P: kernel}.
Now we proceed to the examination of the relation between the remaining $\ca(T)$, $\O_E$, and $\ca(T)/\K(\F_X)$.

\begin{proposition}\label{P: dichotomy}
Let ${}_A E_A$ be the C*-correspondence of a monomial ideal $\I \lh \bC\sca{x_1, \dots, x_d}$.
Then we get the following dichotomy:
\begin{enumerate}
\item $\phi_E$ is injective if and only if $\O_E \simeq \ca(T)/\K(\F_X)$ by the canonical $*$-homomorphism;
\item $\phi_E$ is not injective if and only if $\O_E \simeq \ca(T)$ by the canonical $*$-homomorphism.
\end{enumerate}
Hence there are canonical $*$-epimorphisms $\ca(T) \to \O_E \to \ca(T)/\K(\F_X)$ where in any case exactly one of them is faithful.
\end{proposition}

\begin{proof}
Recall that $J_E = \ker\phi_E^\perp$ since the left action is by the compacts.
Proposition \ref{P: C*(T)} implies that $\O_E \simeq \ca(T)/\K(\F_X)$ if and only if $J_E = A$ which settles item (i).
Hence if $\O_E \simeq \ca(T)$ then $\phi_E$ is not injective since $q \colon \ca(T) \to \ca(T)/\K(\F_X)$ cannot be injective.
On the other hand if $\ker\phi_E \neq (0)$ then $P_\mt \in A$ which implies that $J_E = \sca{1 - P_\mt}$.
Since $(1- P_\mt) P_\mt =0$ we have that $1 - P_\mt \in J$ as in the proof of Proposition \ref{P: C*(T)}.
Hence we get that $J = J_E$ which implies that $\O_E \simeq \O(J,E) = \ca(T)$.
\end{proof}

For the next result recall that $\T_d$ denotes the Toeplitz-Cuntz algebra associated with the row isometries of multiplicity $d$.

\begin{corollary}\label{C: C*(T)}
If ${}_A E_A$ is the C*-correspondence of a monomial ideal $\I \lh \bC\sca{x_1, \dots, x_d}$, then $\ca(T) \simeq \T_E$ by the canonical $*$-homomorphism if and only if $\I = (0)$, if and only if $\ca(T) \simeq \T_d$, if and only if $E = \bC^d$.
\end{corollary}

\begin{proof}
By definition $\T_E \simeq \ca(T) = \O(J,E)$ if and only if $J=(0)$.
By Proposition \ref{P: C*(T)} this is equivalent to having that $T_\mu^* T_\mu = I$ for all $\mu \in \La^*$.
In particular $i \nu \in \La^*$ for all $i = 1, \dots, d$ and $\nu \in \La^*$, which implies that $\I = (0)$.
Conversely if $\I = (0)$ then obviously $\ca(T) \simeq \T_d$ and $E = \bC^d$.
\end{proof}

It is an open problem to determine when the Toeplitz algebra $\T(X)$ or when Cuntz-Pimsner algebra $\O(X)$ (as defined in \cite{Vis12}) of a subproduct system $X$ is nuclear or exact.
For the Toeplitz and the Cuntz-Pimsner algebras of a C*-correspondence these problems were solved in \cite{DykShl01,Kat04}.
Using the above results this problem is now resolved in the case where $X$ is the subproduct system coming from a monomial ideal.
The key is the identification of $\ca(T)$ as a relative Cuntz-Pimsner algebra.

\begin{corollary}\label{C: nuc}
Let $X = X_\I$ be a subproduct system associated with a monomial ideal $\I \lh \bC\sca{x_1, \dots, x_d}$.
Then the Toeplitz algebra $\T(X) := \ca(T) = \ca(\A_X)$ and the Cuntz-Pimsner algebra $\O(X):= \ca(T)/\K(\F_X)$ associated with $X$ are both nuclear.
\end{corollary}

\begin{proof}
Since $A$ is commutative then $\T_E$ is nuclear by \cite[Theorem 7.2]{Kat04}.
Therefore so are its quotients $\ca(T)$ and $\ca(T) / \K(\F_X)$.
\end{proof}

\begin{remark}
Nuclearity of $\ca(T)/\K(\F_X)$ has been observed by Matsumoto \cite[Lemma 4.10]{Mat98} in the case of the two-sided subshifts and by using a different line of reasoning.
\end{remark}

When $\I = (0)$ then $\O_E$ is the Cuntz algebra $\O_d$.
We provide a universal property for $\O_E$ when $\I \neq (0)$ as well.
A universal property for $\ca(T)/\K(\F_X)$ when $\I$ is of finite type will be given in Theorem \ref{T: universal}.
Both Theorem \ref{T: CP} and Theorem \ref{T: universal} should be compared with \cite[Theorem 4.9]{Mat97}.
In particular \cite[Theorem 4.9]{Mat97} follows from Theorem \ref{T: CP} below when $\phi_E$ is injective (see also in conjunction with Remark \ref{R: quotient corre}).

\begin{theorem}\label{T: CP}
Let ${}_A E_A$ be the C*-correspondence of a monomial ideal $\I \lh \bC\sca{x_1, \dots, x_d}$ with $\I \neq \{0\}$.
Then $\O_E$ is the universal C*-algebra generated by $d$ partial isometries $s_i$ such that
\begin{enumerate}
\item the mapping $T_\mu^* T_\mu \mapsto s_\mu^* s_\mu$ extends to a $*$-representation of $A$;
\item $s_j^*s_i = \de_{i,j} s_i^* s_i$;
\item $s_\mu^* s_\mu s_i s_i^*= s_i s_i^* s_{\mu}^* s_{\mu}$ for all $\mu \in \La^*$ and $i=1, \dots, d$;
\item $I - s_{\mu_1}^* s_{\mu_1} \dots s_{\mu_d}^* s_{\mu_d} = \sum_{i=1}^d s_i s_i^*$ for any (and hence for every) $d$-tuple of words $\mu_i \in \La^*$ such that $\mu_i i \notin \La^*$, or $I = \sum_{i=1}^d s_is_i^*$ if no such $d$-tuple of words exist.
\end{enumerate}
\end{theorem}

\begin{proof}
If $s_i$ are as above then they define a representation of $E$, which we will denote by $\pi \colon T_\mu^* T_\mu \mapsto s_\mu^* s_\mu$ and $t \colon \de_i \mapsto s_i$.
Indeed items (i) and (ii) show the compatibility with the inner product and item (iii) shows the compatibility with the left action.

We show that item (iv) implies that $(\pi,t)$ is covariant.
Since $1 \in A$ acts as an identity on $E$ we may suppose that $\pi$ is non-degenerate, i.e. $\pi(1) = I$.
If there are no words $\mu_i i \notin \La^*$ for all $i=1, \dots, d$, then $1 \in J_E$ by Proposition \ref{P: kernel}.
Recall that $\phi_E(1) = \sum_{i=1}^d \Theta^E_{\de_i,\de_i}$ therefore we obtain
\[
\pi(1) = I = \sum_{i=1}^d s_i s_i^* = \psi_t(\Theta^E_{\de_i,\de_i}).
\]
Hence $(\pi,t)$ is covariant in this case.
Now if $\mu_i i \notin \La^*$ for some word $\mu_i \in \La^*$ for all $i=1, \dots, d$, then $T_{\mu_1}^* T_{\mu_1} \dots T_{\mu_d}^* T_{\mu_d} \in \ker\phi_E$ and
\[
\ker\phi_E^\perp = \sca{1 - T_{\mu_1}^* T_{\mu_1} \dots T_{\mu_d}^* T_{\mu_d}} \lhd A,
\]
by Proposition \ref{P: kernel}.
Thus we get that
\[
\phi_E(1 - T_{\mu_1}^* T_{\mu_1} \dots T_{\mu_d}^* T_{\mu_d}) = \phi_E(1) = \sum_{i=1}^d \Theta^E_{\de_i,\de_i} \,.
\]
Covariance of $(\pi,t)$ is then given by the computation
\begin{align*}
\pi(1 - T_{\mu_1}^* T_{\mu_1} \dots T_{\mu_d}^* T_{\mu_d})
& =
I - s_{\mu_1}^* s_{\mu_1} \dots s_{\mu_d}^* s_{\mu_d} \\
& =
\sum_{i=1}^d s_i s_i^*
 =
\psi_t(\sum_{i=1}^d \Theta^E_{\de_i,\de_i}).
\end{align*}

Finally to see that $\O_E$ is the universal C*-algebra of the statement it suffices to find a faithful representation of $\O_E$ that satisfies these properties.
If there are no words $\mu_i \in \La^*$ such that $\mu_i i \notin \La^*$ then $\O_E \simeq \ca(T)/\K(F_X)$ and the $q(T_i) \in \ca(T)/ \K(\F_X)$ satisfy the enlisted properties and provide a faithful representation.
If there are such words then $\O_E \simeq \ca(T)$ and the $T_i$ satisfy the properties and provide a faithful representation of $\O_E$.
\end{proof}

In the literature there are several C*-algebras associated with subshifts and consequently with monomial ideals.
These have been introduced as generalisations of the Cuntz-Krieger algebras.
In Section \ref{Ss: subshift compare} we show that they are distinct from the algebras $\O_E$ and $\T_E$ that we introduced above.
The main difference is that in our case we pass to the quotient by the compacts only when the left action is injective.
As we show with the next example, cutting $\ca(T)$ by the compacts unconditionally may erase important information from the original data.

\begin{example}\label{E: ss not sft}
Let $\I$ be the monomial ideal generated by $\{x_1^2, x_1x_2\}$ inside $\bC\sca{x_1,x_2}$.
By definition
\[
T_1 e_\mu = \begin{cases} e_1 & \qif \mu = \mt, \\ 0 & \qotherwise, \end{cases}
\]
and $T_2 e_\mu = e_{2 \mu}$.
Therefore $T_1$ is a compact operator, and $T_2$ is unitarily equivalent to the direct sum of the forward shift with itself (one on the $e_{2^n}$ and one on the $e_{2^{n}1}$ for $n \in \bZ_+$).
As a consequence we have that $q(T_1) =0$ and $q(T_2) = u \oplus u$, where $u$ is the bilateral shift.
Therefore we obtain that $\ca(T)/\K(\F_X) = C(\bT)$.
\end{example}

\section{Tensor algebras associated with a monomial ideal}\label{S: tensor}

We focus on two kinds of nonselfadjoint operator algebras associated with a monomial ideal $\I$.
These are:
\begin{enumerate}
\item the tensor algebra $\T_E^+$ of the C*-correspondence $_{A} E_A$ in the sense of Muhly and Solel \cite{MuhSol98}; and
\item the tensor algebra $\A_X$ of the subproduct system $X_\I$ in the sense of Shalit and Solel \cite{ShaSol09}.
\end{enumerate}
Let us record here the following corollary of our previous analysis.

\begin{corollary}\label{C: tensor}
Let ${}_A E_A$ and $X$ be the C*-correspondence and the subproduct system respectively of a monomial ideal $\I \lhd \bC\sca{x_1, \dots, x_d}$.
Then there is a completely isometrical embedding $\A_X \hookrightarrow \T_E^+$.
\end{corollary}

\begin{proof}
It is immediate since $\ca(T) = \ca(\A_X)$ is a C*-cover of $\T_E^+$ by Proposition \ref{P: C*(T)}.
\end{proof}

\subsection{The tensor algebra $\T_E^+$}\label{Ss: tensor}

First we show that the quantised dynamics $(A,\al)$ of the allowable words determine the representations of the C*-correspondence ${}_A E_A$.

\begin{proposition}\label{P: piVrep}
Let ${}_A E_A$ be the C*-correspondence of a monomial ideal $\I \lh \bC\sca{x_1, \dots, x_d}$.
Then a family $(\pi,\{V_i\}_{i=1}^d)$ defines a representation $(\pi,t)$ of ${}_A E_A$ if and only if
\begin{enumerate}
\item $\pi \colon A \to \B(H)$ is a $*$-representation;
\item $\pi(a) V_i = V_i \pi\al_i(a)$ for all $a \in A$ and $i=1, \dots d$;
\item $V_j^* V_i = \pi(T_j^* T_i)$ for all $j,i = 1, \dots d$.
\end{enumerate}
\end{proposition}

\begin{proof}
Write $E$ as the direct sum of $E_i = \sca{\de_i}$ by Proposition \ref{P: E dir sum}.
If $(\pi,t)$ is a representation of $E$ then let $V_i = t(\de_i)$.
Then we obtain that
\begin{align*}
\pi(a) V_i = \pi(a) t(\de_i) = t(\phi_E(a) \de_i) = t( \de_i \al_i(a)) = t(\de_i) \pi\al_i(a) = V_i \pi\al_i(a),
\end{align*}
and that
\[
V_j^* V_i = t(\de_j)^* t(\de_i) = \pi(\sca{\de_j, \de_i}) = \pi(T_j^* T_i),
\]
for all $a \in A$ and $j,i = 1, \dots, d$.
Conversely, if a family $(\pi,\{V_i\}_{i=1}^d)$ satisfies the properties of the statement then $(\pi, t)$ defines a representation for $E$ with $t(\sum_{i=1}^d \de_i a_i) := \sum_{i=1}^d V_i \pi(a_i)$.
\end{proof}

The quantised dynamics determine the completely contractive representations of $\T_E^+$ as well.
Since $E$ is a non-degenerate C*-correspondence the following result can be obtained by using the {\em completely contractive covariant representations} of $E$ in the sense of Muhly and Solel \cite{MuhSol98}.
We sketch an alternative proof that settles also the case of $E_0$.

\begin{proposition}
Let ${}_A E_A$ be the C*-correspondence of a monomial ideal $\I \lh \bC\sca{x_1, \dots, x_d}$.
Then a family $(\pi, \{V_i\}_{i=1}^d)$ defines a completely contractive representation of $\T_E^+$ if and only if
\begin{enumerate}
\item $\pi \colon A \to \B(H)$ is a $*$-representation;
\item $\pi(a) V_i = V_i \pi\al_i(a)$ for all $a \in A$ and $i=1, \dots d$;
\item $[V_j^* V_i] \leq [\pi(T_j^* T_i)]$.
\end{enumerate}
A similar conclusion holds for $\T_{E_0}^+$ and representations $\pi \colon A_0 \to \B(H)$.
\end{proposition}

\begin{proof}
If $(\pi,t)$ defines a completely contractive representation $\pi \times t$ of $\T_E^+$ then let
\[
G = \begin{bmatrix} T_1 & \cdots & T_d \\ 0 & \cdots & 0 \\ \vdots & \cdots & \vdots \\ 0 & \cdots & 0 \end{bmatrix}
\]
and compute
\begin{align*}
[V_j^*V_i]
& =
t \otimes \id_n (G)^* \cdot t \otimes \id_n(G) \\
& =
(\pi \times t) \otimes \id_n (G)^* \cdot (\pi \times t) \otimes \id_n(G) \\
& \leq
\left( (\pi \times t) \otimes \id_n \right) (G^*G) = [\pi(T_j^*T_i)],
\end{align*}
for $V_i = t(\de_i)$ with $i=1, \dots, d$.
For the converse we write
\[
[V_j^* V_i] = [V_1, \dots, V_d]^* [V_1, \dots, V_d].
\]
Hence $[V_1, \dots, V_d]$ is a row contraction since $\pi(T_j^*T_i) \leq \de_{i,j} I$.
By using the defect operator
\[
D = \left([\pi(T_j^*T_i)] - [V_j^*V_i]\right)^{1/2}
\]
the proof then follows as in \cite[Theorem 2.1]{DavKat11}.
The case of $E_0$ follows in the same way.
\end{proof}

Suppose that $A$ acts faithfully on a Hilbert space $H$ and let the Hilbert space $K = \oplus_{\mu \in \La^*} H$.
On $K$ we define the $*$-representation
\[
\pi_0(a) = \diag\{ \al_\nu(a) \mid \nu \in \La^*\},
\]
where $\al_\nu = \ad_{T_\nu^*}$.
We remark here that $\al \colon \bF_+^d \to \End(A)$ is an anti-homomorphism because $\al_\nu \al_\mu = \al_{\mu \nu}$.
Moreover let $L_i$ be the regular shifts on $\bF_d^+$, i.e. $L_i e_\nu = e_{i \nu}$.
Let $P = \oplus_{\mu \in \La^*} T_\mu^*T_\mu$ and define
\[
\pi(\cdot) = P \pi_0(\cdot) P \qand V_i = P L_i P
\]
for all $i=1, \dots, d$.
Since $A$ is commutative we get that $P$ commutes with $\pi_0$.
Moreover we have that $P L_i P = P L_i$, hence $V_\mu = P L_\mu$ for all $\mu \in \La^*$.
The family $(\pi, \{V_i\}_{i=1}^d)$ satisfies the conditions of Proposition \ref{P: piVrep}, hence it defines a representation $(\pi,t)$ of $E$.
In addition $(\pi,t)$ admits the gauge action given by $\be_z : = \ad_{u_z}$ with $u_z(\xi \otimes e_\mu) = z^{|\mu|} \xi \otimes e_\mu$.
In particular we get
\[
\sum_{i=1}^d V_iV_i^* (\xi \otimes e_\mu) = \begin{cases} \xi \otimes e_\mu & \qif \mu \neq \mt, \\ 0 & \qotherwise, \end{cases}
\]
and $\sum_{i=1}^d V_iV_i^* = \psi_t(e)$ for the identity $e = \sum_{i=1}^d \Theta^E_{\de_i, \de_i}$ of $\K(E)$.
In view of Lemma \ref{L: cov} we have that $\pi(a) \in \psi_t(\K(E))$ if and only if
\[
0 = \pi(a)(I - \psi_t(e)) = \pi(a) P_\mt = a.
\]
Hence by the gauge invariant uniqueness theorem \cite{Kat04} the pair $(\pi,t)$ defines a faithful representation of $\T_E$.
Therefore the restriction of $\pi \times t$ to the tensor algebra $\T_E^+$ is a completely isometric homomorphism.

We will use that $V_\mu^* V_\mu = \pi(T_\mu^* T_\mu)$ for all $\mu \in \La^*$.
This follows by induction and the computation
\begin{align*}
V_i^* V_\mu^* V_\mu V_i
& =
V_i^* \pi(T_\mu^*T_\mu) V_i
 =
V_i^* V_i \pi\al_i(T_\mu^* T_\mu) \\
& =
\pi(T_i^* T_i T_i^* T_\mu^* T_\mu T_i)
 =
\pi(T_i^* T_\mu^* T_\mu T_i).
\end{align*}
 Since $V_\mu^* V_\mu = \pi(T_\mu^* T_\mu)$ is a projection we obtain
\[
V_\mu \pi(a) = V_\mu V_\mu^* V_\mu \pi(a) = V_\mu \pi(T_\mu^* T_\mu a).
\]
For $\mu \in \La^*$, let $E_\mu \colon \T_E^+ \to A$ be the compression to the $(\mt,\mu)$-entry.
Then $E_\mu$ is a contractive map such that
\[
E_\mu(\sum_{\nu \in \La^*_l} V_\nu \pi(a_\nu)) = T_\mu^* T_\mu a_\mu.
\]
Therefore every element $f \in \T_E^+$ can be written as
\[
f = f' + \sum_{\mu \in \La^*_l} V_\mu \pi(a_\mu)
 \]
with $E_\mu(f') = 0$ for every $\mu \in \La^*_l$.
Moreover if $f = g' + \sum_{\mu \in \La^*_l} V_\mu \pi(b_\mu)$ for some $b_\mu \in A$ with $E_\mu(g') = 0$ for every $\mu \in \La^*_l$, then
\[
T_\mu^* T_\mu (b_\mu - a_\mu) =0 \foral \mu \in \La^*_l,
\]
i.e. the coefficients $a_\mu$ are unique modulo $T_\mu^*T_\mu$, for every $\mu \in \La^*_l$.
Consequently we obtain that
\[
V_\mu \pi(b_\mu) = V_\mu \pi(T_\mu^*T_\mu b_\mu) = V_\mu \pi(T_\mu^*T_\mu a_\mu) = V_\mu \pi(a_\mu),
\]
therefore
\begin{align*}
g'
& =
f - \sum_{\mu \in \La^*_l} V_\mu \pi(b_\mu)
  =
f - \sum_{\mu \in \La^*_l} V_\mu \pi(a_\mu) = f',
\end{align*}
which shows in what sense the decomposition is unique.

The universal property of $\O_E$ obtained in Theorem \ref{T: CP} enables us to provide the following result for the tensor algebra $\T_E^+$.

\begin{theorem}\label{T: hyper T_E^+}
Let ${}_A E_A$ be the C*-correspondence of a monomial ideal $\I \lh \bC\sca{x_1, \dots, x_d}$.
Then the tensor algebra $\T_E^+$ is hyperrigid.
\end{theorem}

\begin{proof}
By \cite{KatKri06} the C*-envelope of $\T_E^+$ is $\O_E$.
Suppose that $\O_E$ is generated by a family $\{s_i\}_{i=1}^d$ as in Theorem \ref{T: CP} and let $\Phi \colon \O_E \to \B(H)$ be a unital $*$-representation with $\Phi(s_i) = S_i$.
Let $\rho \colon \T_E^+ \to \B(K)$ be a maximal dilation of $\Phi|_{\T_E^+}$ and write
\[
\rho(s_i) = \begin{bmatrix} S_i & X_i \\ Y_i & Z_i \end{bmatrix}
\]
with respect to the decomposition of $K = H \oplus H^\perp$.
We will denote by the same symbol the unique extension of $\rho$ to a representation of $\O_E$.
Then the $\rho(s_i)$ satisfy the relations of Theorem \ref{T: CP}.
We aim to show that $\rho$ is a trivial dilation.
It suffices to show this for the generators $a \in A$ and $s_i$ of $\T_E^+$.

First $\Phi|_A$ is a $*$-representation, hence it is a direct summand of $\rho|_A$.
Since $\Phi(T_i^* T_i) = \Phi(s_i)^* \Phi(s_i) = S_i^* S_i$ then by equating the $(1,1)$-entries of the equation $\rho(T_i^* T_i) = \rho(s_i)^* \rho(s_i)$ we get that $Y_i = 0$ for all $i=1, \dots, d$.
If $\ker\phi_E \neq (0)$ then there are words $\mu_i \in \La^*$ such that $\mu_i i \notin \La^*$ and $I - s_{\mu_1}^* s_{\mu_1} \dots s_{\mu_d}^* s_{\mu_d} = \sum_{i=1}^d s_is_i^*$.
By applying $\rho$ we obtain that
\begin{align*}
I - \begin{bmatrix} \Phi(s_{\mu_1}^* s_{\mu_1} \dots s_{\mu_d}^* s_{\mu_d}) & 0 \\ 0 & \ast \end{bmatrix}
& =
\sum_{i=1}^n \begin{bmatrix} S_iS_i^* + X_iX_i^* & \ast \\ \ast & \ast \end{bmatrix},
\end{align*}
and by equating the $(1,1)$-entries we obtain that $X_i=0$ for all $i=1, \dots, d$.
Hence $\rho$ is a trivial dilation of $\Phi|_{\T_E^+}$.

If $\ker\phi_E = (0)$ then $I = \sum_{i=1}^d s_i s_i^*$ and a same computation completes the proof.
\end{proof}

\subsection{The tensor algebra $\A_X$}

Let us continue with the analysis of the nonselfadjoint operator algebra $\A_X$ generated by the $T_\mu$ for $\mu \in \La^*$.

\begin{proposition}\label{P: hyper A_X}
Let $X$ be the subproduct system of a monomial ideal $\I \lh \bC\sca{x_1, \dots, x_d}$ of finite type $k$ and let $q \colon \ca(T) \to \ca(T)/ \K(\F_X)$.
Then $q(\A_X)$ is hyperrigid in $\ca(T) / \K(\F_X)$, hence $\cenv(q(\A_X)) = \ca(T)/\K(\F_X)$.
\end{proposition}

\begin{proof}
Let $\Phi \colon \ca(T)/ \K(\F_X) \to \B(H)$ be a unital $*$-representation and let us write $\Phi(q(T_i)) = W_i$ for all $i=1, \dots, d$.
Let $\rho$ be a maximal dilation of $\Phi|_{q(\A_X)}$ and let us denote by the same symbol the extension of $\rho$ to $\ca(T)/ \K(\F_X)$.
Then we get
\[
\rho(q(T_i)) = \begin{bmatrix} W_i & X_i \\ Y_i & Z_i \end{bmatrix}.
\]
Applying $\Phi$ and $\rho$ to the equation $I = \sum_{i=1}^d q(T_i) q(T_i)^*$ and by equating the $(1,1)$-entries we get that $X_i = 0$ for all $i=1, \dots, d$.

We will require one more equation from \cite[Section 12]{ShaSol09}; that is
\begin{align*}\label{eq:star_idZ}
q(T_i)^* q(T_i) = \sum_{\mu \in E^k_i} q(T_\mu) q(T_\mu)^* \text{ for all } i =1, \dots, d,
\end{align*}
where $E^k_i : = \{\mu \in \La^*_k \mid i \mu \in \La^*\}$.
For a short proof, recall that the $T_i^*T_i$ is the projection on the space
\[
G:= \ol{\spn}\{e_\mu \mid \mu \in \La^* \text{ such that } i\mu \in \La^*\}.
\]
Define the subspace
\[
G':=\ol{\spn}\{e_{\mu} \mid  \mu \in \La^* \text{ such that } i\mu \in \La^* \text{ and } |\mu| \geq k\}
\]
of finite co-dimension in $G$; then
\[
G'=\ol{\spn}\{e_{\mu\nu} \mid \mu\nu \in \La^* \text{ and } \mu \in E^k_i\}.
\]
Indeed, if $w = \mu\nu$ with $i \mu \in \La^*$ then $iw \in \La^*$, otherwise we would have a forbidden word of length greater than $k+1$ that does not contain a forbidden word.
This expression of $G'$ shows that it is the range of the projection of $\sum_{\mu \in E^k_i} T_\mu T_\mu^*$.

Applying $\Phi$ and $\rho$ to the above equation, and by restricting to the $(1,1)$-entries we get that $Y_i=0$ for all $i=1, \dots, d$.
Hence $\rho(q(T_i))$ is a trivial dilation of $W_i$ for all the generators of $\A_X$.
\end{proof}

\begin{theorem}\label{T: cenv A_X}
Let $X$ be the subproduct system of a monomial ideal $\I \lh \bC\sca{x_1, \dots, x_d}$ of finite type and let $q \colon \ca(T) \to \ca(T)/\K(\F_X)$.
Then items $(1)$, $(2)$, and $(3)$ of the following diagram hold:
\vspace{-.5em} {\small
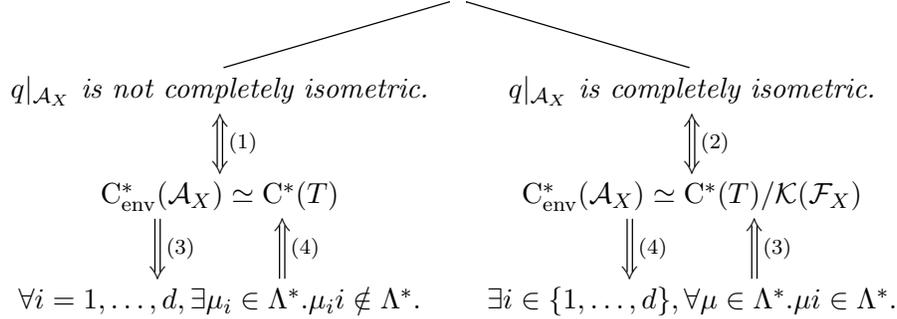
\[
\xymatrix@R=2em@C=.4em{
& \ar@{-}[dl]!<-6ex,0ex> \ar@{-}[dr]!<6ex,0ex> & \\
q|_{\A_X} \text{ is not completely isometric.} \ar@{<=>}[d]^{(1)} & & q|_{\A_X} \text{ is completely isometric.} \ar@{<=>}[d]^{(2)} \\
\cenv(\A_X) \simeq \ca(T) \ar@{=>}@<-5ex>[d]^{(3)} \ar@{<=}@<5ex>[d]^{(4)} & & \cenv(\A_X) \simeq \ca(T)/ \K(\F_X) \ar@{=>}@<-5ex>[d]^{(4)} \ar@{<=}@<5ex>[d]^{(3)} \\
\forall i =1, \dots, d, \exists \mu_i \in \La^*. \mu_i i \notin \La^*. & & \exists i \in \{1, \dots, d\} ,\forall \mu \in \La^* . \mu i \in \La^*.
}
\]
}If in addition the $\mu_i$ can be chosen to have the same length then item $(4)$ also holds.
In particular item (4) holds when the ideal $\I$ has no sinks on the left.
\end{theorem}

\begin{proof}
Referring to item (1), if $q|_{\A_X}$ is not completely isometric then we obtain $\cenv(\A_X) \simeq \ca(T)$ by \cite[Theorem 2.1.1]{Arv72}.
Conversely if $\cenv(\A_X) \simeq \ca(T)$ then $q|_{\A_X}$ cannot be completely isometric since $\K(\F_X) \neq (0)$.

Referring to item (2), if $q|_{\A_X}$ is completely isometric then Proposition \ref{P: hyper A_X} implies that $\cenv(\A_X) \simeq \ca(T) / \K(\F_X)$.
The converse is trivial.

Referring to items (3) and (4), if there is an $i \in \{1, \dots, d\}$ such that $\mu i \in \La^*$ for all $\mu \in \La^*$ then $\ker\phi_E = (0)$ hence $\cenv(\T_E^+) = \O_E = \ca(T)/\K(\F_X)$ by Theorem \ref{T: dichotomy}.
This shows that $q|_{\T_E^+}$ is completely isometric.
Since $\A_X \subseteq \T_E^+$ by Corollary \ref{C: tensor}, then $q|_{\A_X}$ is completely isometric, as well.

What remains to show is that if there are words $\mu_1, \ldots, \mu_d \in \La^*$ of the same length such that $\mu_i i \notin \La^*$ for all $i=1, \ldots, d$, then $q|_{\A_X}$ is not completely isometric.
Initially the words $\mu_i$ may not be distinct.
However we can restrict to a subset and assume that we have $n$ distinct words $\mu_k$ such that for any $i=1, \dots, d$ there exists a $k_i \in \{1, \dots, n\}$ with $\mu_{k_i} i \notin \La^*$.
Let the element $T = T_{\mu_1} + \dots + T_{\mu_n}$ in $\A_X$.
Then $T_{\mu_i}^* T_{\mu_j} = 0$ since the distinct words have the same length so that
\[
\nor{T}^2 = \nor{T^*T} = \nor{T_{\mu_1}^* T_{\mu_1} + \dots + T_{\mu_n}^* T_{\mu_n}}.
\]
Recall that by Lemma \ref{L: kernel} we have $T_{\mu_1}^* T_{\mu_1} \cdots T_{\mu_n}^* T_{\mu_n} = P_\mt$.
By splitting every projection into a sum of minimal subprojections with respect to the family $\{ T_{\mu_i}^* T_{\mu_i} \mid i=1, \dots, n\}$ we find that the common subprojection $P_\mt = T_{\mu_1}^* T_{\mu_1} \dots T_{\mu_n}^* T_{\mu_n}$ appears $n$ times and consequently $\nor{T}^2 = n$.
The same decomposition in the quotient shows that every subprojection of $q(T_{\mu_k}^* T_{\mu_k})$ appears at most $n-1$ times since their only common subprojection $q(T_{\mu_1}^* T_{\mu_1} \dots T_{\mu_n}^* T_{\mu_n})$ is equal to $q(P_\mt) = 0$.
Thus
\[
\nor{q(T)}^2 = \nor{q(T_{\mu_1}^* T_{\mu_1}) + \dots + q(T_{\mu_n}^* T_{\mu_n})} \leq n-1.
\]
Therefore $q|_{\A_X}$ is not completely isometric in this case.

Finally we mention that if there are no sinks on the left for $\I$ then we can add more letters to the left of each $\mu_i$ and arrange them to have the same length.
Note that the new elements will still satisfy the same condition.
\end{proof}

\begin{remark}
Let us compare the C*-correspondences $E$ and $q(E)$.
In Remark \ref{R: quotient corre} we noted that $\ca(T)/\K(\F_X)$ is the Cuntz-Pimsner algebra of $q(E)$.
However $\ca(T)$ may not be in general a relative Cuntz-Pimsner algebra of $q(E)$ in the canonical way.
Indeed, if it were then $\T_E^+$ should be completely isometric to $\T_{q(E)}^+ =\ol{ q(\T_E^+)}$, which we showed that does not hold in general.
\end{remark}

\begin{remark}
We wish to record here a different proof of item (3) of Theorem \ref{T: cenv A_X} that does not use the facts that $\A_X \subseteq \T_E^+$ and $q|_{\T_E^+}$ is completely isometric.
It is an adaptation of \cite[Lemma 3.2]{KatKri06} and it can be generalised to other subproduct systems where a convenient relative Cuntz-Pimsner algebra may not be present.

Let us content ourselves in showing that the quotient map is isometric on the normed algebra $\alg\{1, T_1, \ldots, T_d\}$.
Let $p(\un{T})$ be a polynomial in $T_1, \ldots, T_d$, and let $x$ be a unit vector in the linear span of $X(m)$ for $m \geq 0$, such that $\nor{p(\un{T})x} > \nor{p(\un{T})} -\eps$.
We have that $\mu i^n $ is an allowed word for all $n$ and all $\mu \in \La^*$, thus $p_{m + n}(y \otimes e_{i^n}) = p_m(y) \otimes e_{i^n}$ for all $y \in X$.
Let the contractive operator $R_{i} \in \B(\F_X)$ defined by
\[
R_i y = p_{m+1}(y \otimes e_i) = p_m(y) \otimes e_i \qfor y \in X(m).
\]
Therefore $R_i$ commutes with $p(\un{T})$ and $\nor{R_i^n p(\un{T}) x} \geq \nor{p(\un{T})} - \eps$.
For every $K \in \K(\F_X)$ we then have
\begin{align*}
\nor{p(\un{T}) + K }
& \geq
\nor{(p(\un{T})+K) R_{i}^n x}
 >
\nor{p(\un{T})} - \eps - \nor{K R_i^n x}.
\end{align*}
Letting $n \rightarrow \infty$ we obtain $\nor{q(p(\un{T}))} \geq \nor{p(\un{T})} - \eps$ for all $\eps$, as required.
\end{remark}

\begin{remark}
There is a similarity of Theorem \ref{T: cenv A_X} to what is known for commutative subproduct systems.
By \cite{KenSha14} if $X$ is a commutative subproduct system then $\cenv(q(\A_X)) = \O_X$ if and only if the shift $[T_1, \dots, T_d]$ is essentially normal.
By \cite{Arv05} monomial ideals in commuting variables give rise to essentially normal shifts.
Thus we obtain that $\cenv(q(\A_X)) = \O_X$ for monomial ideals in the commutative case as well.
Recall that in the commutative case all ideals are finitely generated.
\end{remark}

\begin{question}\label{Q: hyperrigidity}
The only part in Theorem \ref{T: cenv A_X} where we use that $\I$ is of finite type is when showing that if $q|_{\A_X}$ is isometric then $\cenv(\A_X) = \ca(T)/ \K(\F_X)$.
Finiteness is needed to apply Proposition \ref{P: hyper A_X}.
We ask if Proposition \ref{P: hyper A_X} holds in general, or even more if equivalence (2) of Theorem \ref{T: cenv A_X} holds in general.
\end{question}

We now give a universal property of $\ca(T)/\K(\F_X)$ for the case where $\I$ is of finite type
(compare the following with Theorem \ref{T: CP}).
This is an illustration of the utility of the dilation techniques used for hyperrigidity.

Also we take this opportunity to fill a gap in the proof of \cite[Theorem 12.7]{ShaSol09}.
In particular, the equation $\ca(T) = \ol{\spn} \A_X \A_X^*$ used in the proof of \cite[Theorem 12.7]{ShaSol09} was not justified therein.
We do not know whether this claim is always true.

\begin{theorem}\label{T: universal}
Let $\I$ be a monomial ideal of finite type $k$. Then the algebra $\ca(T)/\K(\F_X)$ is the universal C*-algebra generated by a row contraction $\un{s} = [s_1, \ldots, s_d]$ such that
\begin{enumerate}
\item $I = \sum_{i=1}^d s_is_i^*$;
\item $p(\un{s}) = 0$ for all $p \in \I$;
\item $s_i^*s_i = \sum_{\mu \in E^k_i} s_\mu s_\mu^*$ where $E_i^k = \{\mu \in \La_k^* \mid i \mu \in \La^*\}$, for all $i =1, \ldots, d$.
\end{enumerate}
\end{theorem}

\begin{proof}
Denote $\un{w} = [q(T_1), \dots, q(T_d)]$.
The identities (i) and (ii) then hold for $\un{w}$.
Identity (iii) for $\un{w}$ is provided in the course of the proof of Proposition \ref{P: hyper A_X}.
It remains to show that whenever $\un{s}$ is a $d$-tuple as in the statement of the theorem, then there exists a $*$-homomorphism $\pi \colon \ca(T)/\K(\F_X) \rightarrow \ca(s)$ with $\pi(w_i)  = s_i$.

To this end, denote $\E_X := \ol{\spn} \A_X \A_X^*$.
By \cite[Theorem 8.2]{ShaSol09}, there is a unital completely positive map $\Psi \colon \E_X \rightarrow B(H)$ satisfying
\[
\Psi(T_\mu T_\nu^*) = s_\mu s_\nu^* \qand \Psi(ab) = \Psi(a) \Psi(b)
\]
for all $\mu , \nu \in \bF_+^d$ and for all $a \in \A_X, b \in \E_X$.
Let $\Psi$ denote also a unital completely positive extension of $\Psi$ to $\ca(T)$.
Using property (i) for $\un{s}$ and Lemma \ref{L: list Ti} items (vi) and (vii) for $T$, we find that $\K(\F_X) \subseteq \ker \Psi$.
Indeed we directly compute
\begin{align*}
\Psi(T_\mu P_\mt T_\nu^*)
& =
\Psi(T_\mu T_\nu^*) - \sum_{i=1}^d \Psi(T_\mu T_i T_i^* T_\nu^*) \\
& =
s_\mu s_\nu^* - \sum_{i=1}^d s_{\mu i} s_{\nu i}^*
 =
s_\mu (I - \sum_{i=1}^d s_is_i^*) s_\nu^* = 0,
\end{align*}
for all rank one operators $T_\mu P_\mt T_\nu^*$ in $\B(\F_X)$.
Thus $\Psi$ induces a unital completely positive map $\Phi \colon \ca(T) / \K(\F_X) \rightarrow B(H)$ for which $\Phi(q(a)) = \Psi(a)$, for all $a \in \ca(T)$.
In particular, we have $\Phi(w_i) = s_i$, as well as
\[
\sum_{i=1}^d \Phi(w_i w_i^*) = I = \sum_{i=1}^d s_i s_i^*,
\]
and
\[
\Phi(w_i^* w_i) = \sum_{\mu \in E^k_i} \Phi (w_\mu w_\mu^*) = \sum_{\mu \in E^k_i} s_\mu s_\mu^* = s_i^* s_i,
\]
for all $i=1, \ldots, d$.
Let $\rho \colon \ca(T)/\K(\F_X) \to \B(K)$ be the Stinespring dilation of $\Phi$.
We can now use the items (i), (ii), and (iii) as in the proof of Proposition \ref{P: hyper A_X} to obtain that $\rho|_{q(\A_X)}$ is a trivial dilation of $\Phi|_{q(\A_X)}$.
Since $q(\A_X)$ generates $\ca(T)/\K(\F_X)$ we get that $\Phi$ is a $*$-representation.
\end{proof}

\begin{corollary}
Let $\I$ be a monomial ideal of finite type $k$, and let $\un{s} = [s_1, \ldots, s_d]$ be a row contraction satisfying items (i)--(iii) of Theorem \ref{T: universal}.
The each $s_i$ is a partial isometry.
\end{corollary}

\section{Encoding via C*-correspondences}\label{S: class cor}

We turn our attention to the classification of our data by using C*-correspon\-dences.
Our aim is to show that C*-correspondences and their tensor algebras form a complete invariant for monomial ideals up to local conjugacy of the induced quantised dynamics.
To allow comparisons we fix the following notation.

Let $\J$ be a monomial ideal of $\bC\sca{y_1, \dots, y_{d'}}$, and let ${}_B F_B$ be the C*-correspon\-dence associated with $\J$.
We write $M^*$ for the allowable words related to $\J$, and $R_w^* R_w$ with $w \in M^*$ for the generators of $B$.
Moreover suppose that $\ze_i = R_i^*R_i \in B$ generate each direct summand $F_i$ of $F$ as in Section \ref{Ss: direct sum}.
For $[n]_2 \equiv [n] = [n_1 n_2 \dots n_{d'}] \in \{0, \dots, 2^{d'}-1\}$ we write
\[
P_{[n]} : = \prod_{n_i \in \supp [n]} R_{i}^* R_{i} \cdot \prod_{n_i=0} (I-R_{i}^*R_{i}),
\]
for the minimal projections in $\ca(I, R_i^*R_i \mid i=1, \dots, d')$.
The quantised dynamics associated with $\J$ are denoted by $(B,\be)$ where $\be_i(b) = R_i^* b R_i$ for $i=1, \dots, d'$.
In analogy to $E_0$ of Remark \ref{R: E0-A0} we write $B_0$ and $F_0$ for $\J$.

\subsection{Unitary equivalence}

We collect a few facts about module maps and their matrix representations.
Since both $E_A$ and $F_B$ are finitely generated, an adjointable map $(\ga,U)$ is characterised by the images $U(\de_j)$ for $j=1, \dots, d$.
If $U(\de_j) = \sum_{i=1}^{d'} \ze_i b_{ij}$ with  $b_{ij} \in B$, then we can represent $U$ as the matrix
\[
[ \Theta_{\ze_i b_{ij}, \de_j} ] \colon \begin{bmatrix} E_1 \\ \vdots \\ E_d \end{bmatrix} \to \begin{bmatrix} F_1 \\ \vdots \\ F_{d'} \end{bmatrix}: \de_j \mapsto \sum_{i=1}^{d'} \ze_i b_{ij},
\]
which shows that $\L(E,F) = \K(E,F)$.
Since $\Theta_{\ze_i b_{ij}, \de_j} = \Theta_{\ze_i, \de_j \ga^{-1}(b_{ij}^*)}$ as well as $\ze_i = \ze_i \cdot R_i^*R_i$ and $\de_j = \de_j \cdot T_j^*T_j$, we obtain that the elements $b_{ij} \in A$ are unique up to the equations
\begin{align*}
R_i^* R_i b_{ij} = b_{ij} = b_{ij} \ga(T_j^* T_j)  \foral i=1, \dots,d' \text{ and } j =1, \dots, d.
\end{align*}
Conversely if there is a $*$-isomorphism $\ga \colon A \to B$ then a matrix $[b_{ij}] \in M_{d' \times d}(B)$ defines a $(\ga, U) \in \L(E,F)$ by
\[
U(\de_j) = \sum_{i=1}^{d'} \ze_i b_{ij} \foral j= 1, \dots, d.
\]
Indeed we can write $U = \sum_{i=1}^{d'} \sum_{j=1}^d \Theta_{\ze_i b_{ij}, \de_i}$.
It is immediate that $(\Theta_{\ze_i b_{ij}, \de_i})^* = \Theta_{\de_j \ga^{-1}(b_{ij})^*, \ze_i}$ so that the adjoint of $U$ is given by
\[
U^*(\ze_j) = \sum_{i=1}^{d} \de_i \ga^{-1}(b_{ij}^*) \foral j= 1, \dots, d.
\]

Both $[b_{ij}]$ and $[R_i^*R_i b_{ij} \ga(T_j^*T_j)]$ produce the same $U$, thus the correspondence $[b_{ij}] \mapsto U$ is not one-to-one.
We will often be replacing the $b_{ij}$ by the $R_i^*R_i b_{ij} \ga(T_j^*T_j)$.
We will call this procedure \emph{the calibration} of $[b_{ij}]$.
Hence the identity maps on $E$ and on $F$ are represented respectively by
\[
[t_{ij}] =\diag\{ T_i^* T_i \mid i=1, \dots, d\}
\]
and
\[
[r_{ij}] =\diag\{ R_i^* R_i \mid i=1, \dots, d'\} .
\]

If $\xi = \sum_{j=1}^d \de_j a_j \in E$ is identified with the column vector $[a_1, \ldots, a_d]^t  \in \oplus_{i=1}^d T_i^* T_i A$, then $U(\xi)$ may be identified with the vector $[b_1, \ldots, b_{d'}]^t \in \oplus_{j=1}^{d'} R_i^* R_i B$ where $b_i = \sum_{j=1}^d b_{ij} \ga(a_j)$.
Hence the composition $(\ga_1, U_1) \in \L(E,F)$ related to $[b_{ij}]$ with $(\ga_2, U_2) \in \L(D,E)$ related to $[a_{ij}]$ is represented by the matrix $[b_{ij}] \cdot [\ga_1(a_{ij})]$.
In particular a $W \in\L(E)$ that is represented by $[w_{ij}]$ is injective (resp. surjective) if and only if there is a matrix $[v_{ij}] \in M_{d}(A)$ such that $[v_{ij}] \cdot  [w_{ij}] = [t_{ij}]$ (resp. $[w_{ij}] \cdot [v_{ij}] = [t_{ij}]$).

If in addition $(\ga, U) \colon {}_A E_A \to {}_B F_B$ is a C*-correspondences map, then the left module properties $a \de_j = \de_j \al_j(a)$ and $b \ze_i = \ze_i \be_i(b)$ give that
\begin{align*}
\sum_{i=1}^{d'} \ze_i \be_i (\ga(a)) b_{ij}
& = \ga(a) U(\de_j)
  = U(\de_j) \ga (\al_j(a))
  = \sum_{j=1}^{d'} \ze_i b_{ij} \ga(\al_j(a)) ,
\end{align*}
for all $a\in A$.
Therefore we get that
\begin{align*}
\be_i \ga(\cdot) \,  b_{ij} = b_{ij} \, \ga\al_j(\cdot) \foral i=1, \dots,d' \text{ and } j =1, \dots, d.
\end{align*}
Conversely if there is a $*$-isomorphism $\ga \colon A \to B$, then a matrix $[b_{ij}] \in M_{d' \times d}(B)$ satisfying $\be_i \ga(\cdot) \, b_{ij} = b_{ij} \, \ga\al_j(\cdot)$ defines a C*-correspondences map $(\ga, U) \colon {}_A E_A \to {}_B F_B$.
In particular, we can replace the coefficients $b_{ij}$ with the calibrated elements $R_i^* R_i b_{ij} \ga(T_j^* T_j)$.
This follows by the fact that $A$ and $B$ are commutative, and because $\be_i(\cdot) R_i^*R_i = \be_i(\cdot)$ and $T_j^*T_j \al_j(\cdot) = \al_j(\cdot)$.

\begin{proposition}\label{P: support}
Let $\I \lh \bC\sca{x_1, \dots, x_d}$ and $\J \lh \bC\sca{y_1, \dots, y_{d'}}$ be monomial ideals.
If there is an invertible C*-correspondence map between the associated C*-correspondences ${}_A E_A$ and ${}_B F_B$, then $d= d'$, and
\[
\left| \left\{j \in \{1, \dots, d\} \mid P_{[n]} \ga(T_j^* T_j) \neq 0 \right\} \right| = |\supp [n] |,
\]
for all $[n] \in \{1, \dots, 2^{d'} - 1\}$.
\end{proposition}

\begin{proof}
Let $(\ga, V) \colon E \to F$ be an invertible C*-correspondence map.
Then there are matrices $[b_{ij}] \in M_{d' \times d}(B)$ and $[a_{ij}] \in M_{d \times d'}(A)$ such that $[b_{ij}] \cdot [\ga(a_{ij})] = [r_{ij}]$.
By symmetry there are also matrices $[f_{ij}] \in M_{d \times d'}(A)$ and $[g_{ij}] \in M_{d' \times d}(B)$ so that $[f_{ij}] [\ga^{-1}(g_{ij})] = [t_{ij}]$.
Without loss of generality we may assume that all $b_{ij}$, $a_{ij}$, $f_{ij}$, and $g_{ij}$ are calibrated.

First we show that $d=d'$.
To this end for $P := P_{[1\dots 1]} = R_1^* R_1 \dots R_{d'}^* R_{d'}$ we obtain that
\[
P \otimes I_{d'} \cdot [b_{ij}] \cdot [\ga(a_{ij})] \cdot P \otimes I_{d'} = P \otimes I_{d'} \cdot [r_{ij}] \cdot P \otimes I_{d'}.
\]
Since $P$ is a non-zero subprojection of every $R_i^* R_i$ and by using commutativity of $B$ we get that
\[
[P b_{ij} P] \cdot [P \ga(a_{ij}) P] = P \otimes I_{d'}.
\]
This shows that the matrix $[P b_{ij} P]$ with entries from the commutative and unital C*-algebra $P B P$ is right invertible.
As a consequence we obtain that $d' \leq d$.
By symmetry on $[f_{ij}] [\ga^{-1}(g_{ij})] = [t_{ij}]$ we obtain that $d \leq d'$ as well.

We repeat for the minimal projections of $\ca(R_i^*R_i \mid i=1, \dots, d)$.
Without loss of generality fix $[n]= [1\dots 1 0 \dots 0]$ and the associated projection
\[
P:= P_{[n]} = R_1^* R_1 \dots R_k^* R_k (I- R_{k+1}^* R_{k+1}) \dots (I - R_d^* R_d).
\]
We may assume so by applying a permutation on the variables $i=1, \dots, d$, which produces a unitary equivalence on ${}_B F_B$.
For convenience let us set
\[
d-l = | \{j=1, \dots, d \mid P \ga(T_j^* T_j) = 0\} |.
\]
After permuting the variables (which produces a unitary equivalence on ${}_A E_A$), we may assume that
\[
\{j=1, \dots, d \mid P \ga(T_j^*T_j) = 0\} = \{l+1, \dots, d\}.
\]
We aim to show that $l=k$.
By commutativity in $B$ we have that the equation
\[
P \otimes I_{d'} \cdot [b_{ij}] \cdot [\ga(a_{ij})] \cdot P \otimes I_{d'} = P \otimes I_{d'} \cdot [r_{ij}] \cdot P \otimes I_{d'},
\]
implies
\[
\begin{bmatrix} [P b_{ij} P ]_{k \times l} & 0 \\ 0 & 0 \end{bmatrix} \begin{bmatrix} [P \ga(a_{ij}) P ]_{l \times k} & 0 \\ 0 & 0 \end{bmatrix} = \begin{bmatrix} [P]_{k \times k} & 0 \\ 0 & 0 \end{bmatrix}.
\]
Reasoning as above for the commutative and unital $PBP$ we get that $k \leq l$.
To reach contradiction suppose that $k < l$.
Now $P$ is written as the sum of $P \ga(Q_{[m]})$ for $[m] =0, \dots, 2^{d}-1$ and $k\geq 1$.
Thus there exists an $[m]$ with $|\supp [m]| = l > 1$ such that $P \ga(Q_{[m]}) \neq 0$.
Since we have permuted the variables to have
\[
\{j=1, \dots, d \mid P \ga(T_j^*T_j) = 0\} = \{l+1, \dots, d\},
\]
we can take
\[
Q \equiv Q_{[m]} = T_1^* T_1 \dots T_l^* T_l (I- T_{l+1}^* T_{l+1}) \dots (I - T_d^* T_d).
\]
Since $Q$ is minimal and $P \ga(Q) \neq 0$, then we obtain that $W := \ga^{-1}(P) Q = Q$.
Applying on $[f_{ij}] \cdot [\ga^{-1}(g_{ij})] = [t_{ij}]$ we obtain
\begin{align*}
[W f_{ij} W ]_{l \times k} \cdot [ W \ga^{-1}(g_{ij}) W ]_{l \times k}
=
[W t_{ij} W]_{l \times l}
=
W \otimes I_l.
\end{align*}
Reasoning as above for the commutative unital C*-algebra $QAQ$ we get that $l \leq k$ which is a contradiction.
Therefore $k$ must be equal to $l$.
\end{proof}

\begin{proposition}\label{P: un eq}
Let $\I \lh \bC\sca{x_1, \dots, x_d}$ and $\J \lh \bC\sca{y_1, \dots, y_{d'}}$ be monomial ideals, and let ${}_A E_A$ and ${}_B F_B$ be the associated C*-correspondences.
Then ${}_A E_A$ is unitarily equivalent to ${}_B F_B$ if and only if:
\begin{enumerate}
\item[\textup{(a)}] $d=d'$;
\item[\textup{(b)}] there exists a $*$-isomorphism $\ga \colon A \to B$; and
\item[\textup{(c)}] there exist $[b_{ij}] \in M_d(B)$ and $[a_{ij}] \in M_d(A)$ such that $\be_i\ga(\cdot) b_{ij} = b_{ij} \ga \al_j(\cdot)$ and
\begin{enumerate}
\item[\textup{(i)}] $\left| \left\{j \in \{1, \dots, d\} \mid P_{[n]} \ga(T_j^* T_j) \neq 0 \right\} \right| = |\supp [n] |$;
\item[\textup{(ii)}] $[P_{[n]} b_{ij} P_{[n]}] \cdot [P_{[n]} \ga(a_{ij}) P_{[n]}] = [P_{[n]} r_{ij} P_{[n]}]$;
\end{enumerate}
for all the minimal projections $P_{[n]} \in \ca(R_i^*R_i \mid i=1, \dots, d)$.
\end{enumerate}
\end{proposition}

\begin{proof}
If $(\ga, V) \colon E \to F$ is an invertible C*-correspondence map then that $d=d'$ and item (i) follow from Proposition \ref{P: support}.
In this case if $[b_{ij}]$ is the matrix associated with $V$ and $[a_{ij}]$ is the matrix associated with $V^{-1}$ then the identity $V V^{-1} = I_F$ implies that $[b_{ij}] \cdot [\ga(a_{ij})] = [r_{ij}]$.
Multiplying the latter equation by $P_{[n]} \otimes I_{d}$ and using commutativity of $B$ gives item (ii).

Conversely, the assumption $\be_i\ga(\cdot) b_{ij} = b_{ij} \ga \al_j(\cdot)$ implies that the matrix $[b_{ij}]$ defines a C*-correspondence map between $E$ and $F$.
We may assume that the $[b_{ij}]$ and (hence) the $[a_{ij}]$ are calibrated.
Fix a minimal projection $P \in \ca(R_i^*R_i \mid i=1,\dots,d)$. Without loss of generality we may suppose that
\[
P = R_1^* R_1 \dots R_k^* R_k (I- R_{k+1}^* R_{k+1}) \dots (I - R_d^* R_d).
\]
After permuting the columns we have that item (ii) implies
\[
\begin{bmatrix} [P b_{ij} P ]_{k \times l} & 0 \\ 0 & 0 \end{bmatrix} \begin{bmatrix} [P \ga(a_{ij}) P ]_{l \times k} & 0 \\ 0 & 0 \end{bmatrix} = \begin{bmatrix} [P]_{k \times k} & 0 \\ 0 & 0 \end{bmatrix},
\]
where $l = |\{j \in \{1, \dots, d\} \mid P \ga(T_j^*T_j) = 0\}|$.
However item (i) implies that $k=l$.
Since $P B P$ is stably finite with identity $P$ we derive that
\[
\begin{bmatrix} [P \ga(a_{ij}) P ]_{k \times k} & 0 \\ 0 & 0 \end{bmatrix} \begin{bmatrix} [P b_{ij} P ]_{k \times k} & 0 \\ 0 & 0 \end{bmatrix} = \begin{bmatrix} [P]_{k \times k} & 0 \\ 0 & 0 \end{bmatrix}.
\]
Equivalently $[P \ga(a_{ij})P] \cdot [P b_{ij} P] = [P r_{ij} P]$.
By using commutativity once more we derive that
\[
P \otimes I_d \cdot [\ga(a_{ij})] \cdot [b_{ij}] = P \otimes I_d \cdot [r_{ij}],
\]
and the symmetrical
\[
P \otimes I_d \cdot [b_{ij}] \cdot [\ga(a_{ij})] = P \otimes I_d \cdot [r_{ij}],
\]
from item (ii).
Since the projections $P$ add up to an identity for the calibrated elements we get that $[b_{ij}] \cdot [\ga(a_{ij})] = [r_{ij}]$ and that $[\ga(a_{ij})] \cdot [b_{ij}] = [r_{ij}]$.
The first equation shows that the C*-correspondence map $U \colon E \to F$ associated with $[b_{ij}]$ is surjective and the second one that $U$ is injective.
Therefore the unitary $U|U|^{-1}$ provides the unitary equivalence.
\end{proof}

\subsection{Isometric isomorphisms}

The following result will be refined later by Corollary \ref{C: classification}.
We include an independent proof as it provides evidence of some additional structure (see also Remark \ref{R: loc conj}).

\begin{theorem}\label{T: isom 1}
Let $\I \lh \bC\sca{x_1, \dots, x_d}$ and $\J \lh \bC\sca{y_1, \dots, y_{d'}}$ be monomial ideals, and let ${}_A E_A$ and ${}_B F_B$ be the associated C*-correspondences.
The following are equivalent:
\begin{enumerate}
\item $\T_E^+$ and $\T_F^+$ are completely isometrically isomorphic;
\item $\T_E^+$ and $\T_F^+$ are isometrically isomorphic;
\item ${}_A E_A$ and ${}_B F_B$ are unitarily equivalent.
\end{enumerate}
\end{theorem}

\begin{proof}
The implications $[\textup{(iii)} \Rightarrow \textup{(i)} \Rightarrow \textup{(ii)}]$ are immediate by the general theory of C*-correspondences.
For the implication $[\textup{(ii)} \Rightarrow \textup{(iii)}]$, fix an isometric isomorphism $\Phi \colon \T_E^+ \to \T_F^+$ and denote $\ga = \Phi|_A$.
Recall that $A$ and $B$ are maximal C*-algebras inside $\T_E^+$ and $\T_F^+$, respectively.
Since $\ga$ is isometric it is a $*$-homomorphism and thus maps $A$ into $B$.
The symmetrical argument for $\ga^{-1} = \Phi^{-1}|_B$ gives that $\ga \colon A \to B$ is a $*$-isomorphism.

Let us fix notation.
Suppose that $\T_E^+$ is generated by $A$ and $\Bv_j$ for $j=1, \dots, d$, and that $\T_F^+$ is generated by $B$ and $\Bw_i$ for $i=1, \dots d'$.

\smallskip

\noindent{\bf Claim 1.} \textit{There exists a C*-correspondence map $(\ga,U) \colon E \rightarrow F$.}

\smallskip

\noindent {\it Proof of Claim 1.}
By the Fourier analysis of Section \ref{Ss: tensor} we can write
\[
\Phi(\Bv_j) = b_0 + \left(\sum_{i=1}^{d'} \Bw_i b_{ij}\right) + f,
\]
with $f$ having zero Fourier coefficients on $\mt, 1, \dots, d'$.
Even more we can choose the $b_{ij}$ so that $R_i^*R_i b_{ij} = b_{ij}$.
Furthermore we have that
\[
\Phi(\Bv_j) = \Phi(\Bv_j T_j^*T_j) = \Phi(\Bv_j) \ga(T_j^*T_j)
\]
hence $\Bw_i b_{ij} = \Bw_i b_{ij} \ga(T_j^*T_j)$.
Thus we obtain the calibrated
\[
b_{ij} = R_i^*R_i b_{ij} = R_i^*R_i  b_{ij} \ga(T_j^*T_j) = b_{ij} \ga(T_j^*T_j).
\]
Recall the algebraic relations
\[
a \cdot \Bv_j = \Bv_j \cdot \al_j(a) \qand b \cdot \Bw_i = \Bw_i \cdot \be_i(b).
\]
Applying then $\Phi$ on the first one and using the second we derive that
\begin{align*}
\ga(a)b_0 + \left(\sum_{i=1}^{d'} \Bw_i \be_i\ga(a) b_{ij}\right) + f'
& =
b_0\ga\al_j(a) + \left(\sum_{i=1}^{d'} \Bw_i b_{ij}\ga\al_j(a) \right) + f''.
\end{align*}
Then the uniqueness of the coefficients gives that $\be_i \ga(\cdot) \,  b_{ij} = b_{ij} \, \ga\al_j(\cdot)$.
Letting $U \colon E \to F$ be determined by $[b_{ij}]$ completes the proof of Claim 1.

\smallskip

Our goal is to show that $U$ is invertible, since then $U|U|^{-1}$ will produce the unitary equivalence between $E$ and $F$.
We proceed in two steps.

\smallskip

\noindent{\bf Claim 2.} \textit{The C*-correspondence map $U \colon E \to F$ is surjective.}

\smallskip

\noindent {\it Proof of Claim 2.}
The arguments of \cite[Lemma 4.2]{KakKat12} apply in our case.
Indeed the arguments of \cite[Lemma 4.2]{KakKat12} depend only on surjectivity, continuity, and the covariant relations $\be_i \ga(\cdot) \,  b_{ij} = b_{ij} \, \ga\al_j(\cdot)$.
Hence for any tuple $[y_1, y_2, \dots, y_{{d'}}] \in \oplus_{i = 1}^{d'}B$ there exist a sequence $ \left( [x_{1}^{k} , x_{2}^{k}, \dots, x_{d}^{k}] \right)_k$ in $\oplus_{j = 1}^{d}B$ such that
\begin{align*}
\Bw_i y_i & = \lim_k \Bw_i (b_{i1}x^k_1 + b_{i2} x^k_2 + \cdots + b_{i d} x^k_{d}), \foral i=1,2, \dots, d'.
\end{align*}
Since $R_i^*R_i b_{ij} = b_{ij}$ we get that
\[
R_i^* R_i y_i = \lim_k b_{i1}x^k_1 + b_{i2} x^k_2 + \cdots + b_{i d} x^k_{d}, \foral i=1,2, \dots, d'.
\]
Applying this for the tuples $[R_1^*R_1, 0 , \dots, 0], \dots, [0,\dots, 0, R_d^*R_d]$ and proceeding as in \cite[Proposition 4.3]{KakKat12} we can find $[x_{ij}]$ such that
\[
\nor{[r_{ij}] - [b_{ij}] [x_{ij}]} < 1,
\]
where $\nor{\cdot}$ denotes the norm in $M_{d'}(B)$.
We may as well substitute $x_{ij}$ by the elements $\ga(T_i^* T_i) x_{ij} R_j^* R_j$ and still reach the same conclusion.
Indeed we have that
\begin{align*}
\nor{[r_{ij}] - [b_{ij}] [\ga(T_i^* T_i) x_{ij} R_j^* R_j]}
& =
\nor{[r_{ij}] - [b_{ij} \ga(T_i^* T_i)] [x_{ij} R_j^* R_j]}\\
& =
\nor{([r_{ij}] - [b_{ij}] [x_{ij}]) [r_{ij}]}\\
& \leq
\nor{[r_{ij}] -  [b_{ij}] [x_{ij}]}
< 1.
\end{align*}
Therefore we obtain that
\[
[r_{ij}] [b_{ij}] [x_{ij}] = [b_{ij}] [x_{ij}] [r_{ij}] = [b_{ij}] [x_{ij}].
\]
Set for convenience $g = [r_{ij}] - [b_{ij}] [x_{ij}]$ which is a matrix in $M_{d'}(B)$.
Then the series $\sum_{n=1}^N g^n$ converge and $[r_{ij}] g = g [r_{ij}] = g$. Therefore
\[
[b_{ij}] [x_{ij}] \cdot \sum_{n=0}^\infty g^n = ([r_{ij}] - g) \sum_{n=0}^\infty g^n = [r_{ij}].
\]
Thus there is an element $[c_{ij}] \in M_{d'}(B)$ such that
\[
[b_{ij}][x_{ij}][c_{ij}] = [r_{ij}].
\]
Letting $[a_{ij}] = [\ga^{-1}(x_{ij})][\ga^{-1}(c_{ij})] \in M_{d \times d'}(A)$ we obtain $[b_{ij}] [\ga(a_{ij})] = [r_{ij}]$.
Without loss of generality we may assume that  $a_{ij} = T_i^* T_i a_{ij} \ga^{-1}(R_j^* R_j)$.
This follows by the computation
\begin{align*}
[b_{ij}] \cdot [\ga(T_i^* T_i) \ga(a_{ij}) R_j^* R_j]
& =
[b_{ij}] \cdot [\ga(t_{ij})] \cdot [\ga(a_{ij})] \cdot [r_{ij}] \\
& =
[b_{ij} \ga(T_j^* T_j)] \cdot [\ga(a_{ij})] \cdot [r_{ij}] \\
& =
[b_{ij}] \cdot [\ga(a_{ij})] \cdot [r_{ij}] = [r_{ij}].
\end{align*}
Let $(\ga^{-1},V) \colon F \to E$ be the adjointable map associated with $[a_{ij}]$.
We compute
\begin{align*}
U V(\ze_i)
& =
\sum_{k=1}^d \sum_{l=1}^{d'} \ze_l b_{lk} \ga(a_{ki})
=
\sum_{l=1}^{d'} \ze_l \de_{l,i} R_i^* R_i
=
\ze_i R_i^* R_i = \ze_i
\end{align*}
hence $UV = I_F$ and the proof of Claim 2 is complete.

\smallskip

Reasoning with $\Phi^{-1}$ instead of $\Phi$, we find that there exists a surjective C*-correspondence map $\wt U \colon F \rightarrow E$. Denote $W = \wt{U} U \colon E \rightarrow E$. The following claim will conclude the proof of the theorem.

\smallskip

\noindent{\bf Claim 3.} \textit{The C*-correspondence map $W \colon E \to E$ is injective.}

\smallskip

\noindent {\it Proof of Claim 3.}
Denote by $[w_{ij}] \in M_d(A)$ the matrix representing $W$. Since $W$ is surjective, there is a matrix $[v_{ij}]$ such that
\[
[w_{ij}] \cdot [v_{ij}] = [t_{ij}] .
\]
It suffices to show that $[v_{ij}] \cdot [w_{ij}] = [t_{ij}]$ also holds.
For this, it is enough to show that
\[
Q_{[m]} \otimes I_d \cdot [v_{ij}] \cdot [w_{ij}] = Q_{[m]} \otimes I_d \cdot [t_{ij}]
\]
for every minimal projection $Q_{[m]}$ in $\ca(T^*_i T_i \mid i = 1, \ldots, d)$, since these projections add up to an identity of the entries by Lemma \ref{L: unital}.
As above, we may assume that every $w_{ij}$ has already been calibrated to satisfy $w_{ij} = T^*T_i w_{ij} T^*_j T_j$ and similarly for $v_{ij}$.

To this end let $Q_{[m]}$ be a minimal projection in $\ca(T^*_iT_i \mid i = 1, \ldots, d)$, which without loss of generality we assume
\[
Q \equiv Q_{[m]} = T_1^* T_1 \dots T_k^* T_k (I- T_{k+1}^* T_{k+1}) \dots (I - T_d^* T_d) .
\]
Using commutativity of $A$, we have that
\[
Q \otimes I_d \cdot [w_{ij}] \cdot [v_{ij}] = Q \otimes I_d \cdot [t_{ij}]
\]
implies
\[
\begin{bmatrix} [Q w_{ij} Q]_{k \times k} & 0 \\ 0 & 0 \end{bmatrix} \begin{bmatrix} [Q v_{ij} Q]_{k \times k} & 0 \\ 0 & 0 \end{bmatrix} = \begin{bmatrix} [Q]_{k \times k} & 0 \\ 0 & 0 \end{bmatrix}.
\]
As $QAQ$ is commutative, we have that
\[
[Q v_{ij} Q ]_{k \times k} \cdot [Q w_{ij} Q]_{k \times k} = Q \otimes I_k .
\]
Working backwards we find that $Q \otimes I_d \cdot [v_{ij}] \cdot [w_{ij}] = Q \otimes I_d \cdot [t_{ij}]$ as well and the proof of Claim 3 is complete.
\end{proof}

\subsection{Local piecewise conjugacy}

Davidson and Katsoulis \cite{DavKat11} introduced the notion of piecewise conjugacy for multivariable classical systems.
Let $X$ and $Y$ be locally compact Hausdorff spaces and suppose that $\si_1, \dots, \si_d$ are proper continuous self-mappings of $X$, and $\tau_1, \dots, \tau_{d}$ are proper continuous self-mappings of $Y$.
The classical systems $(X,\si) \equiv (X, \si_1, \dots, \si_d)$ and $(Y,\tau) \equiv (Y,\tau_1, \dots, \tau_{d})$ are called \emph{piecewise conjugate} if there is a homeomorphism $\ga_s \colon Y \to X$, and for every $y\in Y$ there is a permutation $\pi \in S_d$ and a neighbourhood $\U_\pi$ such that
\[
\ga_s \tau_i|_{\U_\pi} = \si_{\pi(i)} \ga_s|_{\U_\pi} \foral i=1, \dots, d.
\]
Equivalently there is an open cover $\{\U_\pi \mid \pi \in S_d\}$ of $Y$ such that the above equation holds.
By local compactness we can substitute the $\U_\pi$ by open subsets $\V_\pi$ that are relatively compact and $\V_\pi \subseteq \ol{\V_\pi} \subseteq \U_\pi$.

One of the breakthrough results of Davidson and Katsoulis \cite[Theorem 3.22]{DavKat11} is that algebraic homomorphism of the tensor algebras associated with the dynamics implies piecewise conjugacy of the dynamics.
In the appendix we include an alternative proof of \cite[Theorem 3.22]{DavKat11}.
We replace the nest representations of \cite{DavKat11, DavRoy11} by the simpler compressions of the Fock representation.
We hope that this analysis will shed light to a general converse of \cite[Theorem 3.22]{DavKat11} as well as will serve to clarify points as the following.

\begin{remark}\label{R: DK tensor}
Davidson and Katsoulis \cite[Definition 1.2]{DavKat11} consider the algebra $\A(X,\si)$ generated by $\Bs_i$ and $C_0(X)$ separately, where the $\Bs_i$ form a row contraction and $f \cdot \Bs_i = \Bs_i f \cdot \si_i$.
The algebra $\A(X,\si)$ is called \emph{the tensor algebra} in \cite{DavKat11}.
It is shown that $\A(X,\si)$ is completely isometrically isomorphic to the tensor algebra $\T_\E^+$ of a C*-correspondence $\E$ \cite[Theorem 2.10]{DavKat11}.
In particular $\E = \sum_{i=1}^d \E_i$ where every $\E_i = C_0(X)$ becomes a C*-correspondence over $C_0(X)$ by defining
\[
f \cdot \xi \cdot h = (\xi h) f \si_i \qand \sca{\xi, \eta} = \xi^* \eta.
\]
However in general $\T_\E^+$ is generated by $\Bs_i f$ and there is no apparent reason why the $\Bs_i$ can be isolated from the $f \in C_0(X)$ when $X$ is not compact.
What holds is that $\T_\E^+ \subseteq \A(X,\si)$.
However in the C*-correspondences literature the tensor algebra of $(X,\si)$ is $\T_\E^+$ and not $\A(X,\si)$.
The latter could be seen as the tensor algebra of the one-point compactification of $(X,\si)$.

It is not clear to us whether isomorphism of the tensor algebras in the sense of Davidson and Katsoulis \cite{DavKat11} implies isomorphism of the possibly smaller tensor algebras of the C*-correspondences in the sense of Muhly and Solel \cite{MuhSol98}.
In the appendix we show directly that algebraic isomorphism of the tensor algebras of the C*-correspondences implies piecewise conjugacy.
This alternative proof obviously works also for the tensor algebras in the sense of Davidson and Katsoulis \cite{DavKat11}.

The fact that $\A(X,\si)$ is generated by $\Bs_i$ and $C_0(X)$ separately is used in a non trivial way in \cite{DavKat11}.
Examples include \cite[Lemma 3.7 and Lemma 3.9]{DavKat11} where the characters are completely identified.
Proving such remarks for $\T_\E^+$ requires some extra work (respectively, see Claim 1 in the appendix).
A second example is \cite[Example 3.14]{DavKat11} in conjunction with Claim 1 and Claim 6 of the appendix.
\end{remark}

Davidson and Roydor \cite{DavRoy11} notice that the notion of piecewise conjugacy passes naturally to the topological graphs in the sense of Katsura \cite{Kat04-top}.
Let $(\X^0, \X^1, r_\X, s_\X)$ and $(\Ups_0, \Ups^1, r_\Ups, s_\Ups)$ be topological graphs on compact spaces $\X^0$ and $\Ups^0$.
They are said to be \emph{locally conjugate} if there exists a homeomorphism $\ga_0 \colon \Ups^0 \to \X^0$ such that for every $y \in \Ups^0$ there is a neighbourhood $\U$ of $y$ and a homeomorphism $\ga_1$ of $s_{\Ups}^{-1}(\U)$ onto $s_{\X}^{-1}\ga_0(\U)$ such that
\[
s_{\X} \ga_1 = \ga_0 s_\Ups|_{s_{\Ups}^{-1}(\U)} \qand r_\X \ga_1 = \ga_0 r_{\Ups}|_{s_{\Ups}^{-1}(\U)}.
\]
Notice here that by definition $s_{\Ups}^{-1}(\U) = \mt$ if and only if $s_{\X}^{-1} \ga_0(\U) = \mt$.

In one of the main results, Davidson and Roydor \cite[Theorem 4.5]{DavRoy11} show that algebraic isomorphisms of the tensor algebras associated with compact topological graphs implies local conjugacy.
We note here that classical systems form topological graphs.
Therefore \cite[Theorem 4.5]{DavRoy11} contains \cite[Theorem 3.22]{DavKat11} modulo the following remarks.

\begin{remark}\label{R: DavRoy11}
First \cite[Theorem 3.22]{DavKat11} concerns classical systems over locally compact spaces in general.
Secondly in the proof of \cite[Theorem 3.22]{DavKat11} it is shown that algebraic homomorphisms are automatically continuous.
An analogous observation in the proof of \cite[Theorem 4.5]{DavRoy11} is missing.
However Davidson-Roydor make essential use of the fact that the isomorphism is bounded, see for example \cite[Last paragraph of page 1257]{DavRoy11}.
Whether algebraic isomorphisms of the tensor algebras are automatically continuous is something unknown to us.
In fact even for our case we are able to show that this holds only under an assumption on the graph of the ideal (see Proposition \ref{P: auto cont} that will follow).
Therefore one may consider the isomorphisms in the results of \cite{DavRoy11} to be bounded.
\end{remark}

In Section \ref{Ss: topological} we observed that the C*-correspondence associated with a monomial ideal can be realised as the C*-correspondence of a topological graph.
Hence the investigation of Davidson-Roydor \cite{DavRoy11} applies to our case.
Our purpose however in the sequel is to identify the translation of local conjugacy in terms of our original data as well as to expose the implicit structure via the alternative proof that we offer.

Recall from Section \ref{Ss: qd} that the quantised dynamics $(A,\al)$ of the monomial ideal $\I \lhd \bC\sca{x_1, \dots, x_d}$ correspond to continuous mappings $\vpi_i \colon \Om_\I^i \to \Om_\I$ when identifying $A$ with $C(\Om_\I)$.
Similarly we denote by $(\Om_\J, \psi)$ the dynamical system associated with the monomial ideal $\J \lh \bC\sca{y_1, \dots, y_{d}}$.
For the following discussion we identify the projections $P_{[n]}$ and $Q_{[m]}$ with the clopen set in the spectrum of which each one is a characteristic function.

\begin{definition}\label{D: QPloc}
We say that the systems $(\Om_\I, \vpi)$ and $(\Om_\J, \psi)$ are \emph{$Q$-$P$-locally piecewise conjugate} if there exists a homeomorphism $\ga_s \colon \Om_\J \to \Om_\I$, and for every $y \in P_{[n]}$ there is a neighbourhood $\U \subseteq P_{[n]}$ of $y$ and an $[m] \in \{0,1, \dots, 2^d-1\}$ such that $|\supp [m]| = |\supp [n]|$, $\ga_s(\U) \subseteq Q_{[m]}$, and
\[
\ga_s \vpi_i|_{\U} = \psi_{\pi(i)} \ga_s|_{\U},
\]
for a bijection $\pi \colon \supp [m] \to \supp [n]$.
\end{definition}
Equivalently every $P_{[n]} \subseteq \Om_\J$ has an open cover $\{\U_\pi\}_{\pi}$ indexed by the one-to-one correspondences $\pi\colon \supp [n] \to \{1, \dots, d\}$ such that $\ga_s(\U_\pi) \subseteq Q_{[m]}$ for all $[m]$ with $\supp [m] = \pi(\supp [n])$, and $\ga_s \tau_{\pi(i)} |_{\U_\pi} = \vpi_i \ga_s|_{\U_\pi}$.
We do not exclude the case where $\U_\pi = \mt$ for some $\pi$.

\begin{remark}\label{R: loc conj}
The terminology we use follows the geometric observation that when restricting to $P_{[n]}$ then (locally) the systems look piecewise conjugate.
The reason why the equation $|\supp [m]| = |\supp [n]|$ appears in the definition is actually implemented by Proposition \ref{P: support}.
We include the symbols $P$ and $Q$ in the definition to emphasise the use of the particular projections.
\end{remark}

\begin{remark}\label{R: Q0-P0}
Definition \ref{D: QPloc} suggests in particular that $\ga_s(P_{0}) = Q_{0}$ for $Q$-$P$-locally piecewise conjugate systems.
Indeed for $y \in P_{0}$ the only $[m]$ in $\{0,1, \dots, 2^d-1\}$ that has $|\supp [m]| = 0$ is $[m] = 0$, hence $\ga_s(y) \in Q_{0}$.
Thus we also obtain that $\ga(A_0) = B_0$ for the $*$-isomorphism $\ga \colon A \to B$ implemented by $\ga_s \colon \Om_\J \to \Om_\I$.

Furthermore we can extend trivially the equality $\ga_s \vpi_i|_\U = \psi_{\pi(i)} \ga_s|_\U$ from the $i \in \supp[n]$ to all the $i \in \{1, \dots, d\}$ by extending $\pi$ to an arbitrary permutation of $\{1,\ldots, d\}$.
Indeed let $\ga \colon A \to B$ be the $*$-isomorphism implemented by $\ga_s$.
By definition $i \notin \supp[n]$ if and only if $\pi(i) \notin \supp [m]$ since $\pi(\supp [n]) = \supp[m]$.
For all $y \in P_{[n]}$ there exists a neighbourhood $\U$ of $y$ such that $\ga_s(\U) \subseteq Q_{[m]}$.
This implies that $\ga^{-1}(P) T_{\pi(i)}^* T_{\pi(i)} = 0$ for all subprojections $P \subseteq P_{[n]}$ and $i \notin \supp [n]$.
Thus $\ga^{-1}(P_{[n]}) T_{\pi(i)}^* T_{\pi(i)} = 0$ when  $i \notin \supp [n]$.
Now we have to verify that $\be_i \ga(a) P_{[n]} = P_{[n]} \ga \al_{\pi(i)}(a)$ for all $a \in A$; equivalently that
\begin{align*}
R_i^* \ga(a) R_i \cdot R_i^* R_i P_{[n]}
& =
P_{[n]} \ga(T_{\pi(i)}^*T_{\pi(i)}) \cdot \ga(T_{\pi(i)}^* aT_{\pi(i)})
\end{align*}
for all $a \in A$ and $i\in \{1, \dots, d\}$.
This holds trivially for all $i \notin \supp [n]$ since both sides are zero, and it holds by definition when $i \in \supp [n]$.
In particular the equation $\be_i \ga(\cdot) P_{[n]} = P_{[n]} \ga \al_{\pi(i)}(\cdot)$ is redundant for $[n] = 0$.
\end{remark}

The definition of local conjugacy passes to the systems implemented by $(A_0, \al)$ and $(B_0,\be)$ as well by restricting to the $[n] \neq 0$ and the $[m] \neq 0$.
In this case we write $(\Om_{0,\I}, \vpi)$ and $(\Om_{0,\J},\psi)$ for the induced classical systems.
There is a strong connection with $(\Om_{\I}, \vpi)$ and $(\Om_{\J},\psi)$.

\begin{proposition}\label{P: Q0-P0}
Let $\I \lh \bC\sca{x_1, \dots, x_d}$ and $\J \lh \bC\sca{y_1, \dots, y_{d}}$ be monomial ideals.
Then the corresponding systems $(\Om_\I,\vpi)$ and $(\Om_\J,\psi)$ are $Q$-$P$-locally piecewise conjugate if and only if the systems $(\Om_{0,\I}, \vpi)$ and $(\Om_{0,\J},\psi)$ are $Q$-$P$-locally piecewise conjugate.
\end{proposition}

\begin{proof}
For the forward implication, Remark \ref{R: Q0-P0} gives that the isomorphism $\ga \colon A \to B$ implemented by $\ga_s \colon \Om_\J \to \Om_\I$ is such that $\ga(A_0) = B_0$.
For the converse it suffices to show that the $*$-isomorphism $\ga \colon A_0 \to B_0$ implemented by $\ga_s \colon \Om_{0,\J} \to \Om_{0,\I}$ extends to an isomorphism $\wt{\ga} \colon A \to B$.
Then Remark \ref{R: Q0-P0} provides the appropriate context to obtain the locally conjugate relations.
To this end recall that $A_0 \subseteq \ca(T)$ with identity $e = I_{\F_X} - P_{0}$ by Lemma \ref{L: unital}.
Similarly we have $f = I_{\F_Y} - Q_{0}$ for the identity of $B_0 \subseteq \ca(R)$, and so $\ga(e) = f$.
We extend $\ga$ to the unitization
\[
\wt{\ga} \colon A_0 + \bC \cdot I_{\F_X} \to B + \bC \cdot I_{\F_X}
\]
so that $\wt\ga(I_{\F_X}) = I_{\F_Y}$.
However $A_0 + \bC \cdot I_{\F_X} = A$ and $B_0 + \bC \cdot I_{\F_X} = B$ and by construction we obtain
\[
R_{0} = I_{\F_Y} + f = \ga(I_{\F_X}) + \ga(e) = \ga(P_{0}).
\]
The extension $\wt{\ga}$ is then the required $*$-isomorphism.
\end{proof}

The next proposition provides the appropriate link with \cite{DavRoy11}.

\begin{proposition}\label{P: lc same lpc}
Let $\I \lh \bC\sca{x_1, \dots, x_d}$ and $\J \lh \bC\sca{y_1, \dots, y_{d}}$ be monomial ideals.
Then the topological graphs $(\Ups^0_\I, \Ups_\I^1, r_\I, s_\I)$ and $(\Ups^0_\J, \Ups_\J^1, r_\J, s_\J)$ are locally conjugate if and only if the systems $(\Om_\I,\vpi)$ and $(\Om_\J,\psi)$ are $Q$-$P$-locally piecewise conjugate.
\end{proposition}

\begin{proof}
The converse implication is straightforward by applying the definitions.
For the forward implication suppose that the topological graphs are locally conjugate and fix a $y \in P_{[n]}$ for some $[n] \in \{0, 1, \dots, 2^d-1\}$.
By setting $\ga_s = \ga_0$ we get the homeomorphism $\ga_s \colon \Om_\J \to \Om_\I$.
Fix a point $y \in \Om_\J$.
Then $\ga_s(y) \in Q_{[m]}$ for some $[m] \in \{0, 1, \dots, 2^d-1\}$.
By intersecting with $\ga_s^{-1}(Q_{[m]})$ and $P_{[n]}$ we may assume that the neighbourhood $\U$ of $y$ appearing in the definition of the local conjugacy satisfies both $\U \subseteq P_{[n]}$ and $\ga_s(\U) \subseteq Q_{[m]}$.
Furthermore we obtain that the spaces
\[
s_\J^{-1}(\U) = \{ (i,u) \mid i \in \supp [n], u\in \U\} = \sqcup_{i \in \supp [n]} \U
\]
and
\[
s_\I^{-1}\ga_s(\U) = \{ (j,\ga_s(u)) \mid j \in \supp [m], u\in \U\} = \sqcup_{j \in \supp [m]} \ga_s(\U)
\]
are homeomorphic by $\ga_1$.
This implies that $\supp [n] = \mt$ if and only if $\supp [m] = \mt$ thus $\ga_s(y) \in Q_{0}$ for every $y \in P_{0}$.
In this case the local conjugacy is verified trivially.

Hence from now on we turn our attention to the case where $[n] \neq 0$, for which we have that $[m] \neq 0$ as well.
If $\ga_1(i,u) = (j, \ga_s(u'))$ for some $j \in \supp [m]$ and some $u' \in \U$, then we obtain that
\[
\ga_s(u) = \ga_0 s_\J (i, u) = s_\I \ga_1(i,u) = s_\I(j,\ga_s(u')) = \ga_s(u'),
\]
hence $\ga_1(i,u) = (j, \ga_s(u))$ for some $j \in \supp [m]$.
Our aim is to show that for every fixed $i$ we get that the equation $\ga_1(i, u') = (j_i, \ga_s(u'))$ holds for all $u' \in \U$ and with the same $j_i$.
That is, every copy of $\U$ in $s_\J^{-1}(\U)$ maps exactly onto a copy $\ga_s(\U)$ of $s_\I^{-1}\ga_s(\U)$.
Then we get a bijection $\pi \colon i \mapsto j_i$ of $\supp [n]$ onto $\supp [m]$ which for all $u \in \U$ implies the required
\[
\ga_s \psi_i(u) = \ga_0 r_\J(i,u) = r_\I \ga_1(i,u) = r_\I(j_i, \ga_s(u)) = \vpi_{j_i} \ga_s(u) = \vpi_{\pi(i)} \ga_s(u).
\]

To achieve this we will further substitute $\U$ by a $\V \subseteq \U$ in the following way.
Recall that $\ga_1|_{\{1\} \times \U}$ is a homeomorphism onto its image.
If $\ga_1(1,y) =(i_1, \ga_s(y))$ then select $\V_1 \equiv \W_1 \subseteq \U$ such that
\[
\ga_1(\{1\} \times \V_1) \subseteq \{i_1\} \times \ga_s(\U).
\]
This can be achieved by considering (the projection of) the open set
\[
(\{1\} \times \U) \cap \ga_1^{-1}(\{i_1\} \times \ga_s(\U)).
\]
Now repeat the arguments above for $\V_1$ in place of $\U$.
If $\ga_1(2,y) =(i_2, \ga_s(y))$ then select $\V_2 \subseteq \U$ such that
\[
\ga_1(\{2\} \times \V_2) \subseteq \{i_2\} \times \ga_s(\U).
\]
Restrict further to $\W_2:=\V_1 \cap \V_2$ and repeat for $\W_2$ in place of $\U$.
By construction we obtain
\[
\ga_1(\{1\} \times \W_2) \subseteq \{i_1\} \times \ga_s(\U) \qand \ga_1(\{2\} \times \W_2) \subseteq \{i_2\} \times \ga_s(\U).
\]
Proceed inductively.
Without loss of generality we may assume that $\supp[n] = \{1,2,\ldots, N\}$, and $\supp[m] = \{1,2,\ldots, N'\}$.
The bijection $i \leftrightarrow j_i$ will be obtained after $N$ steps, unless $N' \neq N$.
But if $N > N'$ then there would be a step $N'+1$ for which we would have that
\[
\ga_1(\{N'+1\} \times \W_{N'}) \subseteq \cup_{j=1}^{N'} \{i_j\} \times \ga_s(\W_{N'}) = \ga_1(\cup_{j=1}^{N'} \{j\} \times \U),
\]
thus $\{N'+1\} \times \W_{N'}$ intersects some $\{j\} \times \U$ for some $j=1, \dots, N'$ which is a contradiction; hence $N \leq N'$.
On the other hand if $N < N'$ then we would have that the inverse image of $\{N'\} \times \ga_s(\W_{N})$ under $\ga_1$ intersects $\cup_{j=1}^{N} \{j\} \times \W_{N}$ and hence by composing with $\ga_1$ we derive that $\{N'\} \times \ga_s(\W_{N})$ intersects some of the $\{j\} \times \ga_s(\W_{N})$ for $j=1, \dots, N$, which is again a contradiction.
Thus $N' = N$ and the pairing is the one claimed.
\end{proof}

\begin{remark}\label{R: qd graph}
Local conjugacy is easy to illustrate for monomial ideals of finite type.
In this case the C*-algebra $A$ is finite because of Proposition \ref{P: finite A}.
Then the dynamical system $(\Om,\vpi)$ can be represented by a graph with vertices being the points of $\Om$, and from every $\om \in \Om$ there is an edge to $\vpi_i(\om)$ if and only if $\om \in \Om^i$.
That is, every vertex $\om \in Q_{[m]}$ has $\{e_i\}_{i \in \supp[m]}$ emitting edges, so that if $s(e_i) = \om$ then $r(e_i) = \vpi_i(\om)$.
We call the resulting graph $\G_\I$ the \emph{graph of $\I$}.
Then local conjugate systems of finite type monomial ideals correspond to isomorphic graphs of the monomial ideals.
\end{remark}

\subsection{Bounded isomorphisms}

The spaces $\Omega_\I$ and $\Omega_\J$ are compact and have dimension $0$.
Hence the results on topological graphs from \cite{DavRoy11} apply in our case.
Nevertheless we show that such a corollary can be derived by applying locally our alternative proof of \cite[Theorem 3.22]{DavKat11}.
The reader is referred to the appendix which contains the alternative proof of \cite[Theorem 3.22]{DavKat11} and the required elements for obtaining the following corollary.

\begin{corollary}\label{C: classification}
Let $\I \lh \bC\sca{x_1, \dots, x_d}$ and $\J \lh \bC\sca{y_1, \dots, y_{d'}}$ be monomial ideals.
Let ${}_A E_A$ and ${}_B F_B$ be the C*-correspondences associated with $\I$ and $\J$, respectively.
Furthermore let $(\Om_\I,\vpi)$ and $(\Om_\J,\psi)$) be the corresponding quantised dynamics.
The following are equivalent:
\begin{enumerate}
\item $\T_E^+$ and $\T_F^+$ are completely isometrically isomorphic;
\item $\T_E^+$ and $\T_F^+$ are isomorphic as topological algebras;
\item $(\Om_\I,\vpi)$ and $(\Om_\J,\psi)$ are $Q$-$P$-locally piecewise conjugate;
\item $E$ and $F$ are unitarily equivalent.
\end{enumerate}
If $E_0$ and $F_0$ are the minimal C*-correspondences associated with $\I$ and $\J$, and $(\Om_{0,\I},\vpi)$ and $(\Om_{0,\J},\psi)$ are the corresponding quantised dynamics, then any of the items above is equivalent to any of the items below:
\begin{enumerate}[resume]
\item $\T_{E_0}^+$ and $\T_{F_0}^+$ are completely isometrically isomorphic;
\item $\T_{E_0}^+$ and $\T_{F_0}^+$ are isomorphic as topological algebras;
\item $(\Om_{0,\I},\vpi)$ and $(\Om_{0,\J},\psi)$ are $Q$-$P$-locally piecewise conjugate;
\item $E_0$ and $F_0$ are unitarily equivalent.
\end{enumerate}
\end{corollary}

\begin{proof}
For the implication $[\textup{(iii)} \Rightarrow \textup{(iv)}]$ we will construct a matrix $[b_{ij}]$ that provides an invertible map between $E$ and $F$.
Let the projection $P_{[n]}$ for some $[n] \neq 0$.
Without loss of generality we may assume that
\[
P_{[n]} = R_1^* R_1 \dots R_k^* R_k (I- R_{k+1}^* R_{k+1}) \dots (I - R_d^* R_d).
\]
Then the $P_{[n]}$ is paired with some $Q_{[m]}$ with
\[
k = |\supp [n]| = |\supp [m]|.
\]
Since $\Om_\J$ is totally disconnected, so is $P_{[n]}$.
Hence we can replace the open cover of $P_{[n]}$ (given by the piecewise conjugacy on $P_{[n]}$) with a cover by compact subsets.
Now we can proceed as in the second part of the proof of \cite[Proposition 3.20]{DavKat11} to show that there is a partition into clopen sets $\V_\pi$ such that $\ga_s \vpi_i|_{\V_\pi} = \psi_{\pi(i)} \ga_s|_{\V_\pi}$.
As in \cite[Corollary 3.28]{DavKat11} let the $k \times k$ matrix $[b_{ij,[n]}]$ with $b_{ij,[n]} = \de_{i, \pi(i)} \chi_{\V_\pi}$ which intertwines
\[
\{\psi_i\}_{i \in \supp [n]} \qand \{\vpi_j\}_{j \in \supp [m]}.
\]
By direct computation we obtain
\[
[b_{ij,[n]}]^* [b_{ij,[n]}] = P_{[n]} \otimes I_k  = [b_{ij,[n]}] [b_{ij,[n]}]^*.
\]
Then the matrix
\[
[b_{ij}] = \sum_{[n] = 1}^d P_{[n]} \otimes I_k \cdot [b_{ij,[n]}]
\]
defines the required unitary C*-correspondence map.

The implications $[\textup{(iv)} \Rightarrow \textup{(i)} \Rightarrow \textup{(ii)}]$ are immediate.
To show that item (ii) implies item (iii) fix a bounded algebraic isomorphism $\Phi \colon \T_E^+ \to \T_F^+$.
Then $\Phi$ implies a homeomorphism between the character spaces, say $\ga_c$.
The restriction of every character of $\T_F^+$ to $B$ is a point evaluation.
Our first objective is to show that $\ga_c$ induces a homeomorphism $\ga_s \colon \Om_\J \to \Om_\I$ by collapsing the characters into the single points.
To this end recall that the algebra $\T_E^+$ is generated by the $T_i$ and the $a \in A$ separately, and that $a T_i = T_i \al_i(a)$.
These are the requirements for applying \cite[Lemmas 3.7 and 3.9]{DavKat11} to obtain that the maximal analytic sets $\B_\om$ in the character space are parameterized by the $\om \in \Om_\I$ so that
\[
\B_\om := \{ \theta \in \fM_{\T_E^+} \mid \theta|_{C(\Om_\I)} = \ev_\om\}.
\]
Furthermore each $\B_\om$ is homeomorphic to a ball of dimension equal to the number of the $\vpi_i$ for which $\vpi_i(\om) = \om$.
Since $\ga_c$ maps maximal analytic sets onto maximal analytic sets, then the required $\ga_s$ is well defined (and a homeomorphism).
In particular we get that if $\theta$ is a character of $\T_F^+$ such that $\theta|_B = \ev_y$ then $\theta \Phi|_A = \ev_x$ with $\ga_s(y) = x$.

Fix $y$ in the support of the $P_{[n]}$ with $k = |\supp [n]|$.
Here we do not exclude the case $k=0$.
Without loss of generality we may assume that
\[
P_{[n]} = R_1^* R_1 \dots R_k^* R_k (I- R_{k+1}^* R_{k+1}) \dots (I - R_d^* R_d).
\]
Define the orbit representation $\pi_y \colon C(\Om_J) \to \B(K_y)$ where $K_y = \oplus_{\mu \in M^*} \bC$ by
\[
\pi_y(g) = \diag\{ g \psi_\mu(y) \mid \mu \in M^*\}.
\]
Define also $V_{y,i} \equiv V_y(R_i) \colon K_y \to K_y$ by
\[
V_y(R_i) e_\nu = \begin{cases} e_{i\nu} & \text{ if } i\nu \in M^*, \\ 0 & \text{ otherwise}. \end{cases}
\]
Then $(V_y \times \pi_y)$ defines a completely contractive representation of $\T_F^+$.
Let $E$ denote the compression onto the subspace generated by $\{e_\mt, e_1, \dots, e_{d'}\}$ and let $\{E_{ij} \mid i,j = 0,1,\ldots, d'\}$ denote the standard matrix basis in $M_{d'}(\bC)$.
Then
\[
E (V_y \times \pi_y)(R_i g) = g\psi_i(y)  E_{i0} \foral i=1, \dots, k,
\]
whereas
\[
E (V_y \times \pi_y)(R_i g) = 0 \foral i=k+1, \dots, d'.
\]
The former holds by definition and the latter holds since
\begin{align*}
E (V_y \times \pi_y)(R_i g)
& =
E (V_y \times \pi_y)(R_i R_i^*R_i g) \\
& =
E (V_y \times \pi_y)(R_i g \cdot R_i^*R_i) \\
& =
E (V_y \times \pi_y)(R_i g) E (V_y \times \pi_y)(R_i^*R_i) \\
& =
g \psi_i(y)  E_{i0} \left(0 \cdot E_{00} + \sum_{i=1}^{d'} k_i  E_{ii}\right) = 0,
\end{align*}
where the $k_i$ are some values of the function $R_i^* R_i$ along the orbit of $y$.
Consequently if we write $\Phi(Q) = g_0 + (\sum_{i=1}^{d'} \Bw_i g_i) + Z$ for $Q \in \T_E^+$ then we get
\[
E (V_y \times \pi_y) \Phi(Q) =
\left[
\begin{array}{ccc|ccc}
g_0(y) & \dots & 0 & 0 & \dots & 0 \\
\vdots & \ddots & \vdots & \vdots & \dots & \vdots \\
g_k(y) & \dots & g_0 \psi_k(y) & 0 & \dots & 0 \\
\hline
0 & \dots & 0 & g_0 \psi_{k+1}(y) & \dots & 0 \\
\vdots & \vdots & \vdots & \vdots & \ddots & \vdots \\
0 & \dots & 0 & 0 & \dots & g_0 \psi_{d'}(y)
\end{array}
\right].
\]
By compressing even further to the subspace generated by $\{e_\mt, e_1, \dots, e_k\}$ we have set the right context to apply locally the ideas from the appendix.
By proceeding as in Claim 3 of the appendix and thereon we derive item (iii).
Note that if $y \in P_{0}$ then we cannot have that $\ga_s(y) \in Q_{[m]}$ for any $[m] \neq 0$.
Therefore $\ga_s(y) \in Q_{0}$.

Similar arguments show that the items (v), (vi), (vii) and (viii), are equivalent.
Then Proposition \ref{P: Q0-P0} gives the equivalence of item (vii) with item (iii) and the proof is complete.
\end{proof}

\begin{question}\label{Q: aut cont}
We ask whether Corollary \ref{C: classification} holds at the level of algebraic isomorphisms between $\T_E^+$ and $\T_F^+$.
\end{question}

The next example shows that the commutative C*-algebras $A$ and $B$ cannot distinguish between the quantised dynamics related to the ideals $\I$ and $\J$.
Neither can the C*-modules $E_A$ and $F_B$.
It is only after taking into consideration all this structure \emph{and} the left action that we can distinguish between the quantised dynamics at hand.

\begin{example}\label{E: two ideals}
Consider the ideals $\I = \sca{x_1 x_2, x_2x_1}$ and $\J = \sca{x_1^2,x_2^2}$ in $\bC\sca{x_1, x_2}$, as in Example \ref{E: qd conjugacy}.
Let $T_{1}, T_{2}$ (resp. $R_{1}, R_{2}$) be the operators corresponding to $\I$ (resp. to $\J$) and ${}_A E_A$ (resp. ${}_B F_B$) be the associated C*-correspondence.

Then $T_{1}^* T_{1}$ is the projection on the span of $\{e_\mt, e_{1^n} \mid n \geq 1\}$, and $T_{2}^* T_{2}$ is the projection on the span of $\{e_\mt, e_{2^n}\}$.
For every $\mu \neq \mt$, we find that $T_\mu^* T_\mu = T^*_{1} T_{1}$ if $\mu$ has no $2$'s in it,  $T_\mu^* T_\mu = T_{2}^*T_{2}$ if $\mu$ has no $1$'s in it, and $T^*_\mu T_\mu = 0$ otherwise.
Similarly, when $\mu \neq \mt$ then $R^*_\mu R_\mu$ is the projection onto the span of $\{e_\mt, e_{1}, e_{12}, e_{1 2 1}, \dots\}$ if $\mu$ is an alternating sequence of $1$'s and $2$'s ending with $2$.
On the other hand $R^*_\mu R_\mu$ is the projection onto the span of $\{e_\mt, e_{2}, e_{21}, e_{2 1 2}, \dots\}$ if $\mu$ is an alternating sequence of $1$'s and $2$'s ending with $1$; and it is $0$ otherwise.

Then we may identify both $A$ and $B$ with the continuous functions on $\{0, 1, 2\}$, with $T_{1}^* T_{1}$ and $R_{1}^* R_{1}$ corresponding to $\chi_{\{0,1\}}$ and with $T_{2}^* T_{2}$ and $R_{2}^* R_{2}$ corresponding to $\chi_{\{0,2\}}$.
The inner product map $(\ga,U)$ given by
\[
U(\la T_{1} + k T_{2}) = \la R_{1} + k R_{2}
\]
and the $*$-isomorphism
\[
\ga(\la T_{1}^* T_{1} + k T_{2}^*T_{2}) = \la R_{1}^* R_{1} + k R_{2}^*R_{2}
\]
is a unitary equivalence of the Hilbert C*-modules $E$ and $F$.

However, these two Hilbert C*-modules are not unitarily equivalent as C*-correspondences.
Indeed, the quantised dynamical systems of this example were calculated in Example \ref{E: qd conjugacy}, and we see that the resulting graphs obtained by Remark \ref{R: qd graph} are not isomorphic.
Let us also provide the following direct operator theoretic argument to show this.
Write
\[
U(T_{j}) = R_{1} b_{1j} + R_{2} b_{2j} \qfor j = 1, 2,
\]
where $b_{1j} \in R_{1}^* R_{1} B$ and $b_{2j} \in R_{2}^*R_{2} B$.
We must also have that
\[
\ga(T^*_{j} T_{j}) = \langle UT_{j}, UT_{j} \rangle = b_{1j}^*b_{1j}+ b_{2j}^*b_{2j} .
\]
Furthermore the left covariance of $U$ means that
\[
0 = U(T_{1}^* T_{1} \cdot T_{2}) = \ga(T_{1}^*T_{1}) \left( R_{1} b_{12} + R_{2} b_{22} \right).
\]
It must be that $\ga(T_{j}^* T_{j})$ is the sum of two minimal projection, so it is one of  $R_{1}^*R_{1}$, $R_{2}^* R_{2}$ or $1 - R_{1}^*R_{1} R_{2}^* R_{2}$.

First assume that $\ga(T_{1}^* T_{1}) = R_{1}^* R_{1}$. Since
\[
R_{1}^* R_{1} \cdot R_{1} = R_{1} \beta_1(R_{1}^* R_{1})  = R_{1} R_{1}^*R_{1}^*R_{1}R_{1}= 0,
\]
and
\[
R_1^* R_1 \cdot R_2 = R_2 \beta_2(R_1^*R_1) = R_2 R_2^*R_1^*R_1 R_2 = R_2,
\]
we must have that $b_{22} = 0$.
Therefore we get that
\[
\ga(T_2^*T_2) = b_{12}^* b_{12} \in R_1^*R_1 B.
\]
However $\ga(T_2^* T_2)$ is either $R_2^* R_2$ or $1 - R_1^* R_1 R_2^* R_2$, neither of which is in $R_1^* R_1 B$.
This contradiction shows that $\ga(T_1^* T_1)$ cannot be $R_1^* R_1$.
On the other hand the flips
\[
R_1 \leftrightarrow R_2 \qand R_1^* R_1 \leftrightarrow R_2^* R_2
\]
extend to an automorphism of $F$ since $\J$ is symmetrical.
Therefore $\ga(T_1^* T_1)$ cannot be $R_2^* R_2$, either.
Finally we have that $\ga(T_1^* T_1) = 1 - R_1^* R_1 R_2 ^* R_2$ is also not an option.
Indeed in this case $\ga(T_2^* T_2)$ should be either $R_1^* R_1$ or $R_2^* R_2$.
This is impossible by symmetry.
\end{example}

\subsection{Automatic continuity}

We are able to answer Question \ref{Q: aut cont} for a specific class of monomial ideals, for which algebraic isomorphisms onto their tensor algebras are automatically continuous.
The arguments follow a trick that Donsig, Hudson and Katsoulis \cite{DHK01} established following ideas of Sinclair \cite{Sin75}.
Let $\phi \colon \A \to \B$ be an algebraic epimorphism for the Banach algebras $\A$ and $\B$.
The discontinuity of $\phi$ is quantified by the ideal
\[
\S(\phi) : = \{b \in \B \mid \exists (a_n) \subseteq \A \text{ such that } a_n \to 0 \text{ and } \phi(a_n) \to b \}.
\]
By the closed graph theorem then $\phi$ is continuous if and only if $\S(\phi) = (0)$.

\begin{lemma}[Sinclair]\label{L: Sinclair}
Let $\phi \colon \A \to \B$ be an algebraic epimorphism between the Banach algebras $\A$ and $\B$, and let a sequence $(b_n)$ in $\B$.
Then there exists an $N \in \bN$ such that
\[
\ol{b_1 b_2 \dots b_N \S(\phi)} = \ol{b_1 b_2 \dots b_n \S(\phi)}
\]
and
\[
\ol{\S(\phi) b_n b_{n-1} \dots b_1} = \ol{\S(\phi) b_{N} b_{N-1} \dots b_1}
\]
for all $n \geq N$.
\end{lemma}

\begin{proposition}\label{P: auto cont}
Let $E$ be a C*-correspondence associated to a monomial ideal $\I$ in $\bC\sca{x_1, \dots, x_d}$.
Suppose that for every $\nu \in \La^*$ there are words $w, z \in \La^*$ such that $w^n z \nu \in \La^*$ for all $n \in \bN$.
Then an algebraic epimorphism $\phi \colon \A \to \T_E^+$ for any Banach algebra $\A$ is continuous.
\end{proposition}

\begin{proof}
To reach contradiction let $0 \neq x \in \S(\phi)$.
By the Fourier transform we may assume that $x = T_\mu a$ for some $\mu \in \La^*$ and $a \in A$.
Moreover we take $\mu$ to be of minimal length, i.e. if $T_\nu a \in \S(\phi)$ for $|\nu| < |\mu|$ then $T_\nu a = 0$.
Let $\nu_1, \nu_2 \in \La^*$ such that $\sca{T_\mu a e_{\nu_1}, e_{\nu_2}} \neq 0$.
By assumption let $w, z \in \La^*$ such that $w^n z \nu_2 \in \La^*$ for all $n \in \bN$.
Then we obtain
\[
\sca{T_{w}^n T_{z\mu} a e_{\nu_1}, e_{w^n z \nu_2}} = \sca{T_\mu a e_{\nu_1}, T_z^* T_{w^n}^* e_{w^n z \nu_2}} = \sca{T_\mu a e_{\nu_1}, e_{\nu_2}} \neq 0.
\]
Consequently $T_{w}^n T_{z \mu} a \neq 0$ for all $n \in \bN$ and in particular $T_{z \mu} a$ is a non-trivial element in $\S(\phi)$.
By applying Lemma \ref{L: Sinclair} for $b_n = T_w$ we get that
 \[
T_w^{N} T_{z \mu} a \in \ol{T_w^{n} \S(\phi)}
\]
for all $n \geq N$.
In particular this is true for $n = N+ |z| + 1$.
However the Fourier coefficients of the elements in $\ol{T_w^{n} \S(\phi)}$ are supported on words with length at least $N + |z| + 1 + |\mu|$, whereas $T_w^{N} T_{z \mu} a$ is supported on a word with length exactly $N + |z| + |\mu|$.
This leads to the contradiction $T_w^{N} T_{z \mu} a = 0$.
\end{proof}

\begin{remark}
The assumption in Proposition \ref{P: auto cont} can be verified for monomial ideals that are either of finite or infinite type.
An example in the case of infinite type is provided by the monomial ideal generated by $\{x y^n x \mid n \in \bN\}$ of $\bC\sca{x,y}$.

On the other hand there are monomial ideals of infinite type for which the assumption in Proposition \ref{P: auto cont} does not hold.
For example suppose that the monomial ideal in $\bC\sca{x,y}$ is such that the allowable words are either of the form
\[
xyxy^2xy^3\dots y^{n}x \, , \,
xyxy^2xy^3 \dots x y^n \, , \,
\]
and
\[
yxyx^2yx^3\dots x^n y \, , \,
yxyx^2yx^3 \dots yx^n \, , \,
\]
for all $n \in \bZ_+$, or sub-words of these words.
\end{remark}

\begin{question}\label{Q: aut cont 2}
In view of Question \ref{Q: aut cont} we ask whether the assumption in Proposition \ref{P: auto cont} holds for all monomial ideals of finite type.
\end{question}

\subsection{Local conjugacy and C*-algebras}

Next we show how local conjugacy affects the related C*-algebras.
Recall that local conjugacy coincides with unitary equivalence by Corollary \ref{C: classification}.
By the remarks preceding Proposition \ref{P: isom rel} we have that local conjugacy implies $*$-isomorphisms of the Toeplitz-Pimsner and the Cuntz-Pimsner algebras, as well as complete isometric isomorphisms of the tensor algebras.
Below we show that the same is true for the C*-algebras $\ca(T)$ and $\ca(R)$, and their quotients by the compacts.

\begin{corollary}\label{C: rel CP}
Let $\I \lh \bC\sca{x_1, \dots, x_d}$ and $\J \lh \bC\sca{y_1, \dots, y_{d'}}$ be monomial ideals.
If the quantised dynamics are $Q$-$P$-locally piecewise conjugate then $\ca(T) \simeq \ca(R)$, and moreover $\ca(T) / \K(\F_X) \simeq \ca(R) / \K(\F_Y)$.
\end{corollary}

\begin{proof}
Let us denote by $t_\mu$ and $r_\nu$ the generators in $\T_E^+$ and $\T_F^+$.
Local conjugacy implies a unitary equivalence $(\ga, U) \colon E \to F$, in which case we get that $d = d'$ by Proposition \ref{P: support}.
Let $[b_{ij}]$ be the associated matrix and let the $*$-isomorphism $\Phi \colon \T_E \to \T_F$ such that
\[
\Phi(a) = \ga(a) \foral a \in A, \qand \Phi(t_j) = \sum_{i=1}^d r_i b_{ij} \foral j=1, \dots, d.
\]
By Proposition \ref{P: C*(T)} the C*-algebras $\ca(T)$ and $\ca(R)$ are respectively quotients of $\T_E$ and $\T_F$.
By Proposition \ref{P: isom rel} then $\Phi$ implements a $*$-isomorphism $\Phi' \colon \ca(T) \to \ca(R)$ if it maps the ideal generated by $\{I - t_\mu^* t_\mu \mid \mu \in \La^*\}$ in $A$ into the ideal generated by $\{I- r_w^*r_w \mid w \in M^*\}$ in $B$.

For $\mu = j \in \{1, \dots, d\}$ we compute
\[
\Phi(I - t_j^* t_j)  = I - \sum_{k,i=1}^d b_{ij}^* r_i^* r_k b_{kj} = I - \sum_{i=1}^d b_{ij}^* r_i^* r_i b_{kj}.
\]
However by unitary equivalence we have that $r_j^* r_j = \sum_{i=1}^d b_{ij}^* b_{ij}$, therefore
\begin{align*}
\Phi(I - t_j^*t_j)
& =
I - r_j^*r_j + \sum_{i=1}^d b_{ij}^* b_{ij} - \sum_{i=1}^d b_{ij}^* r_i^* r_i b_{ij} \\
& =
I - r_j^* r_j + \sum_{i=1}^d b_{ij}^* \left(I - r_i^*r_i \right) b_{ij}.
\end{align*}
Hence $\Phi(I -t_j^* t_j)$ is in the required ideal.
If $\mu = j l$ then we may write
\begin{align*}
I - t_{jl}^* t_{jl}
& =
I - t_l^*t_l + t_l^*t_l - t_l^*t_j^* t_j t_l
 =
I - t_l^* t_l + t_l^*(I - t_j^* t_j) t_l.
\end{align*}
It suffices to restrict our attention to $t_l^*(I - t_j^* t_j) t_l$.
We then compute
\begin{align*}
\Phi(t_l^*(I - t_j^* t_j) t_l)
& =
\Phi(t_l)^* \Phi(I - t_j^* t_j) \Phi(t_l)
 =
\sum_{i,k=1}^d b_{il}^* r_i^* x r_k b_{ik},
\end{align*}
where $x$ is the sum of elements $b^*(I - r_w^* r_w)c$ for some $b,c \in B$ and some words $w \in M^*$.
By using the covariant relation we get that
\begin{align*}
b_{il}^* r_i^* b^* (I - r_w^* r_w) c r_i b_{ik}
& =
b_{il}^* \be_i(b)^* (r_i^* r_i - r_i^*r_w^*r_w r_i)\be_i(c) b_{ik}\\
& =
-b_{il}^* \be_i(b)^* (I - r_i^* r_i)\be_i(c) b_{ik} + \\
& \hspace{2cm} +
b_{il}^* \be_i(b)^*(I - r_{wi}^* r_{wi})\be_i(c) b_{ik}
\end{align*}
which shows that $\Phi(t_l^*(I - t_j^* t_j) t_l)$ is in the required ideal.
Therefore $\Phi(I - t_{jl}^* t_{jl})$ is in the required ideal.
Induction on the length of the words $\mu$ then finishes the proof of $\ca(T) \simeq \ca(R)$.

On the other hand Proposition \ref{P: C*(T)} implies that the quotient $\ca(T)/ \K(\F_X)$ is the $A$-relative Cuntz-Pimsner algebra of $E$.
Trivially $\ga(A)=B$ and Proposition \ref{P: isom rel} finishes the proof of the second part.
\end{proof}

\begin{remark}
A similar analysis can be carried out for the $q(E)$ and $q(F)$ of Remark \ref{R: quotient corre}.
The reason is that the unitary equivalence mappings are defined by a family of $*$-algebraic relations.

In particular we have that if $E$ and $F$ are unitarily equivalent then $q(E)$ and $q(F)$ are unitarily equivalent as well.
Indeed if $E$ and $F$ are injective then this is immediate by the fact that the restriction of $q$ to $\T_E^+$ is completely isometric and because the $*$-isomorphism $\Phi \colon \ca(T) \to \ca(R)$ is such that $\Phi(a) = \ga(a)$ and $\Phi(t_j) = \sum_{i=1}^d r_i b_{ij}$.
If $E$ is not injective then so is $F$ and in particular $\ga(\ker\phi_E) = \ker\phi_F$.
However Proposition \ref{P: kernel} implies that $\ker\phi_E = \bC P_\mt$ and $\ker\phi_F = \bC P_\mt$.
Consequently we obtain that $\ga(P_\mt) = P_\mt$.
Therefore $\Phi(\K(\F_X)) = \K(\F_Y)$ and $\ca(T)/ \K(\F_X)$ is $*$-isomorphic to $\ca(R)/\K(\F_Y)$ by some $\Phi'$ with $\Phi'(q(a)) = q(\ga(a))$ and $\Phi'(q(t_j)) = \sum_{i=1}^d q(r_i) q(b_{ij})$.
Then $\Phi'|_{q(A)}$ is a $*$-isomorphism onto $q(B)$ and the $q(b_{ij})$ form a matrix that corresponds to a C*-correspondence map.
The fact that $[b_{ij}]$ defines a unitary C*-correspondence map between $E$ and $F$ is translated into some $*$-algebraic relations, e.g. $\sum_{k=1}^d b_{ik} b_{jk}^* = \de_{i,j} t_i^*t_i$.
Applying $\Phi'$ to these we obtain a set of $*$-algebraic relations which give that the C*-correspondence map associated with $[q(b_{ij})]$ is unitary.

The converse of this phenomenon does not hold.
For a counterexample let $E$ be associated with the monomial ideal $\I$ generated by
\[
\begin{pmatrix}
& x_1 x_2 \\
 x_2x_1 & x_2^2
\end{pmatrix}
\]
in $\bC\sca{x_1, x_2}$, and let $F$ be associated with the monomial ideal $\J$
\[
\begin{pmatrix}
& x_1 x_2 & x_1x_3 \\
 x_2x_1 & x_2^2 & x_2x_3 \\
x_3x_1 & x_3x_2 & x_3^2
\end{pmatrix}
\]
in $\bC\sca{x_1, x_2, x_3}$.
Then $E$ and $F$ cannot be unitarily equivalent as they are related with rings on a different number of symbols.
However $q(E)$ and $q(F)$ are both $\bC$ over $\bC$ since there exists only one infinite word in both cases, i.e. the sequence $(x_1)$, and so they are trivially unitarily equivalent.
\end{remark}

\section{Encoding via subproduct systems}\label{S: class sps}

We turn our attention to the classification of our data by using the subproduct systems.
Our aim is to show that subproduct systems and their tensor algebras form a complete invariant for monomial ideals up to a permutation of the generators.
We will require some useful algebraic facts.

\begin{lemma}\label{L: algebraic}
Let $\I$ be a monomial ideal of $\bC\sca{x_1,\dots,x_d}$.
Then the group of the graded automorphisms of $\bC\sca{x_1,\dots,x_d}/\I$ is a linear algebraic group.
\end{lemma}

\begin{proof}
Without loss of generality we may assume that $\I$ does not contain any of the $x_i$.
Otherwise we restrict our attention to $\bC\sca{x_1,\dots,x_n}/\I'$ where $x_1, \dots, x_n \notin \I$, $x_{n+1}, \dots, x_d \in \I$ and $\I' = \I \cap \bC\sca{x_1,\dots,x_n}$.
Let $z_i = x_i + \I$ be the generators of the quotient.
Then the embedding
\[
\phi \mapsto [a_{ij}], \text{ where } \phi(z_j) = \sum_{i=1}^d a_{ij} z_i,
\]
from the graded automorphisms of $\bC\sca{x_1,\dots,x_d}/\I$ inside $GL_d(\bC)$ is an injective group homomorphism.
Injectivity is immediate as the $z_i$ generate the quotient.
To see that $[a_{ij}]$ is indeed invertible, recall that $\phi$ is graded hence its restriction to the linear span of the $z_i$ is onto itself.
Then we get
\[
f(\sum_{i=1}^d a_{i1} x_i, \dots, \sum_{i=1}^d a_{id}x_i) = 0 \foral f \in \I,
\]
and at the same time
\[
f(\sum_{i=1}^d b_{i1} x_i, \dots, \sum_{i=1}^d b_{id}x_i) = 0 \foral f \in \I,
\]
where $[b_{ij}] = [a_{ij}]^{-1}$.
On the other hand, if $[a_{ij}] \in GL_d(\bC)$ satisfies the above equations then it readily defines an automorphism of $\bC\sca{x_1,\dots,x_d}/\I$.

Let the ring $(\bC[y_{11}, \dots, y_{dd}])\sca{x_1, \dots, x_d}$ of polynomials on the noncommuting variables $x_i$ with co-efficients from $\bC[y_{11}, \dots, y_{dd}]$ so that $y_{ij} \cdot x_k = x_k \cdot y_{ij}$.
Then for any polynomial $f \in \I$ we have that
\[
f(\sum_{i=1}^d y_{i1} x_i, \dots, \sum_{i=1}^d y_{id}x_i) = \sum_{\un{x}^\mu \notin \I \, , \, |\mu| \leq \deg f} g_{f,\mu}(y_{11}, \dots, y_{dd}) \un{x}^\mu
\]
for some polynomial expressions $g_{f,\mu}$ in $\bC[y_{11}, \dots, y_{dd}]$.
Since the $\un{x}^\mu$ form a basis we obtain that
\[
f(\sum_{i=1}^d y_{i1} x_i, \dots, \sum_{i=1}^d y_{id}x_i) = 0 \iff g_{f,\mu}(y_{11}, \dots, y_{dd}) = 0 \foral g_{f,\mu}.
\]
Therefore we have that an invertible matrix $[a_{ij}]$ satisfies
\[
f(\sum_{i=1}^d a_{i1} x_i, \dots, \sum_{i=1}^d a_{id}x_i) = 0 \foral f \in \I,
\]
if and only if the tuple $(a_{11}, \dots, a_{dd}) \in \bC^{d^2}$ is a solution for the system of polynomials $\{g_{f,\mu} \mid f \in \I, \un{x}^\mu \in \I\}$.
The latter set may be infinite.
However the $g_{f,\mu}$ are polynomials in $\bC[y_{11},\dots,y_{dd}]$, hence we can find a finite set of such polynomials with the same set of solutions.
In a similar way if $[b_{ij}] = [a_{ij}]^{-1}$ then
\[
f(\sum_{i=1}^d b_{i1} x_i, \dots, \sum_{i=1}^d b_{id}x_i) = 0 \foral f \in \I,
\]
if and only if the $(b_{11}, \dots, b_{dd}) \in \bC^{d^2}$ is a solution for the system of polynomials $\{g_{f,\mu}' \mid f \in \I, \un{x}^\mu \in \I\}$.
This can be substituted again by a finite set of monomials with the same set of solutions.
Recall that the $b_{ij}$ are polynomial expressions of the $a_{ij}$.
Hence we may view the $g_{f,\mu}'$ as polynomials on the $y_{11}, \dots, y_{dd}$ as well.
Therefore we have that an invertible matrix $[a_{ij}]$ defines a graded automorphism of $\bC\sca{x_1,\dots,x_d}/\I$ if and only if the $d^2$-tuple $(a_{ij})$ of its entries forms a solution for a finite set of polynomials, and the proof is complete.
\end{proof}

Now we are in position to apply the arguments of \cite[Theorem 5.27]{BruGub08} to achieve the required classification.

\begin{theorem}\label{T: class sps}
Let $X$ and $Y$ be subproduct systems associated with the homogeneous ideals $\I \lhd \bC\sca{x_1,\dots,x_d}$ and $\J \lhd \bC\sca{y_1,\dots,y_{d'}}$.
Without loss of generality suppose that $x_i \notin \I$ and $y_j \notin \J$ for all $i,j$.
The following are equivalent:
\begin{enumerate}
\item $\A_X$ and $\A_Y$ are completely isometrically isomorphic;
\item $\A_X$ and $\A_Y$ are algebraically isomorphic;
\item $\bC\sca{x_1,\dots,x_d}/\I$ and $\bC\sca{y_1,\dots,y_{d'}}/\J$ are algebraically isomorphic by a graded isomorphism;
\item $X$ and $Y$ are similar;
\item $X$ and $Y$ are isomorphic;
\item $d=d'$ and there is a permutation on the variables $y_1, \ldots, y_d$ such that $\I$ and $\J$ are defined by the same words.
\end{enumerate}
\end{theorem}

\begin{proof}
The implications [(vi) $\Rightarrow$ (v) $\Rightarrow$ (iv)], [(iv) $\Rightarrow$ (iii)], [(iv) $\Rightarrow$ (ii)] and [(vi) $\Rightarrow$ (i)] are immediate.
Moreover the implication [(i) $\Rightarrow$ (v)] is shown in Theorem \ref{T: sps cla}.
To finish the proof we will show that [(ii) $\Rightarrow$ (iii) $\Rightarrow$ (vi)].

Suppose that item (ii) holds.
Then by Lemma \ref{L: vac pre} we may suppose that there is a graded isomorphism $\phi \colon \A_X \to \A_Y$.
As in the proof of \cite[Proposition 6.17]{DorMar14} we get a family of isomorphisms
\[
V_n:= p_n \phi|_{X(n)} \colon X(n) \to Y(n)
\]
where we identify $X(n)$ with $t_\infty(X(n))$.
In particular each $V_n$ implies an isomorphism between $\{f + \I \mid \deg f = n \}$ and $\{g + \J \mid \deg g =n\}$.
Applying to $X(1) = \bC^d$ we obtain that $d = d'$.
Because the family of $V_n$ respects the associative product rule of the subproduct systems, it extends to a graded algebraic isomorphism between $\bC\sca{x_1,\dots,x_d}/\I$ and $\bC\sca{y_1,\dots,y_d}/\J$.

Suppose that item (iii) holds for an algebraic isomorphism $\phi$.
Let $T_\I$ be the subgroup of the graded automorphisms of $\bC\sca{x_1,\dots,x_d}/\I$ defined by $x_i \mapsto \la_i x_i$, i.e. $T_\I$ is a torus.
Now $T_\I$ is a maximal torus inside $GL_d(\bC)$, thus it is also a maximal torus inside the group of the graded automorphisms of $\bC\sca{x_1,\dots,x_d}/\I$.
The group of automorphisms of $\bC\sca{x_1,\dots,x_d}/\J$ contains two maximal tori; the $T_\J$ and the $\phi T_\I \phi^{-1}$.
By Lemma \ref{L: algebraic} we can apply Borel's Theorem \cite[Corollary 11.3 (1)]{Bor91} which states that all maximal tori in an algebraic group are conjugate.
Hence we obtain a graded automorphism $\rho$ of $\bC\sca{x_1,\dots,x_d}/\J$ such that $\rho^{-1} T_\J \rho = \phi T_\I \phi^{-1}$.
By substituting $\phi$ with $\rho \phi$ we have that there exists a graded isomorphism
\[
\phi \colon \bC\sca{x_1,\dots,x_d}/\I \to \bC\sca{x_1,\dots,x_d}/\J
\]
such that $\phi T_\I = T_\J \phi$.
If $[a_{ij}]$ is the invertible matrix related to $\phi$ then we obtain that for every diagonal matrix $D_1$ there is a diagonal matrix $D_2$ such that $[a_{ij}] D_1 = D_2 [a_{ij}]$.

By the Leibniz formula for the determinant there is a permutation $\pi \in S_d$ such that $\Pi_{j=1}^d a_{\pi(j) j} \neq 0$.
Let $A_\pi$ be the corresponding permutation matrix and set $[b_{ij}] = A_\pi [a_{ij}]$.
Since $[a_{ij}] D_1 = D_2 [a_{ij}]$ for the diagonal matrices $D_1$ and $D_2$ then we may write
\[
[b_{ij}] D_1 = A_\pi [a_{ij}] D_1 = A_\pi D_2 [a_{ij}] = A_\pi D_2 A_\pi^{-1} [b_{ij}].
\]
The matrix $A_\pi D_2 A_\pi^{-1}$ is again diagonal, as it is produced by permuting the diagonal elements of $D_2$ by $\pi$.
Hence for every $\la = (\la_1, \dots, \la_d) \in \bC^d$ there exists an $r \equiv r(\la) = (r_1, \dots, r_d) \in \bC^d$ such that
\[
b_{ij} \la_j = r_i b_{ij} \foral i,j = 1, \dots, d.
\]
Since $b_{ii} \neq 0$ we get that $\la_i = r_i$ for all $i=1, \dots, d$.
By choosing non-zero $\la_i \in \bC$ such that $\la_i \neq \la_j$ for $i \neq j$ we get that $b_{ij} = 0$ for $i \neq j$.
Hence the matrix $[b_{ij}]$ is diagonal.
Consequently the matrix $[a_{ij}]$ associated with the graded isomorphism
\[
\phi \colon \bC\sca{x_1,\dots,x_d}/\I \to \bC\sca{x_1,\dots,x_d}/\J
\]
is diagonal up to a permutation $\pi$ of the rows.
Therefore we get that $\phi(x_i + \I) = a_{\pi(i)i} y_{\pi(i)} + \J$ with $a_{\pi(i)i} \neq 0$, which completes the proof.
\end{proof}

\begin{remark}
Theorem \ref{T: class sps} reads the same also for monomials in commuting variables, with the proof following verbatim.
\end{remark}

\begin{remark}\label{R: alt class}
There is a direct proof of the implication [(iv) $\Rightarrow$ (vi)] of Theorem \ref{T: class sps} that does not pass through item (iii).
By Theorem \ref{T: sps cla}, this proof suffices also to give Theorem \ref{T: class sps} for bounded isomorphisms.

Recall that similarity in particular implies that $d=d'$.
Let $\{e_1, \ldots, e_d\}$ and $\{f_1, \ldots, f_d\}$ be the canonical orthonormal bases for $X(1)$ and $Y(1)$, respectively.
Fix a similarity $V=(V_n)$ for which we obtain
\[
V_1 e_j = \sum_{i=1}^d v_{ij} f_i \qfor j=1, \ldots, d.
\]
Now $V_1$ may not be a permutation and we quantify this by defining
\[
\Q(m) = \{i \in\{1, \dots, d\} \mid v_{im} \neq 0\} \qfor m=1, \dots, d.
\]
Thus $V_1$ permutes the basis elements if and only if every $\Q(m)$ is a singleton.

\smallskip

\noindent {\bf Claim.} {\it If $m_1\dots m_n$ is a forbidden word in $X$ then $i_1\dots i_n$ is a forbidden word in $Y$ for all $(i_1,\dots,i_n) \in \Q(m_1)\times \dots \times \Q(m_n)$.}

\smallskip

\noindent {\it Proof of Claim.}
Recall from Section \ref{S: SPS} that $V_n p_n = q_n V_1^{\otimes n}$.
Therefore $V_1^{\otimes n}$ maps $X(1)^{\otimes n} \ominus X(n)$ into $Y(1)^{\otimes n} \ominus Y(n)$.
Hence, if $m_1\dots m_n \in\mathfrak{F}^X_n$ then we get that $V_1^{\otimes n} (e_{m_1} \otimes \cdots \otimes e_{m_n})$ is spanned by some $f_\nu$ with $\nu \in \mathfrak{F}^Y_n$.
For every $(i_1, \ldots, i_n) \in \Q(m_1) \times \cdots \times \Q(m_n)$, there corresponds a nonzero summand in this expansion, namely
\[
v_{i_1 m_1} \cdots v_{i_n m_n} f_{i_1} \otimes \cdots \otimes f_{i_n} .
 \]
It follows that $f_{i_1} \otimes \cdots \otimes f_{i_n}$  is not in $Y(n)$ for each $(i_1, \ldots, i_n) \in \Q(m_1) \times \cdots \times \Q(m_n)$.
Hence we must have that $f_{i_1} \otimes \cdots \otimes f_{i_n} \perp Y(n)$.
Therefore $i_1 \dots i_n$ is a forbidden word, and the proof of the claim is complete.

\smallskip

Since $V_1$ is an isomorphism then $[v_{ij}]$ is in $GL_d(\bC)$.
By the Leibniz formula for the determinant there is a permutation $\pi \in S_d$ such that $v_{i\pi(i)} \neq 0$ for all $i=1, \dots, d$.
This permutation defines an automorphism on $Y$, hence we may assume that $v_{ii} \neq 0$ up to a permutation on the $\{y_1, \dots, y_d\}$.
Then by the claim above we get that $\mathfrak{F}^X_n \subseteq \mathfrak{F}^Y_n$ for all $n$.
This means that after applying a permutation on the $\{y_1, \dots, y_d\}$ we have that $\I \subseteq \J$.
By symmetry and simple dimensional considerations, it follows that $\I = \J$ after permuting the variables.

With a little more care, such arguments could also settle algebraic homomorphisms.
This requires a combination with a weaker notion of similarity $V = (V_n)$ where a uniform bound for the $V_n$ is not provided.
We leave the details to the interested reader.
\end{remark}

\begin{remark}
In Remark \ref{R: alt class} we begin with any similarity and show that it induces a permutation of the elements.
On the other hand, in the proof of [(iv) $\Rightarrow$ (vi)] of Theorem \ref{T: class sps} we isolate a similarity which is shown to be diagonal up to a permutation of the columns.
Thus it is natural to ask whether all similarities are of this form.
Equivalently if the automorphisms of a subproduct system $X$ are toric up to a permutation.
This is not true and here is a counterexample that settles this for more general homogeneous ideals beyond the monomial ones.

Recall that if $\I = (0)$ then any unitary defines an automorphism of $\fA_d$.
Therefore if $\I$ is generated by polynomials in the first symbols $x_1, \dots, x_n$ from $\{x_1, \dots, x_d\}$ with $n <d$, then any matrix $V=I_n \oplus U$ with $U$ a unitary in $M_{d-n}(\bC)$ defines an automorphism of $\A_X$ for $X = X_\I$.
Then $U$ can be chosen so that $V$ is not diagonal up to a permutation.
\end{remark}

\begin{remark}
Theorem \ref{T: class sps} implies that two tensor algebras $\A_X$ and $\A_Y$ are isomorphic as algebras if and only if they are isomorphic as topological algebras.
However we do not claim that every algebraic isomorphism is automatically continuous.
We are able to show automatic continuity under the assumption that for every $\nu \in \La^*$ there are $w,z \in \La^*$ such that $w^n z \nu \in \La^*$ for all $n \in \bZ_+$.
The proof follows verbatim the proof of Proposition \ref{P: auto cont}.
\end{remark}

\begin{remark}\label{R: rig}
Item (iii) in Theorem \ref{T: class sps} shows that rigidity of subproduct systems is connected to rigidity phenomena for the quotients $\bC\sca{x_1,\dots,x_d}/\I$ and $\bC\sca{y_1,\dots,y_d}/\J$.
This is an exceptional behaviour, and for general homogeneous ideals this is far from being true.
In particular in the commutative case there are examples of ideals $\I$ and $\J$ in $\bC[x_1, \dots, x_d]$ such that $\bC[x_1, \dots, x_d]/ \I$ is isomorphic to $\bC[x_1, \dots, x_d]/ \J$, but $\A_X$ is not (algebraically) isomorphic to $\A_Y$.
Such an example appears in \cite[Example 8.6]{DRS11}.
These examples can be pulled back to ideals in noncommuting variables by considering homogeneous ideals that contain the commutator ideal.
\end{remark}

From Theorem \ref{T: class sps} we derive that if $\A_X$ and $\A_Y$ are isomorphic then $\T_E^+$ and $\T_F^+$ are isomorphic as well, where $E$ and $F$ are the C*-correspondences related to $\I$ and $\J$, respectively.
Indeed equality of the monomial ideals up to a permutation produces a unitary equivalence by permuting the generators.
The following counterexample shows that the converse fails.

\begin{example}
Consider the monomial ideals $\I$ and $\J$ in $\bC\sca{x_1, x_2, x_3, x_4}$:

\medskip

\begin{center}
\begin{tabular}{c c}
$\begin{pmatrix}
x_1x_1 & x_1x_2 & x_1x_3 & \\
x_2x_1 & x_2x_2 & & x_2x_4 \\
x_3x_1 & x_3x_2 & x_3x_3 & & \\
x_4x_1 & x_4x_2 & & x_4x_4
\end{pmatrix}$
$\quad$ & $\quad$
$\begin{pmatrix}
x_1x_1 & x_1x_2 & x_1x_3 & x_1x_4 \\
x_2x_1 & x_2x_2 & & \\
x_3x_1 & x_3x_2 & x_3x_3 & & \\
x_4x_1 & x_4x_2 & & x_4x_4
\end{pmatrix}$
\\
\vspace{-.25cm}
\\
generating set for $\I$
$\quad$ & $\quad$
generating set for $\J$
\end{tabular}
\end{center}

\medskip

\noindent Then the resulting graphs from the quantised dynamics are:
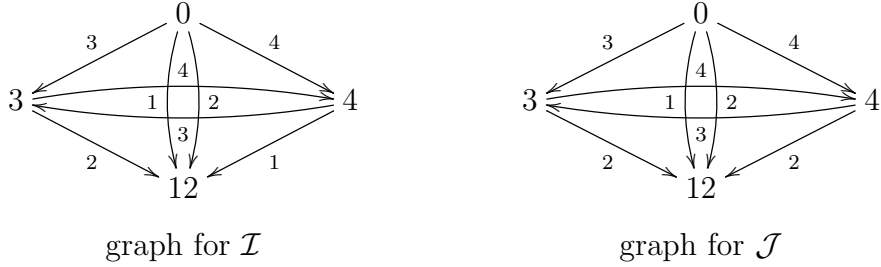
\[
\xymatrix@R=2mm{
& 0 \ar[ddl]_3 \ar[ddr]^4 \ar@/_/[dddd]_1 \ar@/^/[dddd]^2 & & &  & 0 \ar[ddl]_3 \ar[ddr]^4 \ar@/_/[dddd]_1 \ar@/^/[dddd]^2 & \\
& & & & & & \\
3 \ar@/^/[rr]^4 \ar[ddr]_2 &   & 4 \ar@/^/[ll]^3 \ar[ddl]^1 & & 3 \ar@/^/[rr]^4 \ar[ddr]_2 &  & 4 \ar@/^/[ll]^3 \ar[ddl]^2  \\
& & & & & & \\
& 12 & & & & 12 & \\
&\text{graph for $\I$} & & & & \text{graph for $\J$} &
}
\]
where $0 = [\mt]_1$, $12 = [1]_1 = [2]_1$, $3 = [3]_1$ and $4=[4]_1$.
The only difference is in the right bottom arrow: in one graph it is labeled $1$ (meaning that $\vpi_1$ maps $4$ to $1$) and in the other it is labeled $2$.
The two graphs are isomorphic (they are the same when removing labels) hence the quantised dynamical systems are $Q$-$P$-locally piecewise conjugate.
However, the dynamical systems are not conjugate, and there is no permutation taking $\I$ onto $\J$.
Thus $\T_E^+$ and $\T_F^+$ are completely isometrically isomorphic, while $\A_X$ and $\A_Y$ are not even algebraically isomorphic.
\end{example}

\section{Comparison with other constructions}\label{S: compare}

There are several possible ways in which to associate an operator algebra with a given monomial ideal $\I$.
In this section we show how some natural candidates (several of which have been considered in the literature) differ from the algebras that we are considering herein.
In the last subsection we will show that when $\I$ comes from a sofic subshift, and $E$ is the C*-correspondence that we have associated with $\I$, then $\O_E$ can be constructed as the {\em graph C*-algebra} of the {\em follower set graph} of the subshift.

\subsection{Graph constructions}

It is natural to associate a certain graph with every monomial ideal.

\begin{definition}
The \emph{left graph $\G_{\I,l}$ of $\I$} (resp. the \emph{right graph $\G_{\I,r}$ of $\I$}) is the subgraph of the directed left (resp. right) Cayley graph of $\bF_+^d$ consisting of the vertices $\La^*$ and the edges connecting them.
\end{definition}

Therefore the set of vertices of the graph $\G_{\I,l}$ is the set of legal words, and there is a directed edge from $\mu$ to $\nu$ if and only if there is some $i \in \{1, \dots, d\}$ such that $\nu = i \mu$.
We caution the reader not to confuse the left graph with the graph $\G_\I$ constructed from the quantised dynamics in Remark \ref{R: qd graph}.

\subsubsection*{Quiver algebras}

From every directed graph $\G$ one may construct the so-called \emph{quiver algebra} $\T^+(\G)$ \cite{KatKri04, Sol04} (\emph{quiver} is just another way of saying \emph{directed graph}).
This is the tensor algebra of the C*-correspondence related to $\G$.
It was proved in \cite{KatKri04} that two quiver algebras are isomorphic as topological algebras if and only if the underlying graphs are isomorphic.
Therefore, a reasonable way to encode the information in $\I$ is to form the quiver algebra $\T^+(\G_{\I,l})$ of the left graph $\G_{\I,l}$ of the ideal $\I$.

In general, the quiver algebra $\T^+(\G_{\I,l})$ differs from the nonselfadjoint operator algebras $\T_E^+$ and $\A_X$.
To see this, recall Example \ref{E: two ideals}.
For that example we consider the ideals $\I = \sca{x_1x_2, x_2x_1}$ and $\J = \sca{x_1^2, x_2^2}$ in $\bC\sca{x_1, x_2}$ which give rise to the non unitarily equivalent C*-correspondences $E$ and $F$.
Thus $\T_E^+$ and $\T^+_F$ are not isomorphic as topological algebras by Corollary \ref{C: classification}.
On the other hand, the graphs $\G_{\I,l}$ and $\G_{\J,l}$ are the same as they consist of two directed infinite paths with a common source.
Thus the quiver algebras $\T^+(\G_{\I,l})$ and $\T^+(\G_{\J,l})$ are isometrically isomorphic.
Therefore at least one of the quiver algebras must be different as a topological algebra from the tensor algebras.

In fact Proposition \ref{P: auto cont} applies in this case to conclude that any algebraic isomorphism from a Banach algebra onto $\T_E^+$ is continuous.
It then follows from the above argument that, in general, the quiver algebra $\T^+(\G_{\I,l})$ and the tensor algebra $\T_E^+$ are not even algebraically isomorphic.

Using the same example together with Theorem \ref{T: class sps} we find that the tensor algebras $\A_X$ do not coincide (even just as algebras) in general with the quiver algebras.
The same is true when considering the right graphs due to the symmetry of $\I$ and $\J$.

\subsubsection*{Graph C*-algebras with the notation of \cite{Rae05}}

On the other hand, from every graph $\G$ one may also construct its \emph{graph C*-algebra} $\ca(\G)$.
In this section we use the recent terminology as it appears in \cite{Rae05}.
Again this algebra is different from the C*-algebras that we consider here.

The example provided by $\I = \sca{x_1x_2, x_2x_1}$ in $\bC\sca{x_1, x_2}$ does the job.
Indeed, the left graph of $\G_{\I,l}$ can be drawn as follows
\[
\xymatrix{
\cdots & \bullet^{v_{-2}} \ar[l]^{f_{-3}} & \bullet^{v_{-1}} \ar[l]^{f_{-2}}  & \bullet^{v_0} \ar[l]^{f_{-1}} \ar[r]_{f_1} & \bullet^{v_1} \ar[r]_{f_2} & \bullet^{v_2} \ar[r]_{f_3} & \cdots
},
\]
where the central vertex corresponds to the empty word.
Let $\{e_n\}$ be the standard o.n. basis of $\ell^2(\bZ)$.
Then a Cuntz-Krieger family $\{P_n, L_m \mid n \in \bZ, m \in \bZ^*\}$ is given by the following finite rank operators
\begin{align*}
P_n = \Theta_{e_n,e_n} \foral n \in \bZ,
\qand
L_m =
\begin{cases}
\Theta_{e_{m-1},e_{m}} & \text{ for } m \geq 1, \\
\Theta_{e_{m+1},e_{m}} & \text{ for } m \leq -1,
\end{cases}
\end{align*}
where $\Theta_{x,y}(z) = \sca{z,x} y$ (in contrast to what the symbol $\Theta$ means for C*-correspon\-dences).
By the gauge invariant uniqueness theorem, this Cuntz-Krieger family integrates to a faithful representation of the graph algebra.

The graph C*-algebra $\ca(\G_{\I,l})$ is generated by compact operators hence it contains only compact operators.
Since $\Theta_{e_m,e_n} \in \ca(\G_{\I,l})$ for all $m,n \in \bZ$, we find $\ca(\G_{\I,l}) = \K(\ell^2(\bZ))$.
On the other hand the graph algebra of $\ca(T)$ is irreducible and contains a non compact operator, thus it cannot be $*$-isomorphic to $\ca(\G_{\I,l})$ (it is not CCR).
The graph algebra is also not isomorphic to $\ca(T)/\K$, as the latter is $*$-isomorphic to $C(\bT) \oplus C(\bT)$ being generated by two partial isometries $v,u$ satisfying $v^*v = v v^* \perp u^* u = u u^*$.

The same is true when considering the graph algebra of $\G_{\I,r}$ which coincides with $\G_{\I,l}$ due to the symmetry of $\I$.

\subsubsection*{Graph C*-algebras with older notation}

The notation in graph algebras has changed recently.
Previously the role of the source in the Cuntz-Krieger equations was played by the range and vice versa.
Here we mention that even in this case the C*-algebras are again different.

Indeed, let again $\I = \sca{x_1x_2, x_2x_1}$ in $\bC\sca{x_1, x_2}$ with the graph
\[
\xymatrix{
\cdots & \bullet^{v_{-2}} \ar[l]^{f_{-3}} & \bullet^{v^{-1}} \ar[l]^{f_{-2}}  & \bullet^{v_0} \ar[l]^{f_{-1}} \ar[r]_{f_1} & \bullet^{v_1} \ar[r]_{f_2} & \bullet^{v_2} \ar[r]_{f_3} & \cdots
},
\]
where the central vertex corresponds to the empty word.
Let $\{e_n\}$ be the standard o.n. basis of $\ell^2(\bZ)$, let $H = \ell^2(\bZ) \oplus \bC$.
Denote by $f$ a unit vector in the summand $\bC$.
Then a Cuntz-Krieger family $\{P_n, L_m \mid n \in \bZ, m \in \bZ^*\}$ on $H$ is given by the following finite rank operators $L_{-1} = \Theta_{e_{-1},f}$ and
\begin{align*}
P_n =
\begin{cases}
\Theta_{e_{0},e_0} + \Theta_{f,f}  & \text{ if } n = 0,\\
\Theta_{e_{n},e_n} & \text{ for } n \neq 0,
\end{cases}
\text{ and  } \,
L_m =
\begin{cases}
\Theta_{e_{m},e_{m-1}} & \text{ for } m \geq 1, \\
\Theta_{e_{m},e_{m+1}} & \text{ for } m \leq -2, \\
\end{cases}
\end{align*}
In particular this Cuntz-Krieger family integrates to a faithful representation of the graph algebra.
In this case the graph algebra is again $\K \oplus \K$.
Indeed, the right branch of the graph generates a C*-algebra isomorphic to $\K(\ell^2(\bN))$, and the left branch generates another copy of $\K(\ell^2(\bN))$ that is orthogonal to the first copy.
Reasoning as above we arrive at the conclusion that the ``old'' graph C*-algebras are also different from the C*-algebras that we consider.

\subsection{Dynamical systems}\label{Ss: dyn sys}

Let $\La$ be a two-sided subshift, as described in Section \ref{Ss: subshifts}.
Then we have the topological dynamical system $(\La,\si)$ defined by the left shift $\si$ on $\La$.
Reasonable candidates to encode this system are the C*-crossed product $C(\La) \times_\si \bZ$ and the (nonselfadjoint) semicrossed product $C(\La) \rtimes_\si \bZ_+$ \cite{Pet84}.

\subsubsection{Semicrossed product}

The semicrossed product $C(\La) \times_\si \bZ_+$ is defined as the universal nonselfadjoint operator algebra generated by $\Bv^n f$ for $f \in C(\La)$ and $n \in \bZ_+$, where $\Bv$ is a contraction such that $f \cdot \Bv = \Bv f\si$.
Consequently for any $\la = (x_n) \in \La$ we have that the mapping
\[
\Phi(\sum_n \Bv^n f_n) = \begin{bmatrix} f_0(\la) & 0 \\ f_1(\la) & f_0 \si(\la) \end{bmatrix}
\]
defines a completely contractive representation of $C(\La) \times_\si \bZ_+$.
Furthermore conjugate subshifts have completely isometric isomorphic semicrossed products.
A stronger converse is given by Davidson and Katsoulis \cite{DavKat08}: algebraic isomorphism of semicrossed product implies conjugacy of the associated C*-dynamics.

Recall that conjugate subshifts may be defined on a different number of symbols.
Since the number of symbols is an invariant for the tensor algebra $\A_X$ related to $\La$, then $\A_X$ cannot be algebraically isomorphic to $C(\La) \times_\si \bZ_+$.

On the other hand suppose that $\T_E^+$ and $C(\La) \times_\si \bZ_+$ are algebraically isomorphic.
Then we may proceed as in Claim 6 of the appendix and find a column $(c_{11}, \dots, c_{d1})^t \in \bC^d$ that is a left invertible matrix.
However this is a contradiction unless $d=1$.
In this (only) case both algebras $\T_E^+$ and $C(\La) \times_\si \bZ_+$ are simply the disc algebra $A(\bD)$.

\subsubsection{Crossed product}

Furthermore we compare $C(\La) \rtimes_\si \bZ$ with $\ca(T)$ and $\ca(T)/\K(\F_X)$.
This provides also a comparison with $\O_E$.
Since $\O_E$ is a quotient of $\T_E$ this provides also a comparison with $\T_E$.
In particular we claim that in general the C*-crossed product differs from these C*-algebras.

For the first counterexample let $\I$ be the monomial ideal in $\bC\sca{x_1, x_2}$ generated by $\{x_1x_2, x_2x_1\}$.
Then the system $(\La,\si)$ is identified with the identity map on two points.
Therefore $C(\La) \rtimes_\si \bZ \simeq C(\{0,1\}) \otimes C(\bT)$, thus it is commutative.
However both $\O_E \simeq \ca(T)$ and $\T_E$ for this example are not commutative.
Indeed we have that $T_1^* T_2 = 0$ whereas $T_2 T_1^* e_1 = e_2$.

For the second counterexample let $\I$ be the trivial zero ideal in $\bC\sca{x_1, x_2}$.
Then $\O_E$ is the Cuntz algebra $\O_2$ on two generators.
However the system $(\La,\si)$ contains fixed points.
For example let the point $(x_n)$ with $x_n = 1$ for all $n \in \bZ$.
Therefore $C(\La) \rtimes_\si \bZ$ contains non-trivial (Fourier-invariant) ideals and cannot be $*$-isomorphic to the simple C*-algebra $\O_E \simeq \ca(T)/ \K(\F_X)$.

\subsection{Subshift constructions}\label{Ss: subshift compare}

In the following presentation we will follow as much as possible the notation of each work to facilitate comparison.
We point out that it should not be confused with the notation we have fixed for our analysis.

\subsubsection{Matsumoto's approach \cite{Mat97}}

Given a two-sided subshift $(\La,\si)$ Matsumoto \cite{Mat97} builds the C*-algebras $\ca(T)$ and $\ca(T)/ \K(\F_X)$.
The latter is denoted by $\O_\La$ and it is the main subject in Matsumoto's work.
As we have illustrated $\O_E$ is not $\ca(T)/\K(\F_X)$ in general.
Even more the restriction of the quotient map on the space generated by the $T_i$ (or on the C*-algebra $A$) is not isometric unless $E$ is injective.
This follows by the remarks in Section \ref{S: corre} and Proposition \ref{P: dichotomy}.

For a concrete example (and ad-hoc arguments) consider the forbidden words $\{12,21\}$ in the shift space $\{1,2\}$.
Indeed in this case we get that $T_1^*T_1 + T_2^*T_2 = I + P_{\mt}$ therefore
\begin{align*}
\nor{T_1 + T_2}^2 = \nor{T_1^*T_1 + T_2^*T_2} = 2,
\end{align*}
whereas
\begin{align*}
\nor{q(T_1 + T_2)}^2 & = \nor{q(T_1^*T_1 + T_2^*T_2)} = 1,
\end{align*}
for the quotient mapping $q \colon \ca(T) \to \ca(T)/ \K(\F_X)$.

\subsubsection{Carlsen's approach \cite{Car08}}

Carlsen \cite{Car08} revisited the C*-algebras that arise from a right subshift $\mathsf{X}$.
His approach is directed in giving a C*-algebra that has an additional universal property \cite[Introduction]{Car08}.
To do this he constructs an injective C*-correspondence $\mathsf{H}_\mathsf{X}$ of which he takes the relative Cuntz-Pimsner algebra with respect to the ideal $\phi_{\mathsf{H}_\mathsf{X}}^{-1}(\K(\mathsf{H}_\mathsf{X}))$, following the work of Pimsner \cite{Pim97} and Schweizer \cite{Sch01}.
At this point we would like to inform the reader that we could not trace reference no.15 of \cite{Sch01}; this reference is essential for the proof of the main theorem of \cite{Sch01}.

Carlsen's C*-correspondence $\mathsf{H}_\mathsf{X}$ is over the C*-algebra $\wt{D}_\mathsf{X}$, which differs from $A$ of Proposition \ref{P: A is AF}; therefore $\mathsf{H}_\mathsf{X}$ differs from $E$.
Let us give the definition of $\wt{D}_\mathsf{X}$.
On the topological space $\mathsf{X}$ define the sets
\[
C(\mu, \nu) = \{ \nu w \in \mathsf{X} \mid \mu w \in \mathsf{X}, w \in \mathsf{X} \},
\]
and let $\wt{D}_\mathsf{X}$ be the C*-subalgebra of the bounded functions on $\mathsf{X}$ generated by the characteristic functions on $C(\mu,\nu)$.

In the particular case of the forbidden words $\{12, 21\}$ on the symbol space $\{1, 2\}$ we see that $\mathsf{X}$ consists of two points and $\wt{D}_\mathsf{X} = C(\{0,1\})$.
However for the same subshift the C*-algebra $A$ of Proposition \ref{P: A is AF} is $C(\{0,1,2\})$.

\subsubsection{Carlsen-Matsumoto approach \cite{CarMat04}}

Matsumoto \cite{Mat99}, and Carlsen and Matsumoto \cite{CarMat04} focus on the C*-algebra $q(A)$ for the canonical quotient map $q \colon \ca(T) \to \ca(T)/ \K(\F_X)$.
In his early work \cite{Mat99} Matsumoto gives a description for the compact Hausdorff space that identifies the commutative C*-algebra $q(A)$.
Later Carlsen and Matsumoto \cite{CarMat04} revisited this claim which they show to be incorrect in general.
They make several points.

First of all they consider right subshifts $X_\La$ that arise in the following way.
Given a two-sided subshift $(\La,\si)$ on the symbol space $\{1, \dots, d\}$ let
\[
X_\La = \{(x_1, x_2, \dots) \mid \exists (y_n) \in \La . y_n = x_n \foral n \geq 1\}.
\]
This is the ``positive part'' of the elements in $\La$.
For $l \geq 0$ they define the equivalence relation $\sim_l$ on $X_\La$ by
\[
\mu \sim_l \nu \Leftrightarrow \{w \in \La^*_l \mid w \mu \in X_\La\} = \{w \in \La^*_l \mid w \nu \in X_\La\}.
\]
Define $\Om_l : = X_\La / \sim_l$ and note that there is an onto mapping $\Om_{l+1} \to \Om_l$. Then let $\Om$ be the projective limit of $\Om_l$ with respect to these onto maps.
It worths comparing these definitions with the ones in Section \ref{Ss: qd} where we consider elements in $\La^*$ instead of $X_\La$.

Secondly they show \cite[Lemma 2.2]{CarMat04} that $C(\Om)$ coincides with $\ca(S_\mu^* S_\mu \mid \mu \in \La^*)$ where the $S_i$ act on the Hilbert space $H = \sca{ e_{(x_n)} \mid (x_n) \in X_\La}$ by
\[
S_i e_{(x_n)}= \begin{cases} e_{(i,(x_n))} & \qif (i,(x_n)) \in X_\La, \\ 0 & \qotherwise. \end{cases}
\]

Moreover they identify \cite[Lemma 2.9]{CarMat04} the C*-algebra $q(A)$ with the continuous functions on another space $\Om^*$.
Let $\Om^*_l$ be the set of the equivalence classes on the finite words $\mu$ in $\La_l^*$ for which the set $\{w \in \La^* \mid w \sim_l \mu\}$ is infinite.
Then $q(A)$ is identified with the continuous functions on the projective limit of the $\Om_l^*$.

Furthermore they show \cite[Corollary 3.3]{CarMat04} that the mapping $q(T_i) \mapsto S_i$ defines a $*$-isomorphism between $\ca(T)/\K(F_\La)$ and the C*-algebra $\ca(S)$ generated by the $S_i$, under the assumption that the subshift satisfies condition (*) and condition (I).
Condition (I) is crucial for the results in \cite{CarMat04}.
For example \cite[Corollary 3.3]{CarMat04} is true for the full shift on $\Si=\{1,2\}$.
But it fails to be true in general.
For the forbidden words $\{12, 21\}$ on the symbol set $\{1, 2\}$ we see that $S_1 = S_1^*S_1$ since $X_\La$ consists of the points $(11\dots)$ and $(22\dots)$.
If there was a $*$-isomorphism such that $q(T_i) \to S_i$ then $\ca(S)$ would admit a gauge action, say $\{\be_z\}_{z \in \bT}$.
Then
\[
0 = \be_z(S_1 - S_1^*S_1) = z S_1 - S_1^*S_1 = (z-1)S_1,
\]
for every $z \in \bT$ which leads to the contradiction $S_1=0$.

The interested reader should be warned here that several results on subshifts hold for $\ca(S)$ whereas other hold for $\ca(T)/\K(F_\La)$.
Carlsen and Matsumoto provide a very illuminating discussion and description of their results in the introduction of \cite{CarMat04}.

\subsection{$\O_E$ as a graph C*-algebra}\label{Ss: OE as a graph algebra}

Let $\Lambda$ be a two-sided sofic subshift.
If $\I$ is the  monomial ideal generated by forbidden words in $\La$, 
then we can form the C*-algebra $\O_E$ for the associated C*-correspondence $E$.
We thus obtain a new C*-algebra that is constructed out of a subshift $\La$, and one may ask whether this algebra is a reasonable one to consider.
Our goal in this section is  to show that $\O_E$ arises from $\La$ via two well known and natural constructions: roughly speaking, it is the {\em graph C*-algbera} of the {\em follower set graph} of $\La$.
The same analysis can be carried out for one sided sofic subshifts.

Let us introduce the follower set graph of a subshift $\La$.
A useful reference for this material is \cite[Chapter 3]{LinMar95}, but we warn the reader that we reverse some of the notation.
If $\La$ is a two-sided subshift, then the {\em follower set} of $\mu \in \La^*$ is the set
\[
F_\La(\mu) := \{w \in \La^* \mid w \mu \in \La^*\}.
\]
Note that different allowable words can have the same follower set.
In fact, when $\La$ is sofic then there are only finitely many follower sets \cite[Theorem 3.2.10]{LinMar95}.
The follower set graph is then defined to be the labeled graph whose vertices are parameterized by the follower sets, and there is exactly one edge labeled $i$ from $F(\mu)$ to $F(i \mu)$ when $i \mu$ is allowable.
We allow the empty word to have its own follower set --- this may or may not coincide with the follower set of another word.

As in Remark \ref{R: sofic}, when $\La$ is sofic then there exists some $k$ such that $\Om_k = \Om_n = \Om$ for all $n \geq k$ and so the quantized dynamics are defined on a discrete space.
In \cite[Section 5]{BarKak} it was observed that one may identify the follower set graph with the labeled graph of the quantised dynamics (as in Remark \ref{R: graph}).
That is, the quantised dynamics, when viewed as a labeled graph, give us the familiar follower set graph of a subshift.

Now, from the labeled follower set graph of $\La$ we obtain a directed unlabeled graph by simply erasing all the labels.
Note here that we do not identify different edges, we just forget about their labels.
Let us call this graph {\em the underlying graph of the follower set graph}, and let it be denoted by $G = (G^0, G^1)$.
Thus, $G^0$ is the finite set $\Om$, and $G^1$ consists of all pairs of points $(u,v)  \in \Om \times \Om$, for which there is some map $\varphi_i$ with $\varphi_i(u) = v$.

In Section \ref{Ss: topological} we saw that the associated C*-correspondence $E$ coincides with Katsura's \cite{Kat04-top} topological graph determined by the quantised dynamical system.
But under the assumption that $E$ comes form a sofic two-sided subshift, this topological graph is just the finite underlying graph of the follower set graph $G$ discussed in the previous paragraph.
So $E$ is the C*-correspondence of $G$, and hence we conclude (using \cite[Proposition 3.10]{Kat03}) that $\O_E \simeq \ca(G)$, i.e, $\O_E$ is the graph C*-algebra of the follower set graph.

Since the complete details of the isomorphism $\O_E \simeq \ca(G)$ are hard to find in the literature, and also because our notation differs from that used by Katsura in \cite{Kat03}, we pause to justify and make explicit this isomorphism.

\begin{proposition}
Let $E$ be the C*-correspondence of a sofic subshift $\La$, and let $G = (G^0, G^1)$ be the underlying unlabeled graph of the follower set graph of $\La$.
Then $\O_E$ is $*$-isomorphic to $\ca(G)$.
\end{proposition}
\begin{proof}
Let $(\pi,t)$ be the universal covariant representation of $(A,E)$ in $\O_E$ and define the families
\[
\P = \{\pi(a) \mid a \textrm{ is a minimal projection in } A\} ,
\]
and
\[
\S = \{t(T_i a) \mid i = 1, \ldots, d \textrm{ ; } a \leq T_i^*T_i \textrm{ ; } a\,  \textrm{  is a minimal projection in } A\}.
\]
Since $\pi(T_i^* T_i a) = t(T_i a)^* t(T_ia)$, then $t(T_i a) \neq 0$ if and only if $a \leq T_i^* T_i$, when $a$ is minimal.
As the minimal projections of $A$ sum up to the identity we have that
\[
t(T_i) = \sum \{ t(T_i a) \mid a \leq T_i^*T_i \textrm{ ; } a\,  \textrm{  is a minimal projection in } A \}.
\]
We will show that the family $(\P,\S)$ is a Cuntz-Krieger family for $G$.
As $(\P,\S)$ is generating for $\O_E$ and admits a gauge action, the gauge invariant uniquesness theorem (for graph C*-algebras) will then show that $C^*(G)$ and $\O_E$ are $*$-isomorphic.

The minimal projections in the finite dimensional algebra $A$ are precisely the characteristic functions of points in $\Om = G^0$.
Hence the family $\P$ consists of mutually orthogonal (nonzero) projections corresponding to the vertices in $G^0$.

Next, we shall show that every element in $\S$ is a nonzero partial isometry corresponding to an edge in $G^1$.
Let $a$ be the minimal projection corresponding to a point $v_a \in G^0 = \Om$.
If $a \leq T^*_i T_i$, then $v_a \in \Om^i$.
Thinking of $T_i$ as in the picture given in Section \ref{Ss: topological},
we have that $T_i$ corresponds to the characteristic function of all the edges labeled $i$ emitted from $\Om^i$ (see the proof of Proposition \ref{P: E top cor} and the discussion above it).
Therefore, $T_i a$ is the characteristic function of an edge $e \in G^1$, (which originally had a label $i$) coming out of the point $s(e) = v_a \in \Om^i$, and going into some point $\varphi_i(v_a) = r(e)$.
Moreover, as $a \leq T_i^* T_i$, we find that $a^*T^*_i T_i a = a^* a = a\neq 0$, so we get
\[
t(T_i a)^* t(T_i a) = \pi(\langle T_i a, T_i a \rangle)  = \pi(a^*T_i^* T_i a) = \pi(a) .
\]
This shows that all elements of $\S$ are partial isometries which satisfy the first Cuntz-Pimsner relation
\[
S_e^* S_e = P_{s(e)}.
\]
It remains to show the second Cuntz-Pimsner relation, namely
\[
\sum_{e \in r^{-1}(v)} S_e S_e^* = P_v \qfor  \,\,\, |r^{-1}(v)| \neq 0.
\]
Fix a vertex $v \in G^0 = \Om$ that is not source and let $b \in A$ denote the minimal projection corresponding to this point.
We need to show that
\[
\sum_{\varphi_i(v_a) =  v} t(T_i a) t(T_i a)^*  = \pi(b),
\]
where $v_a \in G^0$ is the point corresponding to a minimal projection $a \in A$, such that $(v_a,v) \in G^1$.
We sum over all $i$ and $v_a$ such that $\phi_i(v_a) = v$, and this amounts to summing over all edges $e \in G^1$ for which $v = r(e)$.
Now, by definition of $\psi_t$ (see Section \ref{S: corre}),
\[
\sum_{\varphi_i(v_a) =  v} t(T_i a) t(T_i a)^*  = \psi_t (\sum_{\varphi_i(v_a) =  v}\Theta_{T_i a, T_i a} ).
\]
We will now show that
\[
\sum_{\varphi_i(v_a) =  v}\Theta_{T_i a, T_i a} = \phi_E(b).
\]
To see this, we fix $\xi \in E$, and compute
\[
\sum_{\varphi_i(v_a) =  v}\Theta_{T_i a, T_i a} \xi = \sum_{\varphi_i(v_a) =  v} T_i a \langle T_i a, \xi \rangle.
\]
Now, by using the topological graph picture described in Section \ref{Ss: topological}, we obtain
\[
\langle T_i a, \xi \rangle (u) = \sum_{f\in s^{-1}(u)} \ol{T_i(f) a(s(f))}\xi(f) .
\]
This expression will be zero when $u \neq v_a$, and when $u = v_a$, we have that
\[
\langle T_i a, \xi \rangle (v_a) = \sum_{f\in s^{-1}(v_a)} \ol{T_i(f)}\xi(f) = \xi(e)
\]
for the one edge $e$ which is the unique edge leaving $v_a$ with label $i$.
Therefore $\Theta_{T_i a, T_i a}$ is the projection onto the subspace spanned by $T_i a$.
Now, recall that $T_i a$ is the characteristic function on of the unique edge leaving $v_a$ with label $i$, and that, by assumption, this edge must go into $v$.
This means that for every $f \in G^1$,
\[
\left[\sum_{\varphi_i(u) =  v} \Theta_{T_i a, T_i a} (\xi) \right](f)
=
\begin{cases} 0 & r(f) \neq v, \\
\xi(f) & r(f) = v. \end{cases}
\]
But the left action of $b$ on $\xi$ is given by
\[
\left[\phi_E(b) (\xi) \right](f) = b(r(f)) \xi(f),
\]
which has the same effect, since $b$ is the characteristic function of $v$.
We conclude that
$\sum_{\varphi_i(u) =  v} \Theta_{T_i a, T_i a} = \phi_E(b)$, as we set out to show.
Thus, by putting everything together and using covariance of the representation $(\pi,t)$, we obtain
\[
\sum_{\varphi_i(u) =  v} t(T_i a) t(T_i a)^*  = \psi_t \left(\phi_E(b)\right) = \pi(b)
\]
which completes the proof.
\end{proof}

\begin{remark}
It will be interesting to investigate the dependence of $\O_E \cong C^*(G)$ on the dynamical aspects of the topological dynamical system $(\La, \sigma)$.
We leave this for future work.
\end{remark}

\section{Appendix}

Davidson and Katsoulis \cite[Theorem 3.22]{DavKat11} show that algebraic isomorphisms of the tensor algebras of two classical multivariable systems implies piecewise conjugacy of the systems.
Let us provide here an alternative proof of this breakthrough result.
Our approach relies on appropriate compressions of the Fock representations.
The proof follows a series of steps and let us start by describing the objective.

Let $(X,\si) \equiv (X,\si_1,\dots, \si_d)$ be a classical system, i.e. $X$ is a locally compact Hausdorff space and every $\si \colon X \to X$ is a proper continuous map.
If $\mu = \mu_1 \dots \mu_d \in \bF_+^d$ we write $\si_\mu$ for the composition $\si_{\mu_1} \cdots \si_{\mu_n}$.
The tensor algebra $\T_{(X,\si)}^+$ is defined as the (universal) nonselfadjoint operator algebra generated by $\Bs_1 f, \dots, \Bs_d f$ with $f \in C_0(X)$, such that $[\Bs_1, \dots, \Bs_d]$ is a row contraction, and $f \Bs_j = \Bs_j f \si_j$ for all $f \in C_0(X)$ and $j=1, \dots, d$.
This is slightly different from \cite[Definition 1.2]{DavKat11} and the reader is addressed to Remark \ref{R: DK tensor}.
The universal property of $\T_{(X,\si)}^+$ suggests that if $\pi \colon C_0(X) \to \B(H)$ is a representation, and there is a row contraction $[s_1, \dots, s_d] \in \B(H^d,H)$ such that $f s_i = s_i f\si_i$ then the mapping $\Bs_i f \mapsto s_i \pi(f)$ extends to a completely contractive representation of $\T_{(X,\si)}^+$ in $\B(H)$.

Let $(Y, \tau) \equiv (Y, \tau_1, \dots, \tau_{d'})$ be a second classical system.
We will use the notation $\Bt_1 g, \dots, \Bt_{d'} g$ for the generators of $\T_{(Y,\tau)}^+$.
Fix an algebraic isomorphism
\[
\Phi \colon \T_{(X,\si)}^+ \to \T_{(Y,\tau)}^+.
\]
An important point to keep in mind is that $\Phi$ is automatically continuous \cite[Corollary 3.6]{DavKat11}.
This is not immediate and follows by the trick of Donsig, Hudson and Katsoulis \cite{DHK01}.
We aim to show that there is a homeomorphism $\ga_s \colon Y \to X$ and that for every $y\in Y$ there is a permutation $\pi \in S_d$ and a neighbourhood $\U_\pi$ such that
\[
\ga_s \tau_i|_{\U_\pi} = \si_{\pi(i)} \ga_s|_{\U_\pi} \foral i=1, \dots, d.
\]

Our first task is to find the homeomorphism $\ga_s \colon Y \to X$.
We write $\fM_{(X,\si)}$ for the character space of $\T_{(X,\si)}^+$.
If $\theta \in \fM_{(X,\si)}$ then $\theta|_{C_0(X)}$ is a character, and thus equals to $\ev_x$ for some $x \in X$.

\smallskip

\noindent {\bf Claim 1.} {\it Let $\theta$ be a character of $\T_{(X,\si)}^+$ such that $\theta|_{C_0(X)} = \ev_x$ for some $x \in X$.
Then $\theta$ is completely determined by
\[
(\la_1, \dots, \la_d) : = (\theta(\Bs_1 f), \dots, \theta(\Bs_d f))
\]
for any $f \in C_0(X)$ such that $\theta(f) = f(x) = 1$, in the sense that for every $\mu = \mu_1 \dots \mu_d \in \bF_+^d$ and $g \in C_0(X)$ we have
\begin{align*}
\theta(\Bs_\mu g) =
\begin{cases}
\la_\mu g(x) & \text{ if } \si_{\mu_i}(x) = x \foral \mu_i, \\
0 & \text{ otherwise},
\end{cases}
\end{align*}
where $\la_{\mu_1 \dots \mu_n} = \la_{\mu_1} \cdots \la_{\mu_n}$.
In particular $\la_i = 0$ if $\si_i(x) \neq x$.}

\smallskip

\noindent {\it Proof of Claim 1.} First we compute
\[
\theta(\Bs_i g) = \theta(\Bs_i g) \theta(f) = \theta(\Bs_i f g ) = \theta(\Bs_i f) \theta(g),
\]
which implies that the tuple $(\la_1, \dots, \la_d)$ does not depend on the choice of the function $f$.

Secondly observe that if $\si_i(x) \neq x$ then $\la_i = 0$.
Indeed let $h \in C_0(X)$ such that $h \si_i(x) = 0$ and $h(x) = 1$ for which we have
\[
\la_i = \theta(\Bs_i f) = \theta(h) \theta(\Bs_i f) = \theta( \Bs_i f) \theta(h \si_i) = 0.
\]

Suppose that the claim is true for all words of length less than $k$ and let $\mu = \mu_1 \ldots \mu_{k+1}$ be of length $k+1$.
If $\si_{\mu_i}(x) = x$ for all $i$ then we can write $\mu = w \nu$ for a word $w \neq \mt$ of length strictly less than $k +1$ such that $\si_w(x) = x$, and $\si_\nu(x)=x$.
Then we get that
\begin{align*}
\theta(\Bs_\mu g) = \theta(f) \theta(\Bs_\mu g) = \theta(\Bs_w f \si_w) \theta(\Bs_\nu g) = \la_w f(\si_w(x)) \la_\nu g(x) = \la_\mu g(x),
\end{align*}
by the inductive hypothesis.
On the other hand let $\mu_n$ be the first letter from left to right in the word $\mu$ for which $\si_{\mu_1 \dots \mu_n}(x) \neq x$.
If $n<k + 1$ then let $w = \mu_1 \dots \mu_n$ and $\nu = \mu_{n+1} \dots \mu_{k+1}$ for which we have that $\si_w(x) \neq x$.
Let $h \in C_0(X)$ such that $h(x) = 1$ and $h \si_w(x) = 0$ and compute
\begin{align*}
\theta(\Bs_\mu g) = \theta(h) \theta(\Bs_\mu g) = \theta(\Bs_w h \si_w) \theta(\Bs_\nu g) = \la_w h(\si_w(x)) \theta(\Bs_\nu g) = 0.
\end{align*}
If $n = k+1$ then the word $\mu = \mu_1 \dots \mu_{k+1}$ is such that
\[
x =\si_{\mu_1}(x) = \si_{\mu_1 \mu_2}(x) = \dots = \si_{\mu_1 \dots \mu_{k}}(x)
\]
but $x \neq \si_{\mu}(x)$, so it follows that $\si_{\mu_{k+1}}(x) \neq x$.
Therefore we get that $\theta(\Bs_{\mu_{k+1}} g) = 0$ and so
\[
\theta(\Bs_\mu g) = \theta(f) \theta(\Bs_\mu g)= \theta(\Bs_{\mu_1 \dots \mu_{k}} f\si_{\mu_1} \cdots \si_{\mu_k}) \theta(\Bs_{\mu_{k+1}} g) = 0.
\]
This ends the proof of Claim 1. \hfill{$\blacksquare$}

\smallskip

The next claim establishes the connection between the evaluation functionals on $X$ and the evaluation functionals on $Y$.
The proof of the first part of the claim relies on \cite{DavKat11} and is included for completeness.
It is the second part of the claim that plays a central role in our analysis.

\smallskip

\noindent {\bf Claim 2.} {\it The homomorphism $\Phi$ induces a homeomorphism $\ga_s \colon Y \to X$.
In particular, if $\theta \in \fM_{(Y,\tau)}$ with $\theta|_{C_0(Y)} = \ev_y$ then $\theta \Phi \in \fM_{(X,\si)}$ with $\theta \Phi|_{C_0(X)} = \ev_{\ga_s(y)}$.}

\smallskip

\noindent {\it Proof of Claim 2.} The homomorphism $\Phi$ defines a homeomorphism
\[
\ga_c \colon \fM_{(Y, \tau)} \to \fM_{(X,\si)} : \theta \mapsto \theta \Phi.
\]
It is not immediate but follows as in \cite[Lemma 3.9]{DavKat11} that the maximal analytic sets inside $\fM_{(X,\si)}$ are precisely the sets $\{ \theta \in \fM_{(X,\si)} \mid \theta|_{C_0(X)} = \ev_x\}$ parameterized by the points $x \in X$.
In particular each one is homeomorphic to a ball of dimension equal to the number of the $\si_i$ that fix $x$.
By the Fourier transform the elements in the operator algebras are generalised analytic polynomials.
Consequently $\Phi$ is weakly biholomorphic and therefore $\ga_c$ maps maximal analytic sets onto maximal analytic sets.
Thus $\ga_c$ induces a set map $\ga_s \colon Y \to X$ by collapsing every set of characters $\theta$ that satisfy $\theta|_{C_0(Y)} = \ev_y$ to a single point.
Therefore if $\ga_c(\theta) = \theta \Phi = \theta'$ and $\theta|_{C_0(Y)} = \ev_y$ then $\theta'|_{C_0(X)} = \ev_x$ with $\ga_s(y) = x$.
This set map is moreover a homeomorphism.
To see this, note that $\ga_s$ is the adjoint of
\[
C(X) \xrightarrow{i} \T^+_{(X,\si)} \xrightarrow{\Phi} \T^+_{(Y,\tau)}  \xrightarrow{E_0} C(Y),
\]
where $i$ denotes inclusion and $E_0$ the conditional expectation onto $C(Y)$.
This completes the proof of Claim 2. \hfill{$\blacksquare$}

\smallskip

We now turn to showing that the homeomorphism $\ga_s$ defined above induces the piecewise conjugacy.
It suffices to show that $\ga_s$ implements a piecewise conjugacy in a neighbourhood of every $y \in Y$.
The key is to work with germs.
For a fixed $y \in Y$, we write $\tau_i \sim \tau_j$ if $\tau_i$ and $\tau_j$ are maps on $Y$ for which there exists some neighbourhood $U \ni y$ such that $\tau_i |_U = \tau_j |_U$.
The equivalence class of $\tau_i$ is called \emph{the germ of $\tau_i$ at $y$}.

Compare the germ of $\tau_1$ at $y$ to the germs of the mappings $\ga_s^{-1} \si_1 \ga_s, \dots$, $\ga_s^{-1}\si_d \ga_s$ and $\tau_1, \dots, \tau_{d'}$ at $y$.
After a possible re-enumeration we find that there is a neighbourhood $V = V_1 \ni y$ and there are $k,l$ so that:
\begin{quoting}
(i) $\ga_s^{-1} \si_1 \ga_s(z) = \dots = \ga_s^{-1}\si_k\ga_s(z) = \tau_1(z) = \dots = \tau_l(z)$ for all $z \in V$; and

\smallskip

\noindent (ii) for any $\ga_s^{-1}\si_i\ga_s$ with $i>k$ (resp. $\tau_j$ with $j>l$) there is a net $y_\la \in V$ such that $y_\la \to y$ with $\ga_s^{-1} \si_i \ga_s(y_\la) \neq \tau_1(y_\la)$ for all $\la \in \La$ (resp. $\tau_j (y_\la) \neq \tau_1(y_\la)$ for all $\la \in \La$).
\end{quoting}
The main objectives then are to show that $k=l$ and that $d = d'$.
Indeed in this case we have that $\si_1, \ldots, \si_k$ and $\tau_1, \ldots, \tau_k$ are conjugate on $V_1$.
After repeating the process for $\tau_2, \dots, \tau_d$ we will have that the intersection $V_1 \cap \cdots \cap V_d$ gives the required $\U_\pi$ on which up to the permutation $\pi$ the tuple $\si_1|_{\U_\pi}, \ldots, \si_d|_{\U_\pi}$ is conjugate to $\tau_1|_{\U_\pi}, \ldots, \tau_d|_{\U_\pi}$.
The fact that $d=d'$ will ensure that every $\si_j$ is paired with some $\tau_i$ (and vice versa), and that $\pi$ is a permutation on $d$ symbols.

\smallskip

We now proceed to the proof.
For any point $y \in Y$ we can define the orbit representation $\pi_y \colon C_0(Y) \to \B(\ell^2(\bF^+_{d'}))$ by
\[
\pi_y(g) = \diag\{g \tau_{w}(y) \mid w \in \bF_{d'} \} \foral g \in C_0(Y).
\]
By setting $V(\Bt_i) \in \B(\ell^2(\bF^+_{d'}))$ so that $V(\Bt_i)(e_\mu) = e_{i \mu}$ for all $i = 1, \dots d'$ and $\mu \in \bF^+_{d'}$ we obtain the contractive homomorphism $(V \times \pi_y)$ of $\T_{(Y,\tau)}^+$.
Define the algebraic homomorphism
\[
\Phi_y\colon  \T_{(X,\si)}^+ \stackrel{\Phi}{\longrightarrow} \T_{(Y,\tau)}^+ \stackrel{V \times \pi_y}{\longrightarrow} \B(\ell^2(\bF^+_{d'})) \stackrel{E}{\longrightarrow} M_{d' + 1}(\bC),
\]
where $E$ is the compression by the projection onto the subspace of $\ell^2(\bF_{d'}^+)$ generated by the vectors $e_\mt, e_1, \dots, e_{d'}$.
Consequently if we write $\Phi(Q) = g_0 + \big(\sum_{i=1}^{d'} \Bt_i g_i\big) + Z$ for $Q \in \T_{(X,\si)}^+$ then we get
\[
\Phi_y(Q) =
\begin{bmatrix}
g_0(y) & 0 & 0 & \dots & 0 \\
g_1(y) & g_0 \tau_1(y) & 0 & \dots & 0 \\
g_2(y) & 0 & g_0 \tau_2(y) & \dots & 0 \\
\vdots & \vdots & \vdots & \ddots & 0 \\
g_{d'}(y) & 0 & 0 & \dots & g_0 \tau_{d'}(y)
\end{bmatrix}.
\]
By denoting $E_{ij}$ the rank one operator with one in the $(i,j)$-entry and zeroes elsewhere, we may alternatively write
\[
\Phi_y(Q) = g_0(y)E_{00} + \sum_{i=1}^{d'} g_0 \tau_{i}(y)E_{ii} + g_{i}(y) E_{i0}.
\]
We remark here that the range of $\Phi_y$ contains all $E_{i0}$.
Indeed the range of $\Phi_y$ equals the range of $E(V\times \pi_y)$ since $\Phi$ is onto.
By choosing a $g \in C_0(Y)$ such that $g(y) = 1$ we obtain that $E_{i0} = E(V \times \pi_y)(\Bt_i g)$ for all $i=1,\dots, d'$.

For a matrix $A \in M_{d' + 1}(\bC)$ we write $[A]_{ij}$ for the element in the $(i,j)$-entry.

\smallskip

\noindent {\bf Claim 3.} {\it Suppose that $\Psi_y \colon \T_{(X,\si)}^+ \to \ran \Phi_y$ is an algebraic homomorphism such that $\Psi_y|_{C_0(X)}$ is diagonal on the range of $\Phi_y$, in the strong sense that
\[
\Psi_y(f) = \sum_{i=0}^{d'} [\Phi_y(f)]_{ii} \cdot E_{ii} \foral f \in C_0(X),
\]
and suppose that
\[
\Psi_y(\Bs_j h) =
\begin{bmatrix}
\ast & 0 & 0 & \dots & 0 \\
c_{1j} & \ast & 0 & \dots & 0 \\
c_{2j} & 0 & \ast & \dots & 0 \\
\vdots & \vdots & \vdots & \ddots & 0 \\
c_{d'j} & 0 & 0 & \dots & \ast
\end{bmatrix}
\text{ for some } h \in C_0(X).
\]
If $c_{ij} \neq 0$ then $\tau_i (y) = \ga_s^{-1} \si_j \ga_s (y)$.}

\smallskip

\noindent {\it Proof of Claim 3.} By the covariance relation and commutativity of $C_0(X)$ we obtain
\begin{align*}
\big[\Psi_y(f) \Psi_y(\Bs_j h) \big]_{i0}
& =
\big[\Psi_y(f \cdot \Bs_j h) \big]_{i0}
 =
\big[\Psi_y(\Bs_j h) \Psi_y(f \si_j) \big]_{i0} .
\end{align*}
By using the form of the elements as in the assumption we obtain
\[
[\Phi_y(f)]_{ii} \cdot c_{ij} = c_{ij} \cdot [\Phi_y(f \si_j)]_{00}.
\]
However, the homomorphisms
\[
Q \mapsto  [V \times \pi_y(Q)]_{00} \qand Q \mapsto  [V \times \pi_y(Q)]_{ii}
\]
are characters of $\T_{(Y,\tau)}^+$ whose restrictions on $C_0(Y)$ are $\ev_{y}$ and $\ev_{\tau_i(y)}$, respectively.
Hence by the discussion on the characters we have that
\[
[\Phi_y(f)]_{00} = [V \times \pi_y (\Phi(f))]_{00} = \ev_{\ga_s(y)} (f),
\]
and
\[
[\Phi_y(f)]_{ii} = [V \times \pi_y (\Phi(f))]_{ii} = \ev_{\ga_s\tau_i(y)}(f).
\]
Consequently $f \ga_s \tau_i (y) \cdot c_{ij} = c_{ij} \cdot f \si_j \ga_s(y)$ for all $f \in C_0(X)$ which completes the proof of Claim 3. \hfill{$\blacksquare$}

\smallskip

We will construct a homomorphism $\Psi_y$ as the one appearing in the claim above.
To this end we require the following remark. Let $\fT_2$ denote the lower triangular $2 \times 2$ matrices.

\smallskip

\noindent {\bf Claim 4.} {\it Let $\Phi \colon C_0(X) \to \fT_2$ be an algebraic homomorphism.
If $[\Phi(f)]_{11} = [\Phi(f)]_{22}$ then $[\Phi(f)]_{21} = 0$.}

\smallskip

\noindent {\it Proof of Claim 4.} Let $C_0(K) = \ca(f)$ be the commutative C*-subalgebra of $C_0(X)$ generated by $f$.
Since $g \mapsto [\Phi(g)]_{ii}$ are characters, we obtain that $[\Phi(h)]_{11} = [\Phi(h)]_{22}$ for all $h \in C_0(K)$.
Moreover these characters correspond to an evaluation, say $\ev_x$ for some $x \in K$.
Since $\Phi(hg) = \Phi(h) \Phi(g)$ for all $h, g \in C_0(K)$ we get that
\begin{align*}
[\Phi(h g)]_{21}
& =
[\Phi(h)]_{21} \cdot [\Phi(g)]_{11} + [\Phi(h)]_{22}  \cdot  [\Phi(g)]_{21}\\
& =
[\Phi(h)]_{21}  \cdot  g(x) + h(x) \cdot [\Phi(g)]_{21}.
\end{align*}
Hence $g \mapsto [\Phi(g)]_{21}$ is a point derivation at $x \in K$.
Therefore $[\Phi(f)]_{21}$ is zero and the proof of Claim 4 is completed. \hfill{$\blacksquare$}

\smallskip

Now we have all the required ingredients to establish the existence of a homomorphism $\Psi_y$ as in the Claim 4.
In particular $\Psi_y$ will be similar to $\Phi_y$ by an invertible matrix $A_y$.

\smallskip

\noindent {\bf Claim 5.} {\it Let $\Phi_y$ be the representation associated with a point $y \in Y$.
Then there exists an algebraic homomorphism $\Psi_y \colon \T_{(X,\si)}^+ \to \ran \Phi_y$ such that
\[
\Psi_y(f) = \sum_{i=0}^{d'}  [\Phi_y(f)]_{ii} \cdot E_{ii} \foral f \in C_0(X),
\]
and the range of $\Psi_y$ contains all $E_{i0}$ for $i=1, \dots, d$.}

\smallskip

\noindent {\it Proof of Claim 5.} We will construct the homomorphism $\Psi_y$ step by step.

If $\ga_s(y)= \ga_s \tau_1(y)$ then $[\Phi_y(f)]_{00} = [\Phi_y(f)]_{11}$.
Therefore we obtain
\[
\Phi_y(f) = \begin{bmatrix}
g(y) & 0 & 0 & \dots & 0 \\
0 & g(y) & 0 & \dots & 0 \\
\ast & 0 & \ast & \dots & 0 \\
\vdots & \vdots & \vdots & \ddots & 0 \\
\ast & 0 & 0 & \dots & \ast
\end{bmatrix}
\]
by Claim 4 and we proceed to the second step.
On the other hand, suppose that $\ga_s(y) \neq \ga_s \tau_1(y)$ and choose an $f \in C_0(X)$ of norm less than $1$ such that $f|_{\ol{U}} =1$ and $f|_{\ga_s \tau_1(\ol{U})} = 0$ for an appropriate compact neighbourhood $\ol{U}$ of $\ga_s (y)$ with $\ol{U} \cap \ol{\ga_s \tau_1(U)} = \mt$.
We have already remarked that  $[\Phi_y(f)]_{00} = f \ga_s (y)$ and $[\Phi_y(f)]_{ii} = f \ga_s \tau_i(y)$ for all $f \in C_0(X)$ so that
\[
[\Phi_y(f)]_{00} = 1 \qand [\Phi_y(f)]_{11} = 0.
\]
If $k = [\Phi_y(f)]_{10}$ let $A$ be the invertible matrix $I - k E_{10}$ (with inverse $A^{-1} = I + k E_{10}$).
Then we derive
\[
A \Phi_y(f) A^{-1} =
\begin{bmatrix}
1 & 0 & 0 & \dots & 0 \\
0 & 0 & 0 & \dots & 0 \\
\ast & 0 & \ast & \dots & 0 \\
\vdots & \vdots & \vdots & \ddots & 0 \\
\ast & 0 & 0 & \dots & \ast
\end{bmatrix}
\]
By commutativity we get that
\[
A \Phi_y(f) A^{-1} A \Phi_y(h) A^{-1} = A \Phi_y(h) A^{-1} A \Phi_y(f) A^{-1}
\]
for all $h \in C_0(X)$.
This shows that the $(1,0)$-entry of $A \Phi_y(h) A^{-1}$ is zero.
Furthermore the diagonals of $A \Phi_y(h) A^{-1}$ and $\Phi_y(h)$ are the same.
In addition $A E_{i0} A^{-1} = E_{i0}$ for all $i = 1, \dots ,d'$.

We can continue inductively for the rest of the rows.
In this way we form an algebraic homomorphism $\Psi_y \colon \T_{(X,\si)}^+ \to \ran \Phi_y$ that satisfies the conclusion of Claim 5. \hfill{$\blacksquare$}

\smallskip

The matrices constructed in Claim 5 do not depend on the choice of $f$.
Indeed if $g \in C_0(X)$ is such that
\[
[\Phi_y(g)]_{00} = 1 \qand [\Phi_y(g)]_{11} = 0,
\]
then by using commutativity we obtain that the equality $\Phi_y(f) \Phi_y(g) = \Phi_y(g) \Phi_y(f)$ gives that
\[
[\Phi_y(g)]_{10} = [\Phi_y(f)]_{10}.
\]
We will write $A_y$ for the invertible matrix inside $M_{d' + 1}(\bC)$ that is constructed in Claim 5 and gives the homomorphism $\Psi_y(\cdot) = A_y \Phi_y(\cdot) A_y^{-1}$.

\smallskip

\noindent {\bf Claim 6.} {\it The systems have the same multiplicity, i.e. $d = d'$.}

\smallskip

\noindent {\it Proof of Claim 6.} Let $\Psi_y$ be an algebraic homomorphism that satisfies the conclusion of Claim 5.
Hence the range contains all $E_{i0}$.
Select an $h \in C_0(X)$ such that
\[
1 = h \ga_s (y) = h \ga_s \tau_i(y) \foral i=1, \dots, d'.
\]
By Claim 4 then we have that $\Psi_y(h) = I_{d'}$.
Write $c_{ij} = [\Psi_y(\Bs_j h)]_{i0}$ and observe that
\[
\Psi_y(\Bs_j h f) = \Psi_y(\Bs_j f) \Psi_y(h) = \Psi_y(\Bs_j f) \foral f \in C_0(X).
\]
As in \cite[Proposition 4.3]{KakKat12} we can show that there is a matrix $[c_{ij}']$ such that $[c_{ij}] [c_{ij}'] = I_{d'}$.
This implies that $d \leq d'$.
This inequality holds for any $y \in Y$.
By symmetry on $\Phi^{-1}$ and for an $x \in X$ we get that $d' \leq d$, and the proof of Claim 6 is complete. \hfill{$\blacksquare$}

\smallskip

As a consequence, Claim 6 further implies that the obtained matrix $[c_{ij}]$ is square.
Therefore the matrix $[c_{ij}]$ which is shown to be right invertible in Claim 6, is in fact invertible.

To end the proof fix a $y \in Y$.
Let the setup be as after the proof of Claim 2.
Fix the map $\tau_1$ and suppose that after a re-enumeration the germ $\tau_1$ contains exactly the functions
\[
\ga_s^{-1} \si_1 \ga_s, \dots, \ga_s^{-1}\si_k\ga_s  \text{ and } \tau_1, \dots, \tau_l \text{ for some } k,l.
\]
We have just seen that $d = d'$.
We aim to show that $k=l$.
The next claim gives the required equality $k=l$, modulo a condition on the matrix $[c_{ij}]$ of Claim 6.

\smallskip

\noindent {\bf Claim 7.} {\it Let $[c_{ij}]$ be the invertible matrix obtained in Claim 6 for the fixed $y \in Y$.
If $c_{ij} = 0$ for the $i,j$ that satisfy $\tau_i \sim \tau_1 \not\sim \ga_s^{-1} \si_j \ga_s$, then $k = l$.}

\smallskip

\noindent {\it Proof of Claim 7.} We may assume without loss of generality that $k \leq l$, and it remains to prove that $k<l$ is impossible.
This follows verbatim from the last paragraph of \cite[Proof of Theorem 4.9]{KakKat12}.
In short, we may write
\[
[c_{ij}] = \begin{bmatrix} [a_{ij}] & 0 \\ \ast & \ast \end{bmatrix},
\]
where $[a_{ij}]$ is a $(l \times k)$-matrix.
If $k < l$ then when applying the Gaussian elimination we will be able to eliminate the entire $k+1$ row of the invertible matrix $[c_{ij}]$, which gives the contradiction.
\hfill{$\blacksquare$}

\smallskip

Therefore in the rest of the proof we will focus in showing that $c_{ij} = 0$ for the $i,j$ that satisfy $\tau_i \sim \tau_1 \not\sim \ga_s^{-1} \si_j \ga_s$.

The coefficients $c_{ij}$ are obtained by the $(i,0)$-th entries of $\Psi_y(\Bs_j h)$ for the function $h$ of Claim 6.
For their analysis we can restrict even further to the $(2 \times 2)$-matrix representations.
For $z \in Y$ we define
\[
\psi_{z , i} := P_{i0} \Psi_z \colon \T_{(X,\si)} \to \fT_2,
\]
where $P_{i0}$ is the projection on the space generated by $\{e_\mt, e_i\}$.
Since we are in the lower triangular matrices we have that $P_{i0} \Psi_z P_{i0} = P_{i0} \Psi_z$, hence $\psi_{z , i}$ is indeed a homomorphism.
Then  we get that
\[
c_{ij} = [\psi_{y, i}( \Bs_j h)]_{10} .
\]
We will also write $\phi_{z,i} = P_{i0} \Phi_z$ and $A_{z,i} = P_{i0} A_z$.
We can check that
\[
\psi_{z,i}(\cdot)  = A_{z,i} \phi_{z,i}(\cdot) A_{z,i}^{-1}.
\]

\smallskip

\noindent {\bf Case 1.} Suppose that $y \neq \tau_1(y)$ and choose $f$ with $[\phi_{y,1}(f)]_{11} = 0$ as we do in Claim 5.
We may also choose a neighbourhood $U$ of $y$ where $[\phi_{z,1}(f)]_{00}=1$ and $[\phi_{z,1}(f)]_{11} = 0$ for all $z \in U$.
Then we get the invertible matrices
\[
A_{z,1} = I - [\phi_{z,1}(f)]_{10} \cdot E_{10}
\]
for all $z \in U$.
Consequently, invoking the automatic continuity of $\Phi$, we get that $A_{\bullet,1}$ is continuous on $y$.
Let $j$ such that $\ga_s^{-1} \si_j \ga_s \not\sim \tau_1$; then there exists a net $y_\la \in U$ with $y_\la \rightarrow y$ such that $\ga_s \tau_1(y_\la) \neq \si_j \ga_s (y_\la)$ for all $\la$.
Consider the elements
\[
c_{1j, \la} = [\psi_{y_\la}(\Bs_j h)]_{10} = [A_{y_\la,1} \phi_{y_\la,i}(\Bs_j h) A_{y_\la,1}^{-1}]_{10}.
\]
By continuity we have that $\lim_\la A_{y_\la,1} = A_{y,1}$ hence $c_{1j} = \lim_\la c_{1j, \la}$.
However $c_{1j, \la} = 0$ as follows by the construction of $y_\la$ and Claim 3.
Indeed if $c_{1j, \la} \neq 0$ then Claim 3 indicates that $\ga_s \tau_1(y_\la) = \si_j \ga_s(y_\la)$.
We conclude that $c_{1j} = 0$.

\smallskip

\noindent {\bf Case 2.} Suppose that $y = \tau_1(y)$.
Then $\phi_{y,1}|_{C_0(X)}$ is scalar by Claim 4.
Let $j$ such that $\ga_s^{-1} \si_j \ga_s \not\sim \tau_1$.
Let a net $y_\la \in V$ such that $\ga_s \si_j \ga_s^{-1}(y_\la) \neq \tau_1(y_\la)$ for all $\la$.
Construct as above the invertible matrices $A_{y_\la}$ and note that
\[
A_{y_\la,1} = \begin{bmatrix} 1 & 0 \\ c_\la & 1 \end{bmatrix}.
\]
Now all $A_{y_\la}$ are uniformly bounded by $1 + \nor{\Phi}$ and by passing to a subsequence we may assume that they converge to an operator.
Consequently the $A_{y_\la,1}$ converge to a matrix
\[
A = \begin{bmatrix} 1 & 0 \\ c & 1 \end{bmatrix},
\]
which is invertible.
Then we obtain
\begin{align*}
\phi_{y,1}(\Bs_j h)
& =
A^{-1} \big( \lim_\la A_{y_\la,1} \phi_{y_\la,1}(\Bs_j h) A_{y_\la,1}^{-1} \big) A
 =
A^{-1} \lim_\la \psi_{\la,1}(\Bs_j h) A.
\end{align*}
But by Claim 3 again we get that
\[
c_{1j, \la} = [A_{y_\la,1} \phi_{y_\la,1}(\Bs_j h) A_{y_\la,1}^{-1}]_{10} = [\psi_{\la,1}(\Bs_j h)]_{10}  =0,
\]
since $\ga_s^{-1} \si_j \ga_s \not\sim \tau_1$.
Hence $\phi_{y,1}(\Bs_j h) = A^{-1} D A$ where $D$ is diagonal.
Now by assumption $y = \tau_1(y)$ hence the range of $\phi_{y,1}$ consists of the matrices $a_1I + a_2 E_{10}$, all of which commute with $A$.
Consequently we obtain
\[
D = A \phi_{y,1}(\Bs_j h) A^{-1} = \phi_{y,1}(\Bs_j h),
\]
which shows that $c_{1j} = [\phi_{y,1}(\Bs_j h)]_{10} = [D]_{10} = 0$.

We showed that in either case $c_{1j} = 0$ for all $j$ such that $\tau_1 \not\sim \ga_s^{-1} \si_j \ga_s$.
By substituting $\tau_1$ with $\tau_i$ for $i=1,\dots,k$ we have that $c_{ij} = 0$ for all $i,j$ that satisfy $\tau_i \sim \tau_1 \not\sim \ga_s^{-1} \si_j \ga_s$.
As we explained above, this concludes our proof of \cite[Theorem 3.22]{DavKat11}.

\subsection{Applications}\label{S: app pc}

The above idea applies to other classes of operator algebras.
The key is to analyse the maximal analytic sets of the character space, and also ensure the existence of a Fock representation.
Let us describe briefly how this is achieved in three cases.
We reserve a full discussion for a forthcoming project.

The first case is the second part of \cite[Theorem 3.22]{DavKat11}.
The \emph{semicrossed product $C_0(X) \times_\si \bF_+^d$} related to $(X,\si)$ is defined as the universal nonselfadjoint operator algebra generated by the $\Bs_\mu f$ so that now the $\Bs_i$ are taken to be just contractions, and $f \Bs_i = \Bs_i f \si_i$.
Once again we do not require the generators to be separated into $\Bs_i$ and $f \in C_0(X)$ as in \cite[Definition 1.2]{DavKat11}.
When the semicrossed products are algebraically isomorphic then the maximal analytic sets are polydiscs, and again Claim 1 holds.
Furthermore by definition there is a canonical epimorphism from the semicrossed product onto the tensor algebra.
This sets the appropriate context to pass to an epimorphism of $C_0(X) \times_\si \bF_+^d$ onto $\T_{(Y,\tau)}$ and then follow the steps from Claim 3 and thereon to obtain again piecewise conjugacy.
Indeed, the only ingredient required is that the homomorphisms are onto.

On the other hand, suppose that the $\si_i \colon X \to X$ commute.
Then we may define the universal nonselfadjoint operator algebra $C_0(X) \times_\si \bZ_+^d$ generated by the $\Bs_{\un{x}} f$ for $\un{x} \in \bZ_+^d$ and $f \in C_0(X)$ so that now the $\Bs_i$ are taken to be commuting contractions, and $f \Bs_i = \Bs_i f \si_i$.
It follows that the character space of $C_0(X) \times_\si \bZ_+^d$ coincides with the character space of $C_0(X) \times_\si \bF_+^d$.
Furthermore the family $(\pi_y,\{V_i\}_{i=1}^d)$ defined by
\[
\pi_y(f) = \diag\{f \si_{\un{x}}(y) \mid \un{x} \in \bZ_+^d\} \qfor f \in C_0(X),
\]
and $V_i e_{\un{x}} = e_{i + \un{x}}$ defines a completely contractive representation of $C_0(X) \times_\si \bZ_+^d$.
These are the appropriate ingredients to show that if $C_0(X) \times_\si \bZ_+^d$ is algebraically isomorphic to $C_0(Y) \times_\tau \bZ_+^{d'}$ then the systems are piecewise conjugate.
Indeed the proof reads the same as in the case of the tensor algebras (even the part on automatic continuity).

Even more one may consider the universal nonselfadjoint operator algebra $C_0(X) \times_\si^{\textup{rc}} \bZ_+^d$ generated by the $\Bs_{\un{x}} f$ for $\un{x} \in \bZ_+^d$ and $f \in C_0(X)$ so that now the $\Bs_i$ form a commuting row contraction, and $f \Bs_i = \Bs_i f \si_i$.
Now the character space coincides with the character space of $\T_{(X,\si)}$.
Furthermore we obtain a contractive representation by compressing the family $(\pi_y, \{V_i\}_{i=1}^d)$ of $C_0(X) \times_\si \bZ_+^d$ above by the projection $p$ that gives the symmetric subproduct system.
Indeed commutativity of the $\si_i$ implies that $p$ commutes with $\pi_y$.
Once more one derives piecewise conjugacy when $C_0(X) \times_\si^{\textup{rc}} \bZ_+^d$ is algebraically isomorphic to $C_0(Y) \times_\tau^{\textup{rc}} \bZ_+^{d'}$.
Automatic continuity can be shown to hold as well.


\end{document}